\documentclass[parskip=half,bibliography=totoc]{scrartcl}

\usepackage{comment}
\usepackage[automark]{scrlayer-scrpage}								
\usepackage{amsfonts}												
\usepackage{amsmath}												
\usepackage{mathtools}												
\usepackage{stmaryrd}												
\usepackage{amssymb}												
\usepackage{mathrsfs}												
\usepackage{accents}												
\usepackage[T1]{fontenc}											
\usepackage[utf8]{inputenc}											
\usepackage[top=1in,bottom=1in,left=1.25in,right=1.25in]{geometry}	
\usepackage{enumerate} 												
\usepackage{authblk}												
\usepackage{abstract}												
\usepackage{etoolbox}												
\usepackage[normalem]{ulem}

\usepackage{tikz-cd}												

\usepackage{rotating}
\usepackage{pdflscape}

\usepackage{amsthm}													

\usepackage[colorlinks=true,citecolor=blue,allcolors=blue]{hyperref}				
\usepackage[noabbrev]{cleveref}										

\newcommand{\R}{\mathbb{R}}

\DeclareMathOperator{\tr}{tr}

\DeclareMathOperator{\End}{End}

\newtheoremstyle{mythm}
{}
{}
{\slshape}
{}
{\bfseries\sffamily}
{.}
{ }
{}
\newtheoremstyle{mydef}
{}
{}
{}
{}
{\bfseries\sffamily}
{.}
{ }
{}

\theoremstyle{mythm}
\newtheorem{thm}{Theorem}[section]

\newtheorem{prop}[thm]{Proposition}
\newtheorem{cor}[thm]{Corollary}
\newtheorem{lem}[thm]{Lemma}
\theoremstyle{mydef}
\newtheorem{mydef}[thm]{Definition}
\newtheorem{rem}[thm]{Remark}
\newtheorem{ex}[thm]{Example}

\allowdisplaybreaks

\clearpairofpagestyles
\ihead[]{\headmark}
\ohead[]{\pagemark}
\deffootnote[1em]{0em}{1em}{%
	\textsuperscript{\thefootnotemark}%
}
\setfootnoterule{3em}

\apptocmd{\sloppy}{\hbadness 10000\relax}{}{}

\title{Classification of generalized Einstein metrics on 3-dimensional Lie groups}
\author{V.\ Cort\'es}
\author{D.\ Krusche}

\affil{\normalsize Fachbereich Mathematik\hfill \protect\\  Universit\"at Hamburg\hfill \protect\\  
Bundesstra\ss e 55, D-20146 Hamburg, Germany\hfill \protect\\  
\texttt{vicente.cortes@uni-hamburg.de, david.krusche@uni-hamburg.de}  \vspace{1em} }

\date{\today}

\begin{document}
\maketitle

\begin{abstract}
We develop the theory of left-invariant generalized pseudo-Riemannian metrics on Lie groups. 
Such a metric accompanied by a choice of left-invariant divergence operator 
gives rise to a Ricci curvature tensor and we study the corresponding Einstein equation. 
We compute the Ricci tensor in terms of the tensors (on the sum of the Lie algebra and its dual) 
encoding the Courant algebroid structure, the generalized metric and the divergence operator. 
The resulting expression is  polynomial and homogeneous of degree two in the coefficients 
of the Dorfman bracket and the divergence operator with respect to a left-invariant orthonormal 
basis for the generalized metric. We determine all generalized Einstein metrics on three-dimensional Lie groups. 

\end{abstract}

\clearpage
\setcounter{tocdepth}{2}
\tableofcontents

\section{Introduction}
Generalized geometry was proposed by Hitchin \cite{H} 
as a framework unifying complex and symplectic structures. 
The two latter can be viewed as particular instances of the notion of a generalized 
complex structure, the theory of which was developed in \cite{Gualtieri,Gualtieri2} 
including a geometrization of Barannikov's and Kontsevich's extended deformation theory. 

Similarly, pseudo-Riemannian metrics have a fruitful counterpart in generalized geometry,  
which can be used, for instance, to unify and geometrize the structures involved in type II
supergravity \cite{CSW}.  A generalized pseudo-Riemannian metric together with a divergence operator
is indeed sufficient to define a notion of generalized Ricci curvature and thus to 
pose a generalized Einstein equation as the vanishing of the generalized Ricci curvature \cite{GSt}. 
In the context of supergravity and string theory the divergence operator is related to the dilaton field, which is 
itself subject to a field equation. 

A generalized geometry formulation of minimal six-dimensional supergravity has been given in \cite{GS} 
with a particular case of the generalized Einstein equation as the main bosonic equation of motion.  
It would be interesting to classify left-invariant solutions on six-dimensional Lie groups
using the theory developed in our present work. 
We note that by taking, for instance, the product of a pair of three-dimensional generalized Einstein Lie groups (as defined below in the introduction  and classified in our paper) we obtain a six-dimensional generalized Einstein Lie group. 
If one imposes, in addition, a self-duality condition on the three-form, one arrives at (decomposable) solutions of the equation of motion
mentioned above. Other (indecomposable) solutions on products of three-dimensional Lie groups have been constructed in  \cite{MS}. Examples of invariant Ricci-flat Bismut connections 
on compact homogeneous Riemannian manifolds 
have been constructed in \cite{GSt,PR1,PR2}. They include non Bismut-flat examples \cite{PR1,PR2} and give rise to invariant positive definite solutions of the generalized Einstein equation with Riemannian divergence operator.

In this paper we focus on left-invariant generalized pseudo-Riemannian metrics on Lie groups $G$. 
We develop the theory on arbitrary Lie groups in Section~\ref{2ndSec} and, based on that 
theory, provide a complete classification of left-invariant solutions of the generalized Einstein equation on three-dimensional Lie groups in 
Section~\ref{class:sec}. 

First we show in Proposition~\ref{NF_metric:prop} that, up to an  isomorphism, the generalized metric $\mathcal{G}$ and the Courant algebroid structure are encoded in a pair $(g,H)$ consisting of a left-invariant pseudo-Riemannian metric $g$ and a left-invariant closed three-form $H$ on $G$. Then we describe the space of left-invariant torsion-free and metric 
generalized connections $D$ on $(G,\mathcal{G}_g,H)$ as a finite-dimensional affine space modeled 
on the generalized first prolongation of $\mathfrak{so}(\mathfrak{g}\oplus \mathfrak{g}^*)$ in Proposition~\ref{LC:prop},
where $\mathcal{G}_g$ denotes the generalized metric determined by $g$. Such generalized connections $D$ are called 
left-invariant Levi-Civita generalized connections. As part of the proof, we construct a canonical left-invariant Levi-Civita generalized connection $D^0$, which can serve as an origin in the above affine space. 

A left-invariant divergence operator on $\Gamma (\mathbb{T} G)$, where $\mathbb{T}M$ denotes the 
generalized tangent bundle of a manifold $M$, can be identified with an element $\delta\in E^*$, where $E=\mathfrak{g}\oplus \mathfrak{g}^*$. We say that a left-invariant generalized connection $D$ has divergence operator $\delta$ if 
$\delta_D = \delta$, where $\delta_D(v) \coloneqq \mathrm{tr} (D v)$, $v\in E$. Here $D$ is identified with an element of
$E^*\otimes \mathfrak{so}(E)$, $E\ni u\mapsto D_u\in \mathfrak{so}(E)$.  For instance, we have $\delta_{D^0} =0$ for the 
canonical left-invariant Levi-Civita generalized connection $D^0$, compare Proposition~\ref{LCdiv0:prop}. In Proposition~\ref{S:prop} we specify for every $\delta \in E^*$ a left-invariant Levi-Civita generalized connection
$D$ such that $\delta_D=\delta$. We end Section~\ref{divergence:Section} by observing that $\delta=0$ is not the only canonical choice of left-invariant divergence operator on a Lie group.  
A more general choice is to take $\delta$ as a fixed multiple of the trace-form $ \tau $ of $ \mathfrak{g} $. The choice $\delta^\mathcal{G} =-\tau \circ \pi \in E^*$, where $\pi : E \rightarrow \mathfrak{g}$ is the canonical projection, corresponds precisely to the 
divergence operator associated with the generalized connection trivially extending the Levi-Civita connection of 
any left-invariant pseudo-Riemannian metric, as shown in Proposition~\ref{Riemdiv:prop}. The latter choice does therefore
coincide with what is called the Riemannian divergence operator \cite{GSt}. 

In Section~\ref{mainpart2ndSec} we define the Ricci curvature of any pseudo-Riemannian generalized Lie group $(G,\mathcal{G}_g,H,\delta )$ with prescribed divergence operator $\delta \in E^*$ as a certain element in $E^*\otimes E^*$  
(see Definition~\ref{Ricdef}). Then we express it in terms of the algebraic data on the Lie algebra $\mathfrak{g}$.  
The starting point is the computation of the tensorial part of the curvature of the canonical Levi-Civita generalized  connection $D^0$ in Proposition~\ref{curv:prop} as a homogeneous quadratic polynomial expression in the Dorfman bracket $\mathcal{B} = [ \cdot ,\cdot ]_H$.  The Ricci curvature of any pseudo-Riemannian generalized Lie group $(G,\mathcal{G}_g,H,\delta=0)$ with 
zero divergence operator is then obtained as  a Corollary~\ref{RicCor_delta0}. These results are then generalized to arbitrary $\delta$ by considering $D=D^0 +S$, where $S$ is an arbitrary element of the first generalized prolongation of $\mathfrak{so}(E)$, leading to  Lemma~\ref{improved:lem}, Proposition~\ref{genRic:prop} and Theorem~\ref{Ric:thm}. 

For illustration we give here the explicit expression for the Ricci curvature 
\[ Ric_\delta\in E_-^*\otimes E_+^* \oplus E_+^*\otimes E_-^*\] 
of a pseudo-Riemannian generalized Lie group $(G,\mathcal{G}_g,H,\delta )$, where $E_\pm$ stands for the ei\-gen\-spa\-ces of the generalized metric. For $u_\pm \in E_\pm$ and using the projections $\mathrm{pr}_{E_\pm}:E\rightarrow E_\pm$ we consider the linear maps 
\[ \Gamma_{u_\pm} \coloneqq \mathrm{pr}_{E_\pm}\circ \mathcal{B}(u_\pm ,\cdot )|_{E_{\mp}} : E_{\mp}\rightarrow E_{\pm}.\]
\begin{thm}
	Let $(G$$,\mathcal{G}_g$$,H$$,\delta )$ be any pseudo-Riemannian generalized Lie group. Then its Ricci curvature is given by
\begin{eqnarray*} Ric_{\delta} (u_-,u_+) &=&-\tr \left( \Gamma_{u_-}\circ \Gamma_{u_+}\right) + \delta (\mathrm{pr}_{E_+}\mathcal{B}(u_-,u_+)),\\
Ric_{\delta} (u_+,u_-) &=&-\tr \left( \Gamma_{u_-}\circ \Gamma_{u_+}\right) + \delta (\mathrm{pr}_{E_-}\mathcal{B}(u_+,u_-)).
\end{eqnarray*}
\end{thm}
This implies that the tensor $Ric_\delta$ is polynomial of degree two and homogeneous 
in the pair $(\mathcal{B},\delta)$. Note that it depends on the generalized metric and thus on $g$ through the 
projections $\mathrm{pr}_{E_\pm}$. An equivalent convenient component expression in an adapted basis is given in 
Theorem~\ref{Ric:thm}, where also symmetry properties of $Ric_\delta$ are discussed. 

To derive an explicit expression for $Ric_\delta$ in terms of the data $(\mathfrak{g},g,H)$ rather than $(\mathfrak{g},g,\mathcal{B})$ it suffices to express the Dorfman bracket $\mathcal{B}$ in terms of the Lie bracket and the 
three-form $H$, see Proposition~\ref{B:prop}. In Proposition~\ref{rescalProp} we show that the underlying metric $g$ of an Einstein generalized pseudo-Riemannian Lie group (i.e.\ a left-invariant solution of $Ric_\delta=0$) 
can be freely rescaled without changing the Einstein property, provided that the three-form and the divergence
are appropriately rescaled. In Proposition~\ref{RicRicProp} we relate the Ricci curvature $Ric_\delta$
of the pseudo-Riemannian generalized Lie group to the Ricci curvature of the left-invariant pseudo-Riemannian
metric $g$. We show that $(G,\mathcal{G}_g,H=0,\delta=0 )$ is generalized Einstein if and only if 
$g$ satisfies a certain gradient Ricci soliton equation (\ref{soliton:Eq}) involving the trace-form $\tau$ of $\mathfrak{g}$.  
In particular, in the special case when $\mathfrak{g}$ is unimodular, the generalized Einstein equation reduces to the Einstein (vacuum) equation for $g$.

Next we describe how, building on  the general results of 
Section~\ref{2ndSec}, in Section~\ref{class:sec} we determine all left-invariant solutions $(H,\mathcal{G},\delta)$ to the Einstein equation  on three-dimensional Lie groups $G$, up to isomorphism. Here $H$ stands for the three-form which, together with the Lie bracket, determines the exact Courant algebroid structure, $\mathcal{G}$ stands for the generalized pseudo-Riemannian metric and $\delta$ for the divergence required to define the Ricci curvature uniquely. 
The data $(G,H,\mathcal{G},\delta)$ can be simply referred to as a  generalized Einstein Lie group (three-dimensional in our case). 

Up to  isomorphism, we can assume from the start that  $\mathcal{G}=\mathcal{G}_g$ is associated with a left-invariant pseudo-Riemannian metric $g$ on $G$,  compare Proposition~\ref{NF_metric:prop}. In the remaining part of the introduction we will therefore simply speak of solutions $(H,g,\delta )$ on $\mathfrak{g}$, or more 
precisely as generalized Einstein structures on $\mathfrak{g}$. In particular, we identify the left-invariant structures $(H,g,\delta)$ with tensors 
\[ H\in \bigwedge^3 \mathfrak{g}^*, 
\quad g\in \mathrm{Sym}^2\mathfrak{g}^*\quad\mbox{and}\quad \delta \in E^*= (\mathfrak{g}\oplus \mathfrak{g}^*)^*.\]  

As a preliminary, we explain in Section~\ref{prelim:sec} how, using the metric $g$, the Lie bracket of $\mathfrak{g}$ can be encoded in an endomorphism $L\in \mathrm{End}\, \mathfrak{g}$. Irrespective of the signature of $g$, the endomorphism  $L$ happens to be $g$-symmetric if and only if the Lie algebra is unimodular. This allows for the choice of an 
orthonormal basis of $(\mathfrak{g},g)$ in which $L$ takes one of five parameter-dependent normal forms, provided that 
$\mathfrak{g}$ is unimodular, see Proposition~\ref{NF:prop}. Moreover, the Jacobi identity does not impose any constraint on the normal form.

After these preliminaries, we give in Section~\ref{zerodiv:sec} the classification of solutions with zero divergence, that is solutions of the type $(H,g,\delta=0)$, beginning with the class of unimodular Lie algebras. The final results can be roughly summarized as follows, see Theorem~\ref{Einstein:thm}, Theorem~\ref{nonunimod_divzero:thm} and Remark~\ref{r3,1':rem}. 
\begin{thm} Any \textbf{divergence-free} generalized Einstein structure on a three-di\-men\-sion\-al \textbf{unimodular} Lie algebra is 
isomorphic to one in the following classes (described explicitly in Theorem~\ref{Einstein:thm}).
\begin{enumerate}
\item $\mathfrak{g}$ is abelian and $H=0$. The metric $g$ is flat of any signature. 
\item $\mathfrak{g}$ is simple, $H\neq 0$ and the metric $g$ is of non-zero constant
curvature. It is definite if and only if $\mathfrak{g}=\mathfrak{so}(3)$ and 
indefinite if and only if $\mathfrak{g}=\mathfrak{so}(2,1)$. 
\item $H=0$, $g$ is flat and $\mathfrak{g}$ is one of the following metabelian Lie algebras:  $\mathfrak{g}=\mathfrak{e}(2)$ or $\mathfrak{g}=\mathfrak{e}(1,1)$, where $\mathfrak{e}(p,q)$ denotes the Lie algebra
of the isometry group of $\mathbb{R}^{p,q}$ (the affine pseudo-orthogonal Lie algebra). The metric is definite on $ [\mathfrak{g},\mathfrak{g}] $ if and only if $\mathfrak{g}=\mathfrak{e}(2)$.  
\item $\mathfrak{g}=\mathfrak{heis}$ is the Heisenberg algebra, $H=0$ and $g$ is flat and indefinite. 
\end{enumerate} 

We note that the above list of Lie algebras, 
\[ \mathbb{R}^3, \mathfrak{so}(3), \mathfrak{so}(2,1), \mathfrak{e}(2),  \mathfrak{e}(1,1), \mathfrak{heis},\] 
is precisely the list of all unimodular three-dimensional Lie algebras.

\end{thm}
\begin{thm} Any \textbf{divergence-free} generalized Einstein structure on a three-di\-men\-sion\-al \textbf{non-unimodular} Lie algebra 
is of the type $(H=0,g)$, where $g$ is indefinite, non-degenerate on the unimodular kernel $\mathfrak u = \ker \tau$, $\tau = \mathrm{tr} \circ \mathrm{ad}$, and belongs to a certain one-parameter family of  metrics 
on the metabelian Lie algebra 
\[ \mathbb{R}\ltimes_A \mathbb{R}^2,\quad A=\left( \begin{array}{cc}1&1\\-1&1\end{array}\right).
\]  
The family of metrics (described in Theorem~\ref{nonunimod_divzero:thm}) consists of Ricci solitons which are not of constant curvature.
\end{thm}

The classification in the case of non-zero divergence is the content of Section~\ref{arbdiv:sec}.
The unimodular case is covered in Section~\ref{unimoddivnotzero:sec}, the non-unimodular case in Section~\ref{unimoddivnotzero_notunimod:sec}. To keep the introduction succinct we do only 
summarize the isomorphism types of the Lie algebras resulting from our classification without listing the detailed solutions, 
which can be found in Theorem~\ref{divnotzerounimod:thm}, Proposition~\ref{divnotzero_u_nondeg:prop}
and Proposition~\ref{divnotzero_u_deg:prop}.
\begin{thm} Any three-dimensional \textbf{unimodular} Lie algebra $\mathfrak g$ admits a generalized Einstein structure
with \textbf{non-zero divergence} as well as a  divergence-free solution. (See Theorem~\ref{divnotzerounimod:thm}).
\end{thm} 

\begin{thm}Let  $(H,g,\delta)$ be a generalized Einstein structure  with \textbf{non-zero divergence} on a three-dimensional
\textbf{non-unimodular} Lie algebra $\mathfrak g$. Then either 
\begin{enumerate}
\item
the unimodular kernel  of $\mathfrak g$ is non-degenerate (with respect to $g$) and $\mathfrak{g} = \mathbb{R}\ltimes_A \mathbb{R}^2$ 
where 
\[  A=\left( \begin{array}{cc}1&0\\ 0&\lambda\end{array}\right),\mbox{ }\lambda \in (-1,1],\mbox{ and }\, H \neq 0,\] 
see Proposition~\ref{divnotzero_u_nondeg:prop} for a complete description of $(H,g,\delta)$, or
\item its unimodular kernel is degenerate, $H=0$ and $\mathfrak{g} = \mathbb{R}\ltimes_A \mathbb{R}^2$, 
where $A\in \mathfrak{gl}(2,\mathbb{R})$ is arbitrary with only real eigenvalues and 
such that $\mathrm{tr}\, A \neq 0$. (See Proposition~\ref{divnotzero_u_deg:prop}.)
\end{enumerate}
\end{thm} 
In Proposition~\ref{Riem_div_sol:prop} we indicate for which of the left-invariant generalized Einstein 
structures the divergence $\delta$ coincides with the Riemannian divergence. We find that this is not only the case for all 
divergence-free solutions on unimodular Lie algebras but also for some of the non-unimodular cases with non-zero divergence. In the latter case, the unimodular kernel can be both degenerate or non-degenerate with respect to
the metric $g$.

For better overview the results of our classification are summarized in the tables of Section~\ref{tables:sec}.

{\bfseries Acknowledgements}

This work was supported by the German Science Foundation (DFG) under Germany's Excellence Strategy  --  EXC 2121 ``Quantum Universe'' -- 390833306. We thank Liana David, Mario Garc\'ia Fern\'andez and Carlos Shahbazi for helpful comments. 
\section{Generalized Einstein metrics on Lie groups}
\label{2ndSec}
In this section we develop a general approach for the study of left-invariant generalized Einstein metrics on Lie groups. 
\subsection{Twisted generalized tangent bundle of a Lie group}
\label{prelim:subsec}
Recall that the \emph{generalized tangent bundle} of a smooth manifold $M$ is 
the sum 
\[ \mathbb{T}M \coloneqq TM \oplus T^*M\]
of its tangent and its cotangent bundle and that 
any closed three-form $H$ on $M$ defines on $\mathbb{T}M$ the structure 
of a Courant algebroid, see e.g. \cite[Example 2.5]{G}. We will write $\mathbb{T}_pM$ 
for the fiber at $p\in M$. 

Here we consider only the special case when $M=G$ is a 
Lie group and the Courant algebroid structure is left-invariant. 

Let $G$ be a Lie group with Lie algebra $\mathfrak{g}$ and $H$ a closed left-invariant 
three-form on $G$. The \emph{$H$-twisted generalized tangent bundle} of $G$ is the vector bundle $\mathbb{T}G\rightarrow G$
endowed with the Courant algebroid structure $(\pi , \langle \cdot ,\cdot \rangle , [\cdot ,\cdot]_H)$ given by 
\begin{enumerate}
\item the canonical projection $\pi : \mathbb{T}G \rightarrow TG$,  called the \emph{anchor},  
\item the canonical symmetric bilinear pairing $\langle \cdot ,\cdot \rangle \in \Gamma (\mathrm{Sym}^2 (\mathbb{T}G)^*)$,
given by 
\[ \langle X+\xi , Y+\eta \rangle = \frac12 (\xi (Y) + \eta (Y)),\]
called the \emph{scalar product}, and
\item  the \emph{($H$-twisted) Dorfman bracket} $[\cdot ,\cdot ]_H : \Gamma ( \mathbb{T}G) \times \Gamma ( \mathbb{T}G)
\rightarrow \Gamma ( \mathbb{T}G)$, given by 
\begin{equation} \label{dorfm:eq} [ X+\xi , Y + \eta ]_H = \mathcal{L}_X(Y+\eta )  -\iota_Yd\xi + H(X,Y,\cdot ),\end{equation}
where $X, Y \in \Gamma (TG)$, $\xi, \eta \in \Gamma (T^*G)$, $\mathcal{L}$ denotes the Lie derivative and $\iota$ the interior product.
\end{enumerate} 
The above data satisfy the defining axioms of a \emph{Courant algebroid:} 
\begin{itemize}
\item[(C1)] $[u,[v,w]_H]_H = [[u,v]_H,w]_H+[v,[u,w]_H]_H$,
\item[(C2)] $\pi (u) \langle v,w\rangle = \langle [u,v]_H,w\rangle + \langle v,[u,w]_H\rangle$ and 
\item[(C3)] $\pi (u) \langle v,w\rangle = \langle u,[v,w]_H + [w,v]_H\rangle$. 
\end{itemize} 
for all $u, v, w\in \Gamma ( \mathbb{T}G)$. It is well-known that the above axioms imply 
the following useful relations (compare \cite[Definition 1]{CD} and references therein), which are obvious from (\ref{dorfm:eq}).
\begin{itemize}
\item The homomorphism of bundles $\pi$ is a bracket-homomorphism, that is 
\[ \pi [u,v]_H = [\pi u,\pi v],\] 
where 
$[\pi u,\pi v]= \mathcal{L}_{\pi u}(\pi v)$ denotes the Lie bracket of $\pi u, \pi v \in \Gamma (TG)$, and  
 \item the map $[u,\cdot ]_H : \Gamma ( \mathbb{T}G) \rightarrow \Gamma ( \mathbb{T}G)$ satisfies the Leibniz rule:
 \[ [u,fv ]_H = (\pi u) (f)v + f[u,v]_H,\quad \forall f\in C^\infty (M).\]
 \end{itemize}
For notational simplicity we define 
\begin{equation} \label{uf:eq}u(f) \coloneqq (\pi u)(f).\end{equation}

We will identify left-invariant sections of $\mathbb{T}G$ (by evaluation at the neutral element $e\in G$) with elements 
\begin{equation}\label{E(g):eq} X+\xi \in E= E(\mathfrak{g})\coloneqq \mathfrak{g} \oplus \mathfrak{g}^*\end{equation}
and use the same notation to denote them. 
Correspondingly, the three-form $H\in \Gamma ({\bigwedge}^3 T^*G)$ will be identified with an element $H\in {\bigwedge}^3 \mathfrak{g}^*$. 
With these identifications, $\langle \cdot ,\cdot \rangle \in \mathrm{Sym}^2E^*$ and the Dorfman bracket of $X+\xi$ and 
$Y +\eta\in 
\mathfrak{g} \oplus \mathfrak{g}^*$ 
is   
\begin{equation} \label{dorf_LA:eq} [ X+\xi , Y + \eta ]_H = [X,Y] -\mathrm{ad}_X^*\eta   -\iota_Yd\xi + H(X,Y,\cdot ) \in \mathfrak{g} \oplus \mathfrak{g}^*,\end{equation}
where $[X,Y]$ is the Lie bracket in $\mathfrak{g}$, $\mathrm{ad}_X^*\eta = \eta \circ \mathrm{ad}_X$ and 
$d$ denotes the restriction of the de Rham differential to left-invariant forms, such that $-\iota_Yd\xi = \mathrm{ad}_Y^*\xi$. 
\subsection{Generalized metrics on Lie groups} 
\begin{mydef}A \emph{generalized pseudo-Riemannian metric} on a manifold $M$ is a section 
$\mathcal G \in \Gamma (\mathrm{Sym}^2(\mathbb{T}M)^*)$ such that the endomorphism 
$\mathcal{G}^{\mathrm{end}} \in \Gamma (\mathrm{End}\, \mathbb{T}M)$ 
defined by 
\begin{equation} \label{gm:eq}\langle \mathcal{G}^{\mathrm{end}} \cdot ,\cdot \rangle = \mathcal{G}\end{equation}
is an involution and $\mathcal G|_{\mathrm{Sym}^2(T^*M)}$ is non-degenerate. The pair 
$(M,\mathcal G)$ is called a \emph{generalized pseudo-Riemannian manifold}. 
The prefix pseudo will be omitted when $\mathcal G$ is positive definite. 
\end{mydef}
Note that for a generalized metric the equation (\ref{gm:eq}) is equivalent to 
$\mathcal{G}^{\mathrm{end}}=\mathcal{G}^{-1}\circ \langle \cdot ,\cdot \rangle$, using the 
identification $(\mathbb{T}M)^*\otimes (\mathbb{T}M)^* = \mathrm{Hom}(\mathbb{T}M,(\mathbb{T}M)^*)$ given by 
evaluation in the first argument. We do also remark that the non-degeneracy of $\mathcal G|_{\mathrm{Sym}^2(T^*M)}$
is automatic if $\mathcal G$ is positive or negative definite.

A left-invariant generalized metric on a Lie group $G$ is identified (by evaluation at the neutral element $e\in G$) 
with a generalized metric on $\mathfrak{g} =\mathrm{Lie}\, G$ as defined in the following definition. 
\begin{mydef} Let $H$ be a left-invariant closed three-form on  a Lie group $G$, which we identify (by evaluation at $e\in G$) with an element 
$H\in \bigwedge^3\mathfrak{g}^*$ . A \emph{generalized (pseudo-Riemannian) metric} on its Lie algebra $\mathfrak{g} =\mathrm{Lie}\, G$ is a 
symmetric bilinear form $\mathcal G \in \mathrm{Sym}^2E^*$, compare (\ref{E(g):eq}),  
such that 
$\mathcal{G}^{\mathrm{end}}=\mathcal{G}^{-1}\circ \langle \cdot ,\cdot \rangle$ is an involution
and $\mathcal G|_{\mathrm{Sym}^2\mathfrak g^*}$ is non-degenerate. 
The corresponding triple $(G,H,\mathcal G)$ will be called a 
\emph{pseudo-Riemannian generalized Lie group} and $(\mathfrak g,H,\mathcal G)$ a 
\emph{pseudo-Riemannian generalized Lie algebra.} The prefix pseudo will be omitted when 
$\mathcal G$ is positive definite. 

Two pseudo-Riemannian generalized Lie groups $(G,H,\mathcal G)$ and $(G',H',\mathcal G')$ are called 
\emph{isomorphic} if there exists an isomorphism of Lie groups $\varphi : G \rightarrow G'$ and 
an isomorphism of bundles $\Phi : \mathbb{T}G \rightarrow \mathbb{T}G '$ covering 
$\varphi$ such that $\Phi$ maps the Courant algebroid structure 
$(\pi , \langle \cdot , \cdot \rangle , [\cdot ,\cdot ]_H)$ on $G$ determined by $H$ 
to the Courant algebroid structure on $G'$ determined by $H'$ and the 
generalized metric $\mathcal{G}$ to the generalized metric $\mathcal{G}'$. The map $\Phi$ is called 
an \emph{isomorphism} of pseudo-Riemannian generalized Lie groups. 

Similarly, two pseudo-Riemannian generalized Lie algebras $(\mathfrak g,H,\mathcal G)$ and $(\mathfrak g',H',\mathcal G')$ are called \emph{isomorphic} if there exists an isomorphism of Lie algebras $\varphi : \mathfrak g \rightarrow \mathfrak g'$ and an isomorphism of vector spaces $\phi : E(\mathfrak g )\rightarrow E(\mathfrak{g}')$ covering 
$\varphi$ which maps the data  $(\langle \cdot , \cdot \rangle , [\cdot ,\cdot ]_H, \mathcal G)$ on $\mathfrak{g}$, 
cf.\ (\ref{dorf_LA:eq}), to the data $(\langle \cdot , \cdot \rangle' , [\cdot ,\cdot ]_{H'}, \mathcal G')$ on $\mathfrak{g}'$.
Here $\langle \cdot ,\cdot \rangle'$ denotes the canonical symmetric pairing on $E(\mathfrak g')$ induced by the 
duality between $\mathfrak g'$ and $(\mathfrak g')^*$. The map $\phi$ is called 
an \emph{isomorphism} of pseudo-Riemannian generalized Lie algebras. 
\end{mydef}
\begin{ex} \label{metric:ex}Let $g$ be a left-invariant pseudo-Riemannian metric on $G$. 
We denote the corresponding bilinear form on the Lie algebra $\mathfrak g$ by the same symbol:
$g\in \mathrm{Sym}^2\mathfrak g^*$. It extends to a generalized metric 
$\mathcal{G}_g\in \mathrm{Sym}^2E^*$ such that 
\[ \mathcal{G}_g(X+\xi , Y+\eta ) = \frac12 (g(X,Y) + g^{-1}(\xi ,\eta))\]
for all $X+\xi, Y+\eta \in E$. The corresponding endomorphism $\mathcal{G}^{\mathrm{end}}$ is 
\[ \mathcal{G}^{\mathrm{end}}=g \oplus g^{-1} : E=\mathfrak g \oplus \mathfrak g^* \rightarrow E^* = \mathfrak g^* \oplus \mathfrak g.\]
\end{ex}
\begin{prop}\label{NF_metric:prop}Let $(G,H,\mathcal G)$ be a pseudo-Riemannian generalized Lie group. Then there 
exist a left-invariant pseudo-Riemannian metric $g$ on $G$ and a closed left-invariant 
three-form $H'\in [H]\in H^3(\mathfrak{g})$ such that $(G,H,\mathcal G)$ is isomorphic to $(G,H',\mathcal{G}_g)$, by an 
isomorphism $\Phi$ covering the identity map of $G$.
\end{prop}
\begin{proof} The decomposition $E = \mathfrak g \oplus \mathfrak g^*$ gives rise to the following block decomposition   
\[ 2 \mathcal G = \left( \begin{array}{cc}h&A^*\\
A&\gamma \end{array}\right) ,\]
where $h\in \mathrm{Sym}^2\mathfrak g$, $A\in \mathrm{End} (\mathfrak g)$ and 
$\gamma \in \mathrm{Sym}^2\mathfrak g^*$ is non-degenerate, as follows from the symmetry of $\mathcal G$ and the 
non-degeneracy of $\mathcal G|_{\mathrm{Sym}^2 \mathfrak{g}^*}$. 
In terms of $g \coloneqq \gamma^{-1}\in \mathrm{Sym}^2\mathfrak g$ we can write the necessary and sufficient 
conditions for 
\begin{equation} \label{tr:eq}\mathcal{G}^{\mathrm{end}} = \left( \begin{array}{cc}A&g^{-1}\\
h& A^* \end{array}\right)\end{equation}
to be an involution  as:
\[ A^2 + g^{-1}h= \mathbf{1},\quad gA = -A^*g,\quad h A =- A^*h,\]
where the last two equations mean that $A$ is skew-symmetric for $g$ and $h$. 
In particular, we can write $A = -g^{-1}\beta$ for some $\beta \in \bigwedge^2\mathfrak g^*$.
Solving the first equation for $h$ we obtain
\[ h = g -gA^2 = g +\beta A = g -\beta g^{-1} \beta .\]
This implies that $\mathcal{G}^{\mathrm{end}} = \exp (B)  (\mathcal{G}_g)^{\mathrm{end}}  \exp (-B)$, where 
\[ B = \left( \begin{array}{cc}0&0\\
\beta & 0 \end{array}\right),\]
or equivalently, $\mathcal{G} = \exp (-B)^*\mathcal{G}_g$. 
Now it suffices to check that the map 
\[ \phi = \exp (-B) : E \rightarrow E,\quad  
X+\xi \mapsto X+\xi - \beta X,\] 
defines an isomorphism 
of pseudo-Riemannian generalized Lie algebras from $(\mathfrak g, H,\mathcal G)$ 
to $(\mathfrak g, H',\mathcal G_g)$ covering the identity map of $\mathfrak g$, where $H'=H+d\beta$.  
The corresponding isomorphism $\Phi$ of  pseudo-Riemannian generalized Lie groups
is also given by $\exp (-B)$, now considered as an endomorphism of $\mathbb{T}G$. 
\end{proof}
\begin{rem}\label{tr:rem}Clearly, a decomposition of the form (\ref{tr:eq}) holds for any generalized pseudo-Riemannian 
metric $\mathcal{G}$ on a manifold $M$. This shows that $\tr \mathcal{G}^{\mathrm{end}}=0$, since $A$ is skew-symmetric with respect to $g$. 
\end{rem}
\subsection{Space of left-invariant Levi-Civita generalized connections}
Let $H$ be a closed three-form on a smooth manifold $M$ and consider $\mathbb{T}M$ with 
the Courant algebroid structure defined by $H$.
\begin{mydef} A \emph{generalized connection} on $M$ 
is a linear map 
\[ D : \Gamma (\mathbb{T}M)  \rightarrow  \Gamma ((\mathbb{T}M)^*\otimes \mathbb{T}M),\quad v \mapsto Dv = (u\mapsto  D_uv),\]
such that 
\begin{enumerate}
\item $D_u(fv) = u(f) v + fD_uv$ (anchored Leibniz rule), recall (\ref{uf:eq}), and 
\item $u\langle v,w\rangle = \langle D_uv,w\rangle +  \langle v,D_uw\rangle$  
\end{enumerate}
for all $u, v, w\in \Gamma (\mathbb{T}M)$.
The \emph{torsion} of a generalized connection $D$ (with respect to the Dorfman bracket $[\cdot ,\cdot ]_H$) is the section $T\in \Gamma (\bigwedge^2 (\mathbb{T}M)^*\otimes \mathbb{T}M)$ defined by 
\[ T(u,v) \coloneqq D_uv-D_vu -[u,v]_H +(Du)^*v,\]
where $(Du)^*$ is the adjoint of $Du$ with respect to the scalar product, 
compare \cite{G}. The generalized connection $D$ is called \emph{torsion-free} if $T=0$. 

Given a  generalized pseudo-Riemannian metric $\mathcal G$ on $M$, we say that a 
generalized connection $D$ is \emph{metric} if $D\mathcal G=0$, 
where $D_u : \Gamma (\mathbb{T}M)\rightarrow \Gamma (\mathbb{T}M)$ is extended to space of sections of the tensor algebra over $\mathbb{T}M$ as a tensor derivation for all $u\in \Gamma (\mathbb{T}M)$. More explicitly, 
the latter condition is
\[ u\mathcal{G}(v,w) = \mathcal{G}(D_uv,w) + \mathcal{G}(v,D_uw),\quad \forall u,v,w\in \Gamma (\mathbb{T}M).\] 
This condition is satisfied if and only if $D$ preserves the eigenbundles of $\mathcal{G}^{\mathrm{end}}$. 

Any metric and torsion-free generalized connection on a generalized pseudo-Riemannian manifold $(M,\mathcal G)$ (endowed with the three-form $H$) is called a 
\emph{Levi-Civita generalized connection}. 
\end{mydef}
It is known \cite{G} that the torsion of a generalized connection is totally skew, that is 
$T\in  \Gamma (\bigwedge^2 (\mathbb{T}M)^*\otimes \mathbb{T}M)$ defines a
section of  $\bigwedge^3 (\mathbb{T}M)^*$ upon 
identification $\mathbb{T}M\cong (\mathbb{T}M)^*$ using the scalar product. 

Given a reduction of the structure group $\mathrm{O}(n,n)$ of $\mathbb{T}M$, $n=\dim M$, to a subgroup 
$L= \mathrm{O}(n,n)_S\subset \mathrm{O}(n,n)$ defined by a tensor $S\in \bigoplus_{k=0}^\infty \bigotimes^k \left(\mathbb{R}^n\oplus (\mathbb{R}^n)^*\right)$, we consider the tensor field $\mathcal S$ 
which in any frame of the reduction has the same coefficients as $S$ in the standard basis of $\mathbb{R}^n\oplus (\mathbb{R}^n)^*$. A generalized connection $D$ is called \emph{compatible} with the $L$-reduction if
$D\mathcal S =0$. It was shown in \cite{CD} that a torsion-free generalized connection (on a Courant algebroid) 
compatible with an $L$-reduction exists if and only if its intrinsic torsion (defined in \cite[Definition 15]{CD}) vanishes. 
In that case, it was also shown there that the space of compatible torsion-free generalized connections is an affine space modeled on the space of sections of the generalized first prolongation $(\mathfrak{so}(\mathbb{T}M)_\mathcal{S})^{\langle 1\rangle}$ (defined in \cite[Definition 16]{CD}) of $\mathfrak{so}(\mathbb{T}M)_\mathcal{S}$. 
Note that the fiber of the bundle $\mathfrak{so}(\mathbb{T}M)_\mathcal{S}$ at a point $p\in M$ is 
$\mathfrak{so}(\mathbb{T}_pM)_{\mathcal{S}_p}\cong \mathfrak{so}(n,n)_{S}=\mathfrak{l} = \mathrm{Lie}\, L$, 
so that $(\mathfrak{so}(\mathbb{T}M)_\mathcal{S})^{\langle 1\rangle}|_p\cong \mathfrak{l}^{\langle 1\rangle}$. 

As a special case, we can apply the above theory to the case when $\mathcal S = \mathcal G$ is a generalized 
pseudo-Riemannian metric.  The existence of a Levi-Civita generalized connection shown in \cite[Proposition 3.3]{G} 
implies the following.  
\begin{prop} \label{aff:prop}Let $(M,\mathcal{G})$ be a generalized pseudo-Riemannian manifold and $H$ a closed three-form on $M$. Then the space of Levi-Civita generalized connections  (with respect to the $H$-twisted Dorfman bracket) 
is an affine space modeled on $(\mathfrak{so}(\mathbb{T}M)_{\mathcal{G}})^{\langle 1\rangle}$. 
\end{prop}
A generalized connection $D$ on a Lie group $G$ is called \emph{left-invariant} if $D_uv\in \Gamma (\mathbb{T}G)$ is  
left-invariant for all left-invariant sections $u,v\in \Gamma (\mathbb{T}G)$. A left-invariant generalized connection
on $G$ can be identified with an element $D\in E^*\otimes \mathfrak{so}(E)$, where we recall that 
$E=\mathfrak{g} \oplus \mathfrak{g}^*$.  Its torsion 
$T$ is identified with an element $T\in (\bigwedge^2E^*\otimes E)\cap (E^*\otimes \mathfrak{so}(E))\cong \bigwedge^3 E^*$. We denote by $E_+$ and $E_-$ the eigenspaces of $\mathcal{G}^{\mathrm{end}}\in \mathrm{End}(E)$
for the eigenvalues $\pm 1$, respectively.  Note that $\dim E_+ =\dim E_- =\dim G =:n$ by Remark~\ref{tr:rem}.
\begin{prop}\label{LC:prop}Let $(G,H,\mathcal G)$ be a pseudo-Riemannian generalized Lie group. Then the space 
of left-invariant Levi-Civita generalized connections on $G$ is an affine space modeled on 
$\mathfrak{so}(E)^{\langle 1\rangle} = \Sigma_+ \oplus \Sigma_-$, where $\Sigma_+ \subset 
E_+^*\otimes \mathfrak{so}(E_+)$ is the kernel of the 
map 
\[ \partial : E_+^*\otimes \mathfrak{so}(E_+) \rightarrow {\bigwedge}^3 E^*\] 
defined by 
\begin{equation} \label{partial:eq} (\partial \alpha ) (u,v,w) = \sum_{\mathfrak{S}} \langle \alpha_uv,w\rangle \quad u,v,w\in E, \end{equation}
and similarly for $\Sigma_-\subset 
E_-^*\otimes \mathfrak{so}(E_-)$. Here  $\mathfrak{S}$ indicates the sum over the cyclic permutations and 
$\alpha_u\in \mathfrak{so}(E_+)$ stands for evaluation of $\alpha \in E_+^*\otimes \mathfrak{so}(E_+) = \mathrm{Hom}(E_+,\mathfrak{so}(E_+))$ at $u$. 

Moreover, 
\[ \Sigma_+ = \mathrm{im} (\mathrm{alt}) \cong \frac{\mathrm{Sym}^2 E_+\otimes E_+}{\mathrm{Sym}^3E_+}\]
is the image of the map 
\[ \mathrm{alt} : \mathrm{Sym}^2 E_+^* \otimes E_+^*\rightarrow  E_+^*\otimes \mathfrak{so}(E_+)\]
defined by 
\[ \langle \mathrm{alt}(\sigma)_uv,w\rangle = \sigma (u,v,w) -\sigma (u,w,v)\]
and similarly for $\Sigma_-$. 
\end{prop}
\begin{proof}
 The first part of the proposition follows easily from the existence of a left-invariant Levi-Civita generalized connection (to be shown at the end of the proof), Proposition~\ref{aff:prop} and the definition of the generalized first prolongation \cite{CD}  as the kernel of the natural map 
\[ \partial : E^*\otimes \mathfrak{so}(E)_{\mathcal{G}} \rightarrow {\bigwedge}^3 E^*\] 
given by the formula (\ref{partial:eq}). To compute the kernel we can first 
observe that $\mathfrak{so}(E)_{\mathcal{G}} = \mathfrak{so}(E_+) \oplus \mathfrak{so}(E_-)\cong 
\bigwedge^2 E_+^* \oplus \bigwedge^2 E_-^*$.  Since $\partial$ 
maps $E_{\epsilon_1}^*\otimes \mathfrak{so}(E_{\epsilon_2})$ to
$E_{\epsilon_1}^*\wedge  E_{\epsilon_2}^*\wedge E_{\epsilon_2}^* \subset \bigwedge^3 E^*$, $\epsilon_1, \epsilon_2
\in \{ -1, 1\}$, it suffices to consider the kernels of these four restrictions. On tensors of 
mixed type $\partial$ is injective, such that $\ker \partial = \Sigma_+ \oplus \Sigma_-$. 
The last part of the corollary follows from the exact sequence 
\begin{equation} \label{ex:eq} 0\rightarrow \mathrm{Sym}^3V\rightarrow \mathrm{Sym}^2V\otimes V \stackrel{\mathrm{alt}_V}{\longrightarrow} V \otimes {\bigwedge}^2 V \stackrel{\partial_V}{\longrightarrow} {\bigwedge}^3V\rightarrow 0\end{equation}
that holds for any finite-dimensional vector space $V$ and was used in \cite{G}.  
Here $\mathrm{alt}_V$ is given by
\[ (u\otimes v + v\otimes u)\otimes w \mapsto u\otimes v\wedge w + v\otimes w
\wedge u\]
and $\partial_V$ by 
\[ u\otimes v\wedge w \mapsto u\wedge v\wedge w.\]
We apply the sequence to $V=E_+$ (and similarly to $V=E_-$)
using the metric identifications $E_+ \cong E_+^*$ and $\mathfrak{so}(E_+) \cong \bigwedge^2E_+^*\cong \bigwedge^2E_+$, which allow to identify the natural maps $\mathrm{alt}_V$ and $\partial_V$ with 
$\mathrm{alt}: \mathrm{Sym}^2E_+^*\otimes E_+^* \rightarrow E_+^* \otimes \mathfrak{so}(E_+ )$ 
and $\partial : E_+^* \otimes \mathfrak{so}(E_+ ) \rightarrow \bigwedge^3 E_+^*$, respectively.

Now it suffices to show that there exists a left-invariant Levi-Civita generalized connection. 
We consider the tensor $\mathcal{B} \in  \bigotimes^3 E^*$ defined by 
\begin{equation} 
\mathcal{B}(u,v,w) = \langle [u,v]_H,w\rangle,\quad u,v,w\in E.
\end{equation}
\begin{lem} \label{B_inv:lem}$\mathcal{B}$ is totally skew. 
\end{lem} 
\begin{proof}
The skew-symmetry in 
$(u,v)$ follows from axiom $C3$ in Section~\ref{prelim:subsec}: 
\[ \mathcal{B}(u,v,w)+\mathcal{B}(v,u,w)= \langle w, [u,v]_H+[v,u]_H\rangle = w\langle u,v\rangle =0,\]
since $\langle u,v\rangle$ is a constant function. Using axiom C2, we obtain
\[ \mathcal{B}(u,v,w) = \langle [u,v]_H,w\rangle = u\langle v,w\rangle -\langle v,[u,w]_H\rangle = 
-\mathcal{B}(u,w,v).\]
Now it suffices to observe that skew-symmetry in $(u,v)$ and $(v,w)$ implies total 
skew-symmetry. 
\end{proof}
Next we define 
\[ D^0 \coloneqq \frac13 \mathcal{B}|_{\bigwedge^3E_+} \oplus  \frac13 \mathcal{B}|_{\bigwedge^3E_-}\oplus   \mathcal{B}|_{E_+\otimes \bigwedge^2 E_-} \oplus \mathcal{B}|_{E_-\otimes \bigwedge^2 E_+} .\]
As an element of $E^*\otimes \bigwedge^2 E^*\cong E^*\otimes \mathfrak{so}(E)$, it
defines a left-invariant generalized connection. It is metric, since it takes values in the subalgebra 
$\mathfrak{so}(E_+)\oplus \mathfrak{so}(E_-) \subset \mathfrak{so}(E)$. Since $\partial \mathcal{B}|_{\bigwedge^3E_\pm}
= 3 \mathcal{B}|_{\bigwedge^3E_\pm}$ and $\partial \mathcal{B}|_{E_\mp\otimes \bigwedge^2 E_\pm}= \mathcal{B}|_{E_\mp\wedge E_\pm\wedge E_\pm}$, the torsion $T^{D^0}= \partial D^0 -\mathcal{B}$ of $D^0$ is 
given by 
\[ T^{D^0}= \left( \mathcal{B}|_{\bigwedge^3E_+} \oplus   \mathcal{B}|_{\bigwedge^3E_-}\oplus   \partial \mathcal{B}|_{E_+\otimes \bigwedge^2 E_-} \oplus \partial \mathcal{B}|_{E_-\otimes \bigwedge^2 E_+}\right) -\mathcal{B} = \mathcal{B} -\mathcal{B}=0.\qedhere\]
\end{proof}
\begin{rem} Note that due to Lemma~\ref{B_inv:lem} and the Jacobi identity (axiom C1) the tensor $\mathcal{B}$ together 
with the scalar product $\langle \cdot, \cdot \rangle$ defines on $E(\mathfrak{g})$ the structure of a quadratic Lie algebra. 
Such algebras are examples of Courant algebroids with trivial anchor. Generalized metrics, generalized connections and
curvature on quadratic Lie algebras have been studied in \cite{ADG}. Their formulas are consistent with ours. 
\end{rem}

\subsection{Levi-Civita generalized connections with prescribed divergence}
\label{divergence:Section}
In this subsection we show that every left-invariant divergence operator on the generalized tangent bundle of 
a generalized pseudo-Riemannian Lie group admits a compatible left-invariant Levi-Civita generalized connection.
We then give an explicit construction of such a  generalized connection in the case when 
$\mathcal G$ is associated with a left-invariant pseudo-Riemannian metric as in Example~\ref{metric:ex}. 
In view of Proposition~\ref{NF_metric:prop} there is no loss in generality by considering this special case. 

\begin{mydef} A \emph{divergence operator} on $\mathbb{T}M$ is a first order 
differential operator $\delta : \Gamma (\mathbb{T}M) \rightarrow C^\infty (M)$ which satisfies 
\[ \delta (fv) = v(f) + f\delta v,\]
for all $v\in \Gamma (\mathbb{T}M)$, $f\in C^\infty (M)$.
\end{mydef} 
\begin{ex}\label{divergence of connection:ex}Let $D$ be a generalized connection on $M$. Then 
\[ \delta_Dv = \tr Dv,\quad v\in \Gamma (\mathbb{T}M),\]
defines a divergence operator on $\mathbb{T}M$. 
\end{ex}

When $M=G$ is a Lie group we can ask for a divergence operator $\delta$ on $\mathbb{T}G$ to be 
left-invariant, that is for the function $\delta v$ to be left-invariant (i.e.\ constant) for all left-invariant 
sections $v$ of $\mathbb{T}G$. Such operators can can be identified with 
elements of $E^*=(\mathbb{T}_eG)^*$. 

It was proved in \cite{G} that there always exists a Levi-Civita generalized connection with a prescribed divergence. We now give a proof for this in our setting.
\begin{prop}\label{existence of connection with divergence:prop} Let $(G,H,\mathcal G)$ be a generalized pseudo-Riemannian Lie group of dimension $\dim G\ge 2$ and $\delta\in E^*$.
Then there exists a left-invariant Levi-Civita generalized connection $D$ such that $\delta_D=\delta$. 
\end{prop}
\begin{proof}
Let $D\in E^*\otimes \mathfrak{so}(E)$ be a left-invariant Levi-Civita generalized connection. Any other left-invariant Levi-Civita generalized connection can be written as $D'=D + S$, where 
$S\in \mathfrak{so}(E)^{\langle 1\rangle}\subset E^*\otimes 
\mathfrak{so}(E)$ (see Proposition~\ref{LC:prop}). The divergence operators are related by
\begin{equation} \label{divDDprime:eq}\delta_{D'}v - \delta_Dv = \tr S v = \tr (u\mapsto S_uv),\quad v\in E.\end{equation}
We consider the linear form $\lambda_S \in E^*$ defined by  
\begin{equation}\label{lambdaS:eq} \lambda_S (v) \coloneqq \tr S v.\end{equation}
It suffices to show that the linear map $S \mapsto \lambda_S$ is surjective. 
Given $\alpha, \beta \in E_+^*\cong (E_-)^0\subset E^*$, the element $S=\mathrm{alt} (\alpha^2 \otimes \beta)\in \Sigma_+
\subset \mathfrak{so}(E)^{\langle 1\rangle}=\Sigma_+\oplus \Sigma_-$ has 
\begin{equation} \label{lambdaSspecial:eq}\lambda_S = \langle \alpha,\beta \rangle \alpha -\langle \alpha ,\alpha\rangle \beta.\end{equation}
Since $\dim E_+=\dim G\ge 2$, this proves that $\mathrm{span}\{ \lambda_S \mid S\in \Sigma_+\}=E_+^*$, 
and similarly $\mathrm{span}\{ \lambda_S \mid S\in \Sigma_-\}=E_-^*$. 
\end{proof}
Note that the condition $\dim G\ge 2$ is necessary. If $\dim G =1$, then the Levi-Civita
generalized connection $D$ is unique and $\delta_D\in E^*$ is zero.  

From now on we assume without loss of generality (see Proposition~\ref{NF_metric:prop}) that 
$\mathcal G = \mathcal G_g$ for some left-invariant 
pseudo-Riemannian metric $g$ on $G$. We will first construct a particular left-invariant Levi-Civita 
generalized connection $D$ with $\delta_D=0\in E^*$ and later prescribe an arbitrary divergence operator
by adding a suitable element of the generalized first prolongation. 

\textbf{Adapted bases and notation.}
Let $(v_a)=(v_1,\ldots,v_n)$ be a $g$-orthonormal basis of $\mathfrak{g}$ and set 
$\varepsilon_a \coloneqq g(v_a,v_a)$. Then 
\begin{equation} \label{ea:eq}e_a \coloneqq v_a + gv_a\end{equation}
defines a 
$\mathcal G$-orthonormal basis $(e_a)_{a=1,\ldots ,n}$ of $E_+$ with 
$\mathcal{G}(e_a,e_a) = \varepsilon_a$ and 
\begin{equation} \label{ei:eq}e_{n+a}\coloneqq v_a-gv_a\end{equation}
defines 
a $\mathcal G$-orthonormal basis $(e_i)_{i=n+1,\ldots ,2n}$ of $E_-$ with $\mathcal{G}(e_{n+a},e_{n+a}) = \varepsilon_{a}$. 
Remember that $\langle \cdot ,\cdot \rangle = \pm \mathcal G$ on the summands $E_\pm$ of the 
decomposition $E=E_+ \oplus E_-$, which is orthogonal for both the generalized metric $\mathcal G$ as well as 
the scalar product $\langle \cdot ,\cdot \rangle$. 
Summarizing, we have an orthonormal basis $(e_A)_{A=1,\ldots ,2n}$ of $E$ adapted to the decomposition 
$E=E_+\oplus E_-$. Note that $\langle e_A,e_B\rangle = \varepsilon_A\delta_{AB}$, where $\varepsilon_a = -\varepsilon_{n+a}$ for $a=1,\ldots n$. From now on the indices $a, b, \ldots$ will always range 
from $1$ to $n$, $i, j, \ldots$  will range from $n+1$ to $2n$ and $A, B, \ldots$ from $1$ to $2n$. 

A left-invariant generalized connection $D$ is completely determined by its 
coefficients $\omega_{AB}^C$ with respect to the basis $(e_A)$:
\[ D_{e_A}e_B= \omega_{AB}^Ce_C,\]
where from now on we use Einstein's summation convention, according to which the sum over an upper and a lower repeated index is understood. 
Equivalently, we may use 
\begin{equation} \label{coeff:eq}\omega_{ABC} \coloneqq \langle D_{e_A}e_B,e_C\rangle,\end{equation}
which has 
the advantage that it is skew-symmetric in $(B,C)$.   In fact, any tensor $(\omega_{ABC})$ skew-symmetric in $(B,C)$
defines a left-invariant generalized connection $D$ by the formula (\ref{coeff:eq}). We will say that $(\omega_{ABC})$ are the \emph{connection coefficients} of $D$. 

The next proposition follows from the fact that $D$ is metric if and only if $DE_\pm\subset E_\pm$. 
\begin{prop} \label{metric:prop}A left-invariant generalized connection $D$ is metric if and only if 
$\omega_{ABC}=0$ whenever $B\in \{1,\ldots ,n\}$ and $C\in\{ n+1,\ldots ,2n\}$. 
\end{prop}
Using the orthonormal basis $(e_A)$  of $E$ we define 
\begin{equation} \label{B:eq}\mathcal{B}_{ABC} \coloneqq \mathcal{B}(e_A,e_B,e_C)=\langle [e_A,e_B]_H,e_C\rangle.\end{equation}
\begin{prop}  \label{LCdiv0:prop}
	Let $(G,H,\mathcal G_g)$ be a generalized pseudo-Riemannian Lie group.
The following tensor 
$(\omega_{ABC})$ defines the connection coefficients of a left-invariant 
Levi-Civita generalized connection $D^0$ with zero divergence $\delta_{D^0}$: 
\begin{equation}\label{conn:eq}\omega_{abc} \coloneqq \frac13 \mathcal{B}_{abc},\quad \omega_{ijk} \coloneqq \frac13 \mathcal{B}_{ijk},\quad \omega_{ibc} \coloneqq \mathcal{B}_{ibc},\quad \omega_{ajk} \coloneqq \mathcal{B}_{ajk},\end{equation}
where $a,b,c\in \{ 1,\ldots ,n\}$ and $i,j,k\in \{ n+1,\ldots ,2n\}$ and the remaining components are zero. 
The connection $D^0$ does not depend on the choice of orthonormal basis $(v_a)$ of $\mathfrak{g}$, 
from which the orthonormal basis $(e_A)$ of $E=\mathfrak{g}\oplus \mathfrak{g}^*$ was constructed.  
It is therefore a canonical Levi-Civita generalized connection and will be called \textit{the canonical divergence-free
Levi-Civita generalized connection}. 
\end{prop}

\begin{proof}
The formulas (\ref{conn:eq}) are precisely the connection coefficients of the 
left-invariant Levi-Civita generalized connection $D^0$ defined in the proof of Proposition~\ref{LC:prop}. 
In particular, $D^0$ is independent of the basis $(v_a)$. 
To show that the divergence $\delta$ of $D^0$ vanishes, it suffices to remark that 
$\delta (e_B) = \omega_{AB}^A$ vanishes due to $\omega_{ajc}=\omega_{ibk}=0$
and the total skew-symmetry of $\omega_{abc}$ and $\omega_{ijk}$ (with the above index ranges), implied 
by Lemma~\ref{B_inv:lem}.  
\end{proof}
\begin{prop} \label{S:prop}
Let $(G,H,\mathcal G_g)$ be a generalized pseudo-Riemannian Lie group
endowed with the canonical divergence-free Levi-Civita generalized connection $D^0$ of Proposition~\ref{LCdiv0:prop}. Fix an element $\delta \in E^*$. Then a left-invariant Levi-Civita generalized connection 
$D$ with divergence $\delta_D=\delta$ can be obtained as follows. Choose, as above\footnote{Compare (\ref{ea:eq}) and (\ref{ei:eq}).}, a 
left-invariant orthonormal basis $(e_A)$ of $E$ associated with an orthonormal 
basis of $\mathfrak g$.   Define the tensor $ S:=S_+ +S_- $ where
\[ S_+\coloneqq -\mathrm{alt}\left(\delta_1 \varepsilon_2 (e^2)^2 \otimes e^1+ \sum_{a=2}^{n}  \delta_a\varepsilon_1(e^1)^2\otimes e^a \right)\in \Sigma_+,\]
and similarly for $ S_- \in\Sigma_-$. Here $(e^A)$ denotes the basis of $E^*$ dual to $(e_A)$ and $\delta_A = \delta (e_A)$. 
Then the left-invariant Levi-Civita generalized connection $D=D^0 +S$ has divergence $\delta$. 
\end{prop}
\begin{proof}
From (\ref{divDDprime:eq}), (\ref{lambdaS:eq}) and (\ref{lambdaSspecial:eq}) we see that $D=D^0 + S$ has divergence $ \delta $, since
\[ \lambda_{S_+} = -\delta_1 \varepsilon_2\lambda_{(e^2)^2 \otimes e^1} -\sum_{a=2}^{n}\delta_a\varepsilon_1\lambda_{(e^1)^2\otimes e^a}=\sum_{a=1}^{n}\delta_ae^a=\delta|_{E_+}\]
and similarly $ \lambda_{S_-}=\delta|_{E_-} $.
\end{proof}
We want to close this section by introducing a special divergence operator, the so-called Riemannian divergence, which is considered in the literature (\cite[Definition 2.46]{GSt}). If $ (M,\mathcal{G}) $ is a generalized pseudo-Riemannian manifold, one defines for all $v\in \Gamma (\mathbb{T}M)$
\begin{equation*}
	\delta^\mathcal{G}(v)= \tr \left(\nabla \pi v\right)=\tr\left(\Gamma(TM)\ni Y \mapsto \nabla_Y\pi(v)\in \Gamma(TM)\right),
\end{equation*}
where $ \nabla $ is the Levi-Civita connection of the pseudo-Riemannian metric $ g $ associated to $ \mathcal{G} $ via Proposition~\ref{NF_metric:prop}. Denoting by $ \mu $ the Riemannian density associated to $ g $, we recall the well known fact that the divergence $\tr \left(\nabla X\right)$ of a vector field $X$ can also be expressed by 
$\frac{\mathcal{L}_X \mu}{\mu}$, since
\begin{equation*}
	\mathcal{L}_X \mu = \nabla_X\mu -(\nabla X) \cdot \mu = \tr (\nabla X)\mu. 
\end{equation*}
The divergence operator $\delta^\mathcal{G}$ can be recovered as the divergence of a generalized connection as in Example \ref{divergence of connection:ex}. For that one first extends the Levi-Civita connection to a connection on $ \mathbb{T}M $ and then pulls it back to a generalized connection $ \widetilde{\nabla} $ via the anchor $ \pi $. Then 
\begin{equation*}
	\delta_{\widetilde{\nabla}}(v)=\tr_{\mathbb{T}M}\left(\widetilde{\nabla}v\right)=\tr_{TM}\left(\nabla\pi(v)\right)=\delta^\mathcal{G}(v),
\end{equation*}
since $ \widetilde{\nabla}v|_{T^*M}=0 $ and $ \pi\circ\widetilde{\nabla}v|_{TM}=\nabla\pi(v)$. Furthermore, note that $ \widetilde{\nabla} $ is a Levi-Civita generalized connection of $ \mathcal{G} $, if $ \mathcal{G}=\mathcal{G}^g $ and $H=0$.
\begin{prop}\label{Riemdiv:prop}
	Let $(G, H,\mathcal G)$ be a generalized pseudo-Riemannian Lie group. Then the Riemannian divergence satisfies
	\begin{equation*}
			\delta^\mathcal{G}(v)=-\tau(\pi(v)),  \quad v\in E,
		\end{equation*}
	where $\tau\in \mathfrak{g}^*$ is the trace-form defined by $\tau (X) = \tr \mathrm{ad}_{X}$, $X\in\mathfrak{g}$. In particular, the Riemannian divergence is zero, if the Lie group $ G $ is unimodular.
\end{prop}
\begin{proof}
	Let $ v=X+\xi\in E $ and $ (v_a) $ as usual a basis of $ \mathfrak{g} $, which is orthonormal with respect to $ g $. Furthermore, let $ \nabla $ be the Levi-Civita connection of $ g $. It satisfies 
	\begin{equation*}
		g\left(\nabla_XY,Z\right)=\frac12\left(g\left([X,Y],Z\right)-g\left([Y,Z],X\right)+g\left([Z,X],Y\right)\right)
	\end{equation*}
	for $ X,Y,Z\in\mathfrak{g} $. We can thus compute 
	\begin{eqnarray*}
			\delta^\mathcal{G}(X+\xi)&=&\tr(\nabla X)\\
			&=&\sum_a\varepsilon_ag(\nabla_{v_a}X,v_a)\\
			&=&\frac12\sum_a\varepsilon_a\left(g([v_a,X],v_a)-g([X,v_a],v_a)+g([v_a,v_a],X)\right)\\
			&=&-\sum_a\varepsilon_ag([X,v_a],v_a)\\
			&=&-\tr\mathrm{ad}_X\\
			&=&-\tau(\pi(X+\xi)).
		\end{eqnarray*}
\end{proof}

\subsection{Ricci curvatures and generalized Einstein metrics}
\label{mainpart2ndSec}
After fixing a left-invariant section $\delta$ of $(\mathbb{T}G)^*$ over a generalized 
pseudo-Riemannian Lie group $(G,H,\mathcal G)$ we 
define and compute two canonical Ricci curvature tensors 
$Ric^+\in E_-^*\otimes E_+^*$ and $Ric^-\in E_+^*\otimes E_-^*$, which depend only on 
the data $(H,\mathcal G,\delta )$. A left-invariant solution $\mathcal{G}$ of the system $Ric^+=0, Ric^-=0$ is what we will call a \emph{generalized Einstein metric} on $G$ with three-form $H$ and \emph{dilaton}  
$\delta$.

Consider the generalized tangent bundle $\mathbb{T}M$ of a smooth manifold endowed with the 
Courant algebroid structure associated with a closed three-form $H$ on $M$ and a generalized 
pseudo-Riemannian metric $\mathcal G$. We denote by $(\mathbb{T}M)_\pm$ the eigenbundles 
of $\mathcal G^{\mathrm{end}}$.

Given a Levi-Civita generalized connection $D$ on $\mathbb{T}M$ 
and two sections $u,v\in \Gamma (\mathbb{T}M)$,  
we consider the differential operator $R(u,v): \Gamma (\mathbb{T}M)\rightarrow \Gamma (\mathbb{T}M)$ defined by 
\[ R(u,v)w \coloneqq D_uD_vw-D_vD_uw-D_{[u,v]_H}w,\] for all 
$w\in \Gamma (\mathbb{T}M)$. 
It was observed in \cite{G} that $R$ restricts to tensor fields 
\begin{eqnarray*} R_D^+&\in& \Gamma \left((\mathbb{T}M)_+^*\otimes (\mathbb{T}M)_-^*\otimes \mathfrak{so}((\mathbb{T}M)_+)\right),\\
R_D^-&\in& \Gamma \left((\mathbb{T}M)_-^*\otimes (\mathbb{T}M)_+^*\otimes \mathfrak{so}((\mathbb{T}M)_-)\right). \end{eqnarray*}
Hence there are tensor fields $Ric_D^+\in \Gamma ((\mathbb{T}M)_-^*\otimes (\mathbb{T}M)_+^*)$ and 
$Ric_D^-\in \Gamma ((\mathbb{T}M)_+^*\otimes (\mathbb{T}M)_-^*)$ defined by 
\begin{equation*}\begin{split}
	Ric_D^+(u,v) = \tr R_D^+(\cdot , u)v = \tr \left(\Gamma({\mathbb{T}M}_+)\ni w\mapsto R(w,u)v\in \Gamma({\mathbb{T}M}_+)\right)&,\\
	u\in \Gamma({\mathbb{T}M}_-)&, \, v\in \Gamma({\mathbb{T}M}_+),\\
	Ric_D^-(u,v) = \tr R_D^-(\cdot , u)v = \tr \left(\Gamma({\mathbb{T}M}_-)\ni w\mapsto R(w,u)v\in  \Gamma({\mathbb{T}M}_-)\right)&,\\ 
	u\in \Gamma({\mathbb{T}M}_+)&, \, v\in \Gamma({\mathbb{T}M}_-).
\end{split}\end{equation*}
It was also shown in \cite{G} that the tensor fields $Ric_{D_1}^\pm$ and $Ric_{D_2}^\pm$ 
are the same for any pair of Levi-Civita generalized connections $D_1$, $D_2$ 
with the same divergence operator $\delta_{D_1}=\delta_{D_2}$.  

As a consequence, the following definition is meaningful. 
\begin{mydef} \label{Ricdef} Let $(G,H,\mathcal G)$ be a generalized pseudo-Riemannian Lie group and 
$\delta\in E^*$. Then the \emph{Ricci curvatures}  
\[ Ric^+=Ric^+_\delta\in E_-^*\otimes E_+^*\quad\mbox{and}\quad Ric^-=Ric^-_\delta\in E_+^*\otimes E_-^*\] 
of  $(G,H,\mathcal G, \delta)$ (or of $(\mathfrak{g},H,\mathcal G, \delta)$)  are defined by evaluation of $Ric_D^+$ and $Ric_D^-$ at $e\in G$, where 
$D$ is any left-invariant Levi-Civita generalized connection $D$ with divergence $\delta$. 
$(G,H,\mathcal G, \delta)$ is called \emph{generalized Einstein} if 
\[ Ric \coloneqq Ric^+\oplus Ric^- =0 \in E_-^*\otimes E_+^* \oplus E_+^*\otimes E_-^*.\]
We will consider $Ric$ as a bilinear form on $E$ vanishing on $E_+\times E_+$ and $E_-\times E_-$. 
\end{mydef}
Next we compute the Ricci curvatures in the case $\delta=0$ using the canonical divergence-free 
Levi-Civita generalized connection of Proposition~\ref{LCdiv0:prop}, which in the following 
we denote by $D^0$. The case of general 
divergence is then obtained by computing how the Ricci curvatures
change under addition of an element of the generalized first prolongation. 
We denote by $R^\pm_{D^0}\in E_\pm^*\otimes E_\mp^*\otimes \mathfrak{so}(E_\pm)$
the tensors which correspond to the left-invariant 
tensor fields $R^\pm_{D^0}\in \Gamma \left( (\mathbb{T}G)_\pm^*\otimes (\mathbb{T}G)_\mp^*\otimes \mathfrak{so}((\mathbb{T}G)_\pm)\right)$. 

\begin{prop}\label{curv:prop}Let $D^0$ be the canonical divergence-free Levi-Civita generalized connection of a 
generalized pseudo-Riemannian Lie group $(G,H,\mathcal G_g)$, defined  in Proposition~\ref{LCdiv0:prop}. 
The components $R_{ABCD}\coloneqq \langle R(e_A,e_B)e_C,e_D\rangle$, $A,B,C,D\in \{ 1,\ldots ,2n\}$, of the tensors 
$R^\pm_{D^0}$ are given by
\begin{eqnarray*} R_{ajcd} &=&  \frac23 \mathcal{B}_{aj}^\ell \mathcal{B}_{c\ell d} +\frac13 \mathcal{B}_{jc}^\ell \mathcal{B}_{\ell a d} +\frac13 \mathcal{B}_{ca}^\ell \mathcal{B}_{\ell j d},\\
R_{ibk\ell } &=& \frac23 \mathcal{B}_{ib}^c \mathcal{B}_{kc \ell}  +\frac13 \mathcal{B}_{bk}^c\mathcal{B}_{ci\ell }
+\frac13 \mathcal{B}_{ki}^c\mathcal{B}_{cb\ell },
\end{eqnarray*}
where $a,b,c,d\in \{ 1,\ldots ,n\}$ and $i,j,k,\ell \in \{ n+1,\ldots ,2n\}$. 
\end{prop}
\begin{proof}
We denote by $\eta_{AB} =\langle e_A,e_B\rangle$ the coefficients of the 
scalar product with respect to the orthonormal basis $(e_A)$ and by $\omega_{ABC}$ and $\omega_{AB}^C=\sum_{D} \eta^{CD} \omega_{ABD}$ the connection coefficients of $D^0$. Here $\eta^{AB}=\eta_{AB}$ are the coefficients of the 
induced scalar product on $E^*$. Then (taking into account the agreed index ranges) we compute 
\begin{eqnarray*}
R_{D^0}^+(e_a,e_j)e_c&=&(\omega_{jc}^d\omega_{ad}^f -\omega_{ac}^d\omega_{jd}^f - 
\mathcal{B}_{aj}^D\omega_{Dc}^f)e_f,\\
&=&(\omega_{jc}^d\omega_{ad}^f -\omega_{ac}^d\omega_{jd}^f - 
\mathcal{B}_{aj}^d\omega_{dc}^f- 
\mathcal{B}_{aj}^\ell \omega_{\ell c}^f)e_f\\
&=&(\frac13 \mathcal{B}_{jc}^d\mathcal{B}_{ad}^f -\frac13 \mathcal{B}_{ac}^d\mathcal{B}_{jd}^f - \frac13 
\mathcal{B}_{aj}^d\mathcal{B}_{dc}^f- 
\mathcal{B}_{aj}^\ell \mathcal{B}_{\ell c}^f)e_f,
\end{eqnarray*}
where the index $f$ runs from $1$ to $n$ and $\mathcal{B}_{AB}^C=\mathcal{B}_{ABD}\eta^{DC}$. 
Next we observe that the axiom (C1), the Jacobi identity for the Dorfman bracket, 
can be written in components as
\[ \sum_{\mathfrak{S}(A,B,C)} \mathcal{B}_{AD}^F\mathcal{B}_{BC}^D=0,\]
where the cyclic sum is over $(A,B,C)$. 
Specializing to $(A,B,C,F)=(a,j,c,f)$ we get 
\[ 0=\sum_{\mathfrak{S}(a,j,c)} \mathcal{B}_{aD}^f\mathcal{B}_{jc}^D=\sum_{\mathfrak{S}(a,j,c)}(\mathcal{B}_{ad}^f\mathcal{B}_{jc}^d +\mathcal{B}_{a\ell }^f\mathcal{B}_{jc}^\ell ).\]
So we obtain
\begin{eqnarray*} R_{D^0}^+(e_a,e_j)e_c&=& -\left(  \mathcal{B}_{aj}^\ell \mathcal{B}_{\ell c}^f +\frac13 \sum_{\mathfrak{S}(a,j,c)} \mathcal{B}_{jc}^\ell \mathcal{B}_{a\ell }^f \right) e_f\\
&=&\left( \frac23 \mathcal{B}_{aj}^\ell \mathcal{B}_{c\ell }^f +\frac13 \mathcal{B}_{jc}^\ell \mathcal{B}_{\ell a}^f +\frac13 \mathcal{B}_{ca}^\ell \mathcal{B}_{\ell j}^f\right) e_f
\end{eqnarray*}
Taking the scalar product with $e_d$ gives the claimed formula for $R_{ajcd}$. The other 
formula is obtained similarly. 
\end{proof}
\begin{cor}
\label{RicCor_delta0}
Let $(G,H,\mathcal G_g)$ be  a 
generalized pseudo-Riemannian Lie group. Then the Ricci curvature of $(G,H,\mathcal G_g,\delta=0)$ is 
symmetric, in the sense that $Ric^+(u,v)=Ric^-(v,u)$ for all 
$u\in E_-$, $v\in E_+$. The components $R_{ia}\coloneqq Ric^+(e_i,e_a)$ of 
$Ric^+$ are given by
\[ R_{ia} =  \mathcal{B}_{bi}^j \mathcal{B}_{aj}^b,
\]
where $a,b\in \{ 1,\ldots ,n\}$ and $i,j\in \{ n+1,\ldots ,2n\}$.  \end{cor}
\begin{proof}From Proposition~\ref{curv:prop}, by taking the trace using the complete skew-symmetry of $\mathcal{B}_{ABC}$, see Lemma~\ref{B_inv:lem}, we get:
\begin{eqnarray*} R_{ia} &=&  R_{ai}= \frac23 \mathcal{B}_{bi}^j \mathcal{B}_{aj}^b +\frac13 \mathcal{B}_{ab}^j \mathcal{B}_{j i}^b\\
&=& \eta^{bb'} \eta^{jj'}\left( \frac23 \mathcal{B}_{bij} \mathcal{B}_{aj'b'} +\frac13 \mathcal{B}_{ab'j'} \mathcal{B}_{j ib}\right)\\
&=& \eta^{bb'} \eta^{jj'} \mathcal{B}_{bij} \mathcal{B}_{aj'b'} = \mathcal{B}_{bi}^j \mathcal{B}_{aj}^b.
\end{eqnarray*}
\end{proof}
For $u_\pm \in E_\pm$ we define 
\begin{equation}\label{Gamma:eq} \Gamma_{u_+} \coloneqq \mathrm{pr}_{E_+}\circ {[}u_+, \cdot {]}_H\big|_{E_-} : E_- \rightarrow E_+,\quad\Gamma_{u_-} \coloneqq \mathrm{pr}_{E_-}\circ {[}u_-, \cdot {]}_H\big|_{E_+} : E_+ \rightarrow E_-.
\end{equation}

\begin{cor} A necessary and sufficient condition for $(G,H,\mathcal G_g,\delta=0)$ to be generalized Einstein is that 
the subspace 
\[ \Gamma_{E_+} \subset \mathrm{Hom}(E_-,E_+)\quad \mbox{is perpendicular to}\quad  
\Gamma_{E_-} \subset \mathrm{Hom}(E_+,E_-),\]
with respect to the non-degenerate pairing $\mathrm{Hom}(E_-,E_+)\times \mathrm{Hom}(E_+,E_-)\rightarrow \mathbb{R}$ given by $(A,B) \mapsto \tr (AB) = \tr (BA)$. A sufficient condition in terms of the 
subspaces $\Gamma_{E_\pm}E_{\mp} \subset E_\pm$ is that 
\begin{equation} \label{suff:eq}\Gamma_{E_+}E_-\perp [E_-,E_-]_H\quad\mbox{or}\quad \Gamma_{E_-}E_+\perp [E_+,E_+]_H.\end{equation}
\end{cor}
\begin{proof}The necessary and sufficient condition follows immediately from \[R_{ia} = R_{ai}=  \mathcal{B}_{bi}^j \mathcal{B}_{aj}^b=-\tr (\Gamma_{e_a}\circ \Gamma_{e_i}).\] Any of the two (non-equivalent) conditions $\Gamma_{E_+} \circ \Gamma_{E_-}=0$ or $\Gamma_{E_-} \circ \Gamma_{E_+}=0$ is clearly sufficient.  These can be reformulated as (\ref{suff:eq}), since by Lemma~\ref{B_inv:lem}, 
\[ \langle \Gamma_{u_+}v_-,w_+\rangle = -\langle v_-,[u_+,w_+]_H\rangle\quad\mbox{and}\quad
\langle \Gamma_{u_-}v_+,w_-\rangle = -\langle v_+,[u_-,w_-]_H\rangle,\]
for all $u_+,v_+, w_+\in E_+$, $u_-,v_-,w_-\in E_-$. 
\end{proof}
Next we will compute the Ricci curvature of an arbitrary left-invariant Levi-Civita generalized connection 
$D=D^0 + S$ on $(G,H,\mathcal G_g)$, where $D^0$ is the canonical divergence-free Levi-Civita 
generalized connection and $S$ is an arbitrary element of the first generalized prolongation of $\mathfrak{so}(E)$.   

\begin{lem}\label{curv:lem}The curvature tensors $R^\pm_D\in  \mathrm{Hom}(E_\pm \otimes E_{\mp}\otimes E_\pm,E_\pm)$ of $D$ are given by
\begin{equation} \label{R:eq}R_{D}^\pm = R_{D^0}^\pm + d^{D^0}S|_{E_\pm \otimes E_{\mp}\otimes E_\pm},\end{equation}
where 
\[ (d^{D^0}S)(u,v,w) = (d^{D^0}S)(u,v)w  \coloneqq D^0_{u}(S_{v})w-D^0_{v}(S_{u})w -S_{{[}u,v{]}_H}w,\quad u,v,w\in E.\]
\end{lem}
\begin{proof} A straightforward calculation shows that $R_{D}^\pm = R_{D^0}^\pm + (d^{D^0}S + [S,S])|_{E_\pm \otimes E_{\mp}\otimes E_\pm},$
where 
\[ {[}S,S{]}(u,v,w) = {[}S,S{]}(u,v)w\coloneqq {[}S_{u},S_{v}{]}w,\quad u,v,w\in E.\]
We observe that the map $[S,S]: (u,v,w) \mapsto [S_u,S_v]w$ vanishes on $E_+ \otimes E_-\otimes E_+$ and on $E_-\otimes E_+ \otimes E_-$, since 
$S_EE_\pm \subset E_\pm$ and $S_{E_\pm}E_\mp=0$. This proves (\ref{R:eq}). 
\end{proof}
In the following, we denote by $(d^{D^0}S)^\pm$ the restriction of $d^{D^0}S$ to an element 
\[(d^{D^0}S)^\pm\in \mathrm{Hom}(E_\pm \otimes E_{\mp}\otimes E_\pm,E_\pm)\cong \mathrm{Hom}(E_\pm \otimes E_{\mp}, \mathrm{End}\, E_\pm).\] 

\begin{lem}\label{improved:lem}We have $R_{D}^\pm = R_{D^0}^\pm + (d^{D^0}S)^\pm$ and 
\[ (d^{D^0}S)^\pm (u,v)w  =  -(D^0_vS)_uw,\]
for all $(u,v,w)\in E_\pm \times E_{\mp}\times E_\pm$. 
\end{lem}
\begin{proof}The first formula is just (\ref{R:eq}). Since $D^0E_\pm \subset E_\pm$ and 
$S_{E_\pm}E_\mp=0$, we have 
\[ (d^{D^0}S)^\pm (u,v)w  = -D^0_{v}(S_{u})w -S_{{[}u,v{]}_H}w= -(D^0_{v}S)_{u}w -S_{D^0_vu}w -S_{{[}u,v{]}_H}w.\]
Using that $D^0$ is torsion-free we can write $[u,v]_H= D^0_uv-D^0_vu$, since $(D^0u)^*v=0$
for all $(u,v)\in E_\pm \times E_{\mp}$. Hence 
\[ -S_{D^0_vu}w -S_{{[}u,v{]}_H}w= -S_{D^0_uv}w=0,\]
again because $D^0E_\pm \subset E_\pm$ and $S_{E_\pm}E_\mp=0$. This proves the lemma. 
\end{proof}
\begin{prop}\label{genRic:prop} Let $ \delta $ be a divergence operator on $ E $ and $ S\in\mathfrak{so}(E)^{\langle 1\rangle} $ such that the Levi-Civita generalized connection $ D^0+S $ has divergence $ \delta $. Then the Ricci curvatures $Ric^\pm_\delta$ of a generalized pseudo-Riemannian Lie group 
$(G,H,\mathcal G_g,\delta)$ with arbitrary divergence $\delta\in E^*$ are related to the 
Ricci curvatures $Ric^\pm_0$ of $(G,H,\mathcal G_g,0)$ by 
\begin{equation}\label{Ric:eq} Ric^\pm_\delta = Ric^\pm_0 + \tr_{E_\pm}(d^{D^0}S)^\pm = Ric^\pm_0-D^0\delta|_{E_\mp \otimes E_\pm} ,\end{equation}
where 
 \[ (\tr_{E_+}\alpha)(e_i,e_b) = \tr (u\mapsto \alpha (u,e_i)e_b),\]
 for any $\alpha \in E_+^*\otimes E_-^*\otimes E_+^*\otimes E_+$ and, similarly, 
 \[ (\tr_{E_-}\beta)(e_a,e_j) = \tr (u\mapsto \beta(u,e_a)e_j),\] 
 when $\beta \in E_-^*\otimes E_+^*\otimes E_-^*\otimes E_-$. Here we are assuming the usual index ranges 
 for $a,b$ and $i,j$. 
\end{prop}
\begin{proof} An element $ S $ of the first generalized prolongation of $\mathfrak{so}(E)$ such that $ D^0+S $ has divergence $ \delta $ exists due to Proposition~\ref{existence of connection with divergence:prop}. The first equation follows from Lemma~\ref{curv:lem} by taking traces. The 
formula \[\tr_{E_\pm}(d^{D^0}S)^\pm = -D^0\delta|_{E_\mp \otimes E_\pm}\] is a consequence of
Lemma~\ref{improved:lem}, since the trace maps $\tr_{E_+}$ and $\tr_{E_-}$ are parallel for any metric generalized connection. 
In fact, for instance, 
\[ \tr_{E_+}(d^{D^0}S)^+(e_i,e_b) = -\tr_{E_+} \left((D^0_{e_i}S)e_b\right)= -D^0_{e_i}\left(\tr_{E_+} S\right) e_b=-(D^0_{e_i}\delta) e_b,\]
where the $\tr_{E_+} S\in E^*$ is defined by $(\tr_{E_+} S)v \coloneqq \tr_{E_+}(Sv)= \tr (E_+\ni u \mapsto S_uv\in E_+)$ for all $v\in E$ and we 
have used that   $\tr_{E_+}(Sv)=\tr (Sv)=\delta (v)$ for all $v\in E_+$.
\end{proof}
Summarizing we obtain the following theorem. 
\begin{thm}\label{Ric:thm}The components $R_{ia}^\delta=Ric^+_\delta (e_i,e_a)$ and $R_{ai}^\delta=Ric^-_\delta (e_a,e_i)$ of the Ricci curvature
tensors $Ric^\pm_\delta$ of a generalized pseudo-Riemannian Lie group 
$(G,H,\mathcal G_g,\delta)$ with arbitrary divergence $\delta\in E^*$ are given as follows:
\begin{eqnarray*}
R_{ia}^\delta&=& \mathcal{B}_{bi}^j \mathcal{B}_{aj}^b +\mathcal{B}_{ia}^c\delta_c,\\ 
R_{ai}^\delta &=& \mathcal{B}_{bi}^j \mathcal{B}_{aj}^b + \mathcal{B}_{ai}^j\delta_j.
\end{eqnarray*} 
In particular, the Ricci tensor $Ric_\delta = Ric_{\delta}^+ \oplus Ric_{\delta}^-$ is symmetric 
if and only if, $\delta$ satisfies the equation $\mathcal{B}_{ia}^c\delta_c = \mathcal{B}_{ai}^j\delta_j$. 
It is skew-symmetric if $(G,H,\mathcal G_g,0)$ is generalized Einstein and $\mathcal{B}_{ia}^c\delta_c = -\mathcal{B}_{ai}^j\delta_j$. (Recall that we are always assuming the usual index ranges for $a, b$ and $i,j$.)

In terms of the linear maps $\Gamma_{u_\pm}: E_\mp \rightarrow E_\pm$ defined in (\ref{Gamma:eq}) for $u_\pm \in E_\pm$ we have
\begin{eqnarray*} Ric_{\delta}^+ (u_-,u_+) &=&-\tr \left( \Gamma_{u_-}\circ \Gamma_{u_+}\right) + \delta (\mathrm{pr}_{E_+}[u_-,u_+]_H),\\
Ric_{\delta}^- (u_+,u_-) &=&-\tr \left( \Gamma_{u_-}\circ \Gamma_{u_+}\right) + \delta (\mathrm{pr}_{E_-}[u_+,u_-]_H).
\end{eqnarray*}
\end{thm}
The theorem shows that the Ricci curvature is completely determined by the one-form $\delta$ and the  
 coefficients $\mathcal{B}_{ajk}$ and $\mathcal{B}_{ibc}$ of the Dorfman bracket 
 in the orthonormal basis $(e_A)=(e_a,e_i)$. For future use we do now compute 
 the latter coefficients in terms of the coefficients of the Lie bracket (the structure constants) and the 
 coefficients of the three-form $H$ using (\ref{dorf_LA:eq}). Recall that $(v_a)$ was a $g$-orthonormal basis
 of $\mathfrak g$. More precisely, we have $g_{ab}=g(v_a,v_b) = \langle e_a,e_b\rangle = \eta_{ab}$. 
 We denote the corresponding structure constants of the Lie algebra $\mathfrak g$ by $\kappa_{ab}^c$, such that 
 \[ [v_a,v_b]= \kappa_{ab}^cv_c.\]
 Note that $\kappa_{abc} = \kappa_{ab}^dg_{dc}=\kappa_{ab}^d\eta_{dc}$ for $ \kappa_{abc}\coloneqq\langle[v_a,v_b],v_c\rangle $. 
 \begin{prop} \label{B:prop}The Dorfman coefficients $\mathcal{B}_{ajk}$, $\mathcal{B}_{ibc}$, $\mathcal{B}_{abc}$ and $\mathcal{B}_{ijk}$ ($a,b,c\in \{ 1,\ldots ,n\}$, 
 $i,j,k\in \{ n+1,\ldots ,2n\}$) are related to the 
 structure constant $\kappa_{abc}$ as follows: 
 \begin{eqnarray*}
 \mathcal{B}_{ajk} &=& \frac12\left( H_{aj'k'}-\kappa_{aj'k'}+\kappa_{j'k'a}-\kappa_{k'aj'}\right)\\
 \mathcal{B}_{ibc} &=& \frac12\left( H_{i'bc}+\kappa_{i'bc}-\kappa_{bci'}+\kappa_{ci'b}\right)\\
 \mathcal{B}_{abc} &=& \frac12\left( H_{abc}+(\partial\kappa)_{abc}\right)\\
  \mathcal{B}_{ijk} &=& \frac12\left( H_{i'j'k'}-(\partial\kappa)_{i'j'k'} \right),
 \end{eqnarray*}
 where $i'=i-n$, for $i\in \{ n+1,\ldots, 2n\}$ and $(\partial\kappa)_{abc}=\kappa_{abc}+\kappa_{bca}+\kappa_{cab}$. 
 \end{prop}
 \begin{proof}Using (\ref{dorf_LA:eq}) we compute
 \begin{eqnarray*}
  [e_a,e_j]_H &=& [v_a+gv_a,v_{j'}-gv_{j'}]_H =  [v_a,v_{j'}]_H -[v_a,gv_{j'}]_H+[gv_a,v_{j'}]_H\\
  &=& [v_a,v_{j'}] + H(v_a,v_{j'},\cdot ) +ad_{v_a}^*(gv_{j'}) -\iota_{v_{j'}}d(gv_a)\\
  &=&[v_a,v_{j'}] + H(v_a,v_{j'},\cdot ) +g(v_{j'},[v_a,\cdot ])+ g(v_a,[v_{j'}, \cdot ]).\\
 \end{eqnarray*}
 It follows that 
 \begin{eqnarray*}
  \mathcal{B}_{ajk} &=& \langle [e_a,e_j]_H,e_k\rangle =  \langle [e_a,e_j]_H,v_{k'}-gv_{k'}\rangle\\ 
  &=& \frac12 \left( H(v_a,v_{j'},v_{k'}) +g(v_{j'},[v_a,v_{k'} ])+ g(v_a,[v_{j'}, v_{k'} ])-g(v_{k'},[v_a,v_{j'}])\right)\\
  &=&\frac12 \left( H_{aj'k'} +\kappa_{ak'j'}+ \kappa_{j'k'a}-\kappa_{aj'k'}\right)\\
  &=&\frac12 \left( H_{aj'k'} -\kappa_{k'aj'}+ \kappa_{j'k'a}-\kappa_{aj'k'}\right).
 \end{eqnarray*}
 The proof of the second formula is similar, where now 
 \[ [e_i,e_b]_H = [v_{i'},v_b] + H(v_{i'},v_b,\cdot ) -g(v_b,[v_{i'},\cdot ]) -g(v_{i'},[v_b,\cdot ]).\]
 The remaining equations are obtained in the same way. 
 \end{proof}
 The next result shows that the underlying metric $g$ of an Einstein generalized pseudo-Riemannian Lie group 
 can be freely rescaled without changing the Einstein property, provided that the three-form and the divergence
 are appropriately rescaled. 
 \begin{prop} \label{rescalProp}Let $g$ be a left-invariant pseudo-Riemannian metric and $H$ a closed left-invariant three-form on a Lie group $G$. Consider $g'=\varepsilon \mu^{-2}g$ and $H'=\varepsilon \mu^{-2}H$, where $\varepsilon\in \{\pm 1\}$ and $\mu>0$. Then the generalized pseudo-Riemannian Lie group 
$(G, H,\mathcal G_g)$ is Einstein with divergence $\delta\in E^*$ if and only if $(G,H',\mathcal G_{g'})$ is Einstein with 
divergence $\delta' = \mu \delta$.  
 \end{prop}
 \begin{proof}
 Let $(v_a)$ be a $g$-orthonormal basis of $\mathfrak{g}$. Then $v_a' = \mu v_a$ defines a 
 $g'$-orthonormal basis $(v_a')$. The corresponding basis $(e_A')$ of $E$, where 
 $e_a'=v_a'+g'v_a'$ and $e_i = v_i-g'v_i'$,   is still 
 orthonormal with respect to the scalar product: $\langle e_A',e_B'\rangle = \varepsilon \langle e_A,e_B\rangle$. 
 The structure constants $\kappa_{abc}'\coloneqq g'([v_a',v_b'],v_c')$ with respect to the basis $(v_a')$ are 
 $\kappa_{abc}' = \varepsilon \mu \kappa_{abc}$. Similarly, $H'(v_a',v_b',v_c') = \varepsilon \mu H(v_a,v_b,v_c)$. 
 Finally, from these formulas and Proposition~\ref{B:prop} we see that 
 $\mathcal{B}_{ABC}' \coloneqq \langle [e_A',e_B']_{H'},e_C'\rangle =  \varepsilon\mu \mathcal{B}_{ABC}$. 
 Taking into account that $\langle e_A',e_B'\rangle = \varepsilon \langle e_A,e_B\rangle$, we conclude that 
 $(\mathcal{B}')_{AB}^C \coloneqq (\eta')^{CD}\mathcal{B}_{ABD}'= \mu\mathcal{B}_{AB}^C$. Now Theorem~\ref{Ric:thm}
 together with Proposition~\ref{LCdiv0:prop} shows that the coefficients of the Ricci curvatures $Ric$ of  $(G, H,\mathcal G_g,\delta)$ and 
 $Ric'$ of $(G,H',\mathcal G_{g'},\delta')$ are related by $Ric'(e_A',e_B') = \mu^2 Ric(e_A,e_B)$. 
 \end{proof}
 \begin{rem}\label{Christoffel:rm}
	Denote by $ \nabla $ the Levi-Civita connection of the pseudo-Riemannian metric $ g $ and define its coefficients with respect to the orthonormal frame $ (v_a) $ as $ \Gamma_{abc}\coloneqq g(\nabla_{v_a}v_b,v_c). $ Then \[ \Gamma_{abc}=\frac12\left(g([v_a,v_b],v_c)-g([v_b,v_c],v_a)+g([v_c,v_a],v_b) \right)=\frac12 (\kappa_{abc}-\kappa_{bca}+\kappa_{cab}) \] and hence the Dorfman coefficients $\mathcal{B}_{ajk}$ and $\mathcal{B}_{ibc}$ can be expressed by
	\begin{eqnarray*}
		\mathcal{B}_{ajk}&=&\frac12\left(H_{aj'k'}-2\Gamma_{aj'k'}\right)=\frac12H_{aj'k'}-\Gamma_{aj'k'}\\
		\mathcal{B}_{ibc}&=&\frac12\left(H_{i'bc}+2\Gamma_{i'bc}\right)=\frac12H_{i'bc}+\Gamma_{i'bc}.
	\end{eqnarray*}
\end{rem}
\begin{prop}\label{RicRicProp}Let $g$ be a left-invariant pseudo-Riemannian metric on a Lie group $G$. Consider the generalized pseudo-Riemannian Lie group $(G, H=0,\mathcal G_g)$.  Then the Ricci curvature $Ric^+ _0={Ric^\pm_\delta|}_{\delta=0}$ 
of the generalized metric $\mathcal G_g$ is related to the Ricci curvature $Ric^g$ of the metric $g$ by
\[ Ric^+_0(v-gv,u+gu) =  Ric^-_0(u+gu,v-gv)=Ric^g(u,v) + (\nabla_u \tau)(v),\quad u,v\in \mathfrak{g}, \]
where $\tau\in \mathfrak{g}^*$ is the trace-form defined by $\tau (v) = \tr \mathrm{ad}_{v}$.  
\end{prop}
 \begin{proof} The symmetry of the Ricci tensor of $\mathcal G_g$ follows from $\delta=0$. Therefore 
 it suffices to compute $R_{ia}=Ric^+(e_i,e_a)$ from Theorem~\ref{Ric:thm} and to compare with 
 $R^g_{ai'}=Ric^g(v_a,v_{i'})$, $i'=i-n$. Note first that, by Remark~\ref{Christoffel:rm} we have 
 \[ \mathcal{B}_{aj}^k = \Gamma_{aj'}^{k'},\quad \mathcal{B}_{ib}^c=\Gamma_{i'b}^c,\]
 since $H=0$ and $\langle e_j,e_k\rangle = -\langle e_{j'},e_{k'}\rangle = -g(v_{j'},v_{k'})$. 
 Hence, using Lemma~\ref{B_inv:lem} and the 
 fact that the Levi-Civita connection has zero torsion, we obtain 
 \[ R_{ia} = \mathcal{B}_{bi}^j \mathcal{B}_{aj}^b= -\Gamma_{bi'}^{j'}\Gamma_{j'a}^b=-\Gamma_{bi'}^{j'}(\Gamma_{aj'}^b+
 \kappa_{j'a}^b).\]
 On the other hand, we have 
 \[ R^g_{ai'}=\Gamma_{ai'}^d\Gamma_{fd}^f -\Gamma_{fi'}^d\Gamma_{ad}^f-\kappa_{fa}^d\Gamma_{di'}^f= 
 \Gamma_{ai'}^d\Gamma_{fd}^f+R_{ia}.\]
 To compute the first term we note that since the Levi-Civita connection is metric, we have 
 \[ \Gamma_{fd}^f = \kappa_{fd}^f = -\tau_d=-\tau (v_d)\]
 and, hence, 
 \[ \Gamma_{ai'}^d\Gamma_{fd}^f = -\Gamma_{ai'}^d\tau_d = (\nabla \tau)_{ai'}=(\nabla_{v_a}\tau)v_{i'}.\qedhere\]
 \end{proof}
 \begin{cor}\label{soliton:cor}
 	Let $g$ be a left-invariant pseudo-Riemannian metric on a Lie group $G$. Then the generalized pseudo-Riemannian Lie group $(G, H=0,\mathcal G_g)$ is Einstein with divergence $\delta=0$ if and only if $g$ 
 satisfies the following Ricci soliton equation
 \begin{equation}\label{soliton:Eq}
 	Ric^g+\nabla \tau =0,
 \end{equation}
 where $\tau$ is the trace-form. The form $\tau$ is always closed and, hence, 
 the solutions of the above equation are gradient Ricci solitons, if the first 
 Betti number of the manifold $G$ vanishes. 
 \end{cor}
 \begin{proof} For all $u,v\in \mathfrak{g}$ we have 
 \[ (d\tau )(u,v) = -\tau ([u,v]) = -\mathrm{tr}\, \mathrm{ad}_{[u,v]}= -\mathrm{tr}\, 
 [\mathrm{ad}_u,\mathrm{ad}_v] =0. \qedhere\]
  \end{proof}
 \begin{cor}\label{flat:cor}
 	Let $g$ be a left-invariant pseudo-Riemannian metric on a unimodular Lie group $G$. Then the generalized pseudo-Riemannian Lie group $(G, H=0,\mathcal G_g)$ is Einstein with divergence $\delta=0$ if and only if $g$ 
 is Ricci-flat. 
 \end{cor}

 \section{Classification results in dimension 3}\label{class:sec}
 \subsection{Preliminaries}
 \label{prelim:sec}
 Let $G$ be a three-dimensional Lie group endowed with 
 a left-invariant pseudo-Riemannian metric $g$ and an orientation. We will 
 identify $g$ with a non-degenerate symmetric bilinear form  $g\in \mathrm{Sym}^2\mathfrak{g}^*$. 
 We begin by showing that the Lie bracket can be encoded in an endomorphism $L$ of $\mathfrak{g}$ and study its
 properties. 
 
 Following Milnor \cite{M}, but allowing indefinite metrics, we denote by $L\in \mathrm{End}\, \mathfrak{g}$ the endomorphism such that 
 \begin{equation}\label{defL:eq} [u,v] = L(u\times v),\quad \forall u,v\in \mathfrak{g},\end{equation}
 where the cross-product $\times \in \bigwedge^2\mathfrak{g}^*\otimes \mathfrak{g}$ is defined by 
 \begin{equation} \label{Ldef:eq} g(u\times v,w) = \mathrm{vol}_g(u,v,w),\end{equation}
 using the metric volume form $\mathrm{vol}_g$. In terms of an oriented orthonormal basis
 $(v_a)$, we have 
 \[ v_a\times v_b = \varepsilon_c v_c,\quad  \varepsilon_c=g(v_c,v_c),\] 
 for every cyclic permutation $(a,b,c)$ of $\{ 1,2,3\}$. This implies 
 that 
 \begin{equation} \label{Liebracket:eq} [v_a,v_b] = \varepsilon_c Lv_c,\quad\forall\quad\mbox{cyclic}\quad (a,b,c)\in \mathfrak{S}_3.\end{equation}
We denote by $(L^a_{\phantom{a}b})$ the matrix of $L$ in the above basis,
\[ Le_b= L^a_{\phantom{a}b}e_a,\]
and by $L^{ab}=L^a_{\phantom{a}c}g^{cb}$ the coefficients of the corresponding 
tensor $L \circ g^{-1} \in \mathrm{Hom}(\mathfrak{g}^*,\mathfrak{g}) \cong \mathfrak{g}\otimes \mathfrak{g}$.

From (\ref{Liebracket:eq}), we see that the structure constants $\kappa_{ab}^c$ of $\mathfrak{g}$ 
with respect to the basis $(v_a)$ can be written as 
\[ \kappa_{ab}^c = \varepsilon_{abd}L^{cd},\]
where $\varepsilon_{abd}=\mathrm{vol}_g(v_a,v_b,v_d)$ (in particular, $\varepsilon_{123}=1$). 

The following lemma is a straightforward generalization of \cite[Lemma 4.1]{M}.
\begin{lem} The endomorphism $L$ is symmetric with respect to $g$ if and only if $\mathfrak{g}$ is 
unimodular.
\end{lem}
\begin{proof} Note first that $L$ is symmetric with respect to $g$ if and only if the matrix $(L^{ab})$ is symmetric.  
Therefore, the calculation 
\[ \tr \mathrm{ad}_{v_a} = \kappa_{ab}^b =  \varepsilon_{abc}L^{bc}\] 
shows that $L$ is symmetric if and only if  $\tr \mathrm{ad}_{v_a}=0$ for all $a$, i.e.\ if 
and only if $\mathfrak g$ is unimodular. 
\end{proof}
 \begin{prop}\label{NF:prop}Let $g$ be a non-degenerate symmetric bilinear form on an oriented 
 three-di\-men\-sion\-al unimodular Lie algebra $ \mathfrak{g} $. Then there exists an orthonormal basis
 $(v_a)$ of $(\mathfrak{g},g)$ such that $g(v_1,v_1)=g(v_2,v_2)$ and such that the symmetric endomorphism $L$ 
 defined in Equation (\ref{defL:eq}) is represented by one of the following matrices:
 \begin{eqnarray*}L_1(\alpha , \beta, \gamma )  &=& \left( \begin{array}{ccc}\alpha &0&0\\
 0&\beta& 0\\
 0&0&\gamma
 \end{array}\right),\quad 
 L_2(\alpha , \beta,\gamma )=\left( \begin{array}{ccc} \gamma &0&0\\
 0&\alpha &-\beta\\
 0&\beta&\alpha
 \end{array}\right),\\
 L_3(\alpha , \beta)&=&\left( \begin{array}{ccc}\beta&0&0\\0&\frac12+\alpha &\frac12\\
 0&-\frac12&-\frac12 +\alpha\\
 \end{array}\right),\quad 
 L_4(\alpha , \beta)= \left( \begin{array}{ccc} \beta&0&0\\0&-\frac12+\alpha &-\frac12\\
 0&\frac12&\frac12 +\alpha\\
\end{array}\right),\\
 L_5(\alpha) &=&\left( \begin{array}{ccc}\alpha&\frac{1}{\sqrt{2}}&0\\
 \frac{1}{\sqrt{2}}&\alpha& \frac{1}{\sqrt{2}}\\
 0&-\frac{1}{\sqrt{2}}&\alpha\end{array}\right), 
 \end{eqnarray*}
 where $\alpha, \beta ,\gamma\in \mathbb{R}$ and $g(v_3,v_3)=-g(v_2,v_2)$ for the normal forms 
 $L_2,\ldots , L_5$. 
 If $g$ is definite, then the orthonormal basis can be chosen such that $L$ is represented by a 
diagonal matrix $L_1(\alpha , \beta , \gamma )$ and each diagonal matrix is realized in this way. 
If $g$ is indefinite, then each of the above normal forms is realized by some unimodular Lie bracket.  
 \end{prop}
 \begin{proof}
 It is well known that every symmetric endomorphism on a Euclidean vector space 
 can be diagonalized. According to \cite[Lemma~2.2]{CEHL} and references therein, 
 for an indefinite scalar product on a three-dimensional vector space
 there are the five normal forms of a symmetric bilinear form, from which one easily obtains the five normal forms $L_1(\alpha , \beta, \gamma ), L_2(\alpha , \beta, \gamma ), L_3(\alpha , \beta),
 L_4(\alpha, \beta )$ and $L_5(\alpha )$ for a symmetric endomorphism. 
 It remains to check that for each of  these normal forms $(L^a_{\phantom{a}b})$, the 
 bracket with structure constants $\kappa_{ab}^c = \varepsilon_{abd}L^{cd}$ satisfies the Jacobi identity. 
 
 All the cases can be treated simultaneously by considering $(L^a_{\phantom{a}b})$ 
 of the form 
 \[ \left(\begin{array}{ccc}\alpha &\lambda &0\\
 \lambda &\beta &\mu \\
 0&\varepsilon_2\varepsilon_3\mu&\gamma\end{array}
 \right),  
 \]
 where $\lambda,\mu \in \mathbb{R}$. For the corresponding endomorphism $L$ we have
 \begin{eqnarray*} \mathrm{Jac}(v_1,v_2,v_3) &\coloneqq & [v_1,[v_2,v_3]] + [v_2,[v_3,v_1]]+[v_3,[v_1,v_2]]= \sum [v_a,\varepsilon_a Lv_a]\\ 
 &=& 
\varepsilon_1 \lambda [v_1,v_2]  +\varepsilon_2\lambda [v_2,v_1] +\varepsilon_3\mu [v_2,v_3]+\varepsilon_3\mu [v_3,v_2]=0,
 \end{eqnarray*}
 where we have used that $\varepsilon_1=\varepsilon_2$.
  \end{proof}
 \subsection{Classification in the case of zero divergence}
 \label{zerodiv:sec}
 \subsubsection{Unimodular Lie groups}
\begin{prop}\label{EinsteinNF:prop}
	If $(H,\mathcal G_g,\delta=0)$ is a divergence-free generalized Einstein structure on an oriented three-dimensional unimodular Lie group $G$, then there exists a $ g $-orthonormal basis $(v_a)$ of $\mathfrak g$ such that $g(v_1,v_1)=g(v_2,v_2)$ and such that the symmetric endomorphism $L$ defined in Equation (\ref{defL:eq}) is either of the form $ L_1(\alpha,\beta,\gamma) $, that is $ L $ is diagonalizable by an orthonormal basis, or of one of the forms $ L_3(0,0) $ or $ L_4(0,0) $. In the non-diagonalizable case the three-form $ H $ is zero. 
\end{prop}
\begin{proof}
	In the Euclidean case any symmetric endomorphism is always diagonalizable by an orthonormal basis. So we may assume that the scalar product is indefinite. By Proposition~\ref{NF:prop}, there is an orthonormal basis $ (v_a) $, such that the endomorphism $ L $ takes one of the normal forms $L_1(\alpha , \beta, \gamma ), L_2(\alpha , \beta, \gamma ), L_3(\alpha , \beta),
	L_4(\alpha, \beta )$ or $L_5(\alpha )$ from said proposition. As in the proof of Proposition~\ref{NF:prop}, we can treat all cases at once by considering the matrix 
	\[ \left(\begin{array}{ccc}\alpha &\lambda &0\\
		\lambda &\beta &\mu \\
		0&-\mu&\gamma\end{array}
	\right). 
	\]
	 Recall that we assume $ \varepsilon_{1}=\varepsilon_2=-\varepsilon_{3} $, where $ \varepsilon_a=g(v_a,v_a) $. Using  equation (\ref{Liebracket:eq}) we obtain the structure constants $ \kappa_{abc}=\varepsilon_c\kappa_{ab}^c $ of the Lie algebra in the following way. The bracket is given by
	\begin{equation*}\begin{split}
		\kappa_{12}^av_a&=[v_1,v_2]=\varepsilon_{3}Lv_3=\varepsilon_3\mu v_2+\varepsilon_3\gamma v_3=-\varepsilon_2\mu v_2+\varepsilon_3\gamma v_3\\
		\kappa_{23}^av_a&=[v_2,v_3]=\varepsilon_{1}Lv_1=\varepsilon_{1}\alpha v_1+\varepsilon_{1}\lambda v_2=\varepsilon_{1}\alpha v_1+\varepsilon_{2}\lambda v_2\\
		\kappa_{31}^av_a&=[v_3,v_1]=\varepsilon_{2}Lv_2=\varepsilon_{2}\lambda v_1+\varepsilon_{2}\beta v_2-\varepsilon_{2}\mu v_3=\varepsilon_{1}\lambda v_1+\varepsilon_{2}\beta v_2+\varepsilon_{3}\mu v_3,
	\end{split}\end{equation*}
	and hence 
	\[ \begin{array}{lll}	\kappa_{121}=0,&\kappa_{122}=\varepsilon_{2}\kappa_{12}^2=-\mu,&\kappa_{123}=\varepsilon_{3}\kappa_{12}^3=\gamma\\
		\kappa_{231}=\varepsilon_{1}\kappa_{23}^1=\alpha,&\kappa_{232}=\varepsilon_{2}\kappa_{23}^2=\lambda,&\kappa_{233}=0\\
		\kappa_{311}=\varepsilon_{1}\kappa_{31}^1=\lambda,&\kappa_{312}=\varepsilon_{2}\kappa_{31}^2=\beta,&\kappa_{313}=\varepsilon_{3}\kappa_{31}^3=\mu.\end{array} \]
	The remaining structure constants are determined by the skew-symmetry of $ \kappa_{abc} $ in the first two components. 
	
	By Proposition~\ref{B:prop}, the Dorfman coefficients are given as follows 
	\begin{eqnarray*}
		\mathcal{B}_{145}&=&\frac12\left(H_{112}-\kappa_{112}+\kappa_{121}-\kappa_{211}\right)=\kappa_{121}=0\\
		\mathcal{B}_{146}&=&\frac12\left(H_{113}-\kappa_{113}+\kappa_{131}-\kappa_{311}\right)=-\kappa_{311}=-\lambda\\
		\mathcal{B}_{156}&=&\frac12\left(H_{123}-\kappa_{123}+\kappa_{231}-\kappa_{312}\right)=\frac12\left(h-\gamma+\alpha-\beta\right)\\
		\mathcal{B}_{245}&=&\frac12\left(H_{212}-\kappa_{212}+\kappa_{122}-\kappa_{221}\right)=\kappa_{122}=-\mu\\
		\mathcal{B}_{246}&=&\frac12\left(H_{213}-\kappa_{213}+\kappa_{132}-\kappa_{321}\right)=\frac12\left(-h+\gamma-\beta+\alpha\right)\\
		\mathcal{B}_{256}&=&\frac12\left(H_{223}-\kappa_{223}+\kappa_{232}-\kappa_{322}\right)=\kappa_{232}=\lambda\\
		\mathcal{B}_{345}&=&\frac12\left(H_{312}-\kappa_{312}+\kappa_{123}-\kappa_{231}\right)=\frac12\left(h-\beta+\gamma-\alpha\right)\\
		\mathcal{B}_{346}&=&\frac12\left(H_{313}-\kappa_{313}+\kappa_{133}-\kappa_{331}\right)=-\kappa_{313}=-\mu\\
		\mathcal{B}_{356}&=&\frac12\left(H_{323}-\kappa_{323}+\kappa_{233}-\kappa_{332}\right)=\kappa_{233}=0\\
		\mathcal{B}_{412}&=&\frac12\left(H_{112}+\kappa_{112}-\kappa_{121}+\kappa_{211}\right)=-\kappa_{121}=0\\
		\mathcal{B}_{413}&=&\frac12\left(H_{113}+\kappa_{113}-\kappa_{131}+\kappa_{311}\right)=\kappa_{311}=\lambda\\
		\mathcal{B}_{423}&=&\frac12\left(H_{123}+\kappa_{123}-\kappa_{231}+\kappa_{312}\right)=\frac12\left(h+\gamma-\alpha+\beta\right)\\
		\mathcal{B}_{512}&=&\frac12\left(H_{212}+\kappa_{212}-\kappa_{122}+\kappa_{221}\right)=-\kappa_{122}=\mu\\
		\mathcal{B}_{513}&=&\frac12\left(H_{213}+\kappa_{213}-\kappa_{132}+\kappa_{321}\right)=\frac12\left(-h-\gamma+\beta-\alpha\right)\\
		\mathcal{B}_{523}&=&\frac12\left(H_{223}+\kappa_{223}-\kappa_{232}+\kappa_{322}\right)=-\kappa_{232}=-\lambda\\
		\mathcal{B}_{612}&=&\frac12\left(H_{312}+\kappa_{312}-\kappa_{123}+\kappa_{231}\right)=\frac12\left(h+\beta-\gamma+\alpha\right)\\
		\mathcal{B}_{613}&=&\frac12\left(H_{313}+\kappa_{313}-\kappa_{133}+\kappa_{331}\right)=\kappa_{313}=\mu\\
		\mathcal{B}_{623}&=&\frac12\left(H_{323}+\kappa_{323}-\kappa_{233}+\kappa_{332}\right)=-\kappa_{233}=0\\
	\end{eqnarray*}	
	Now, Theorem~\ref{Ric:thm} allows us to compute the Ricci curvature (for zero divergence $\delta$) with respect to the orthonormal basis $(e_A)=(e_a,e_i)$, $e_a = v_a + gv_a$, $e_i=v_{i'}-gv_{i'}$, of $E=\mathfrak g \oplus \mathfrak g^*$ as
	\begin{equation}\label{RicCoordinates:eq}
		R_{ia} = \sum_{j,b}\mathcal{B}_{bi}^j \mathcal{B}_{aj}^b =  \sum_{j,b} \mathcal{B}_{bij} \mathcal{B}_{ajb}(-\varepsilon_{j'})\varepsilon_b = \sum_{j,b} \mathcal{B}_{bij} \mathcal{B}_{jab}\varepsilon_{j'}\varepsilon_b
	\end{equation}
	where we have used that $\langle e_i,e_i\rangle = -\langle e_{i'},e_{i'}\rangle= -\varepsilon_{i'}$ and the standard index ranges $a,b\in \{ 1,2,3\}$, $i,j\in \{ 4,5,6\}$. 
	\begin{eqnarray*}
		R_{41}&=&\mathcal{B}_{245}\mathcal{B}_{512}\varepsilon_{2}\varepsilon_2+\mathcal{B}_{345}\mathcal{B}_{513}\varepsilon_{2}\varepsilon_3+\mathcal{B}_{246}\mathcal{B}_{612}\varepsilon_{3}\varepsilon_2+\mathcal{B}_{346}\mathcal{B}_{613}\varepsilon_{3}\varepsilon_3\\
		&=&\mathcal{B}_{245}\mathcal{B}_{512}-\mathcal{B}_{345}\mathcal{B}_{513}-\mathcal{B}_{246}\mathcal{B}_{612}+\mathcal{B}_{346}\mathcal{B}_{613}\\
		&=&-\mu^2-\frac14\left(h-\beta+\gamma-\alpha\right)\left(-h-\gamma+\beta-\alpha\right)\\
		&&\qquad-\frac14\left(-h+\gamma-\beta+\alpha\right)\left(h+\beta-\gamma+\alpha\right)-\mu^2\\
		&=&-2\mu^2-\frac14\left(\alpha^2-\left(h-\left(\beta-\gamma\right)\right)^2+\alpha^2-\left(h+\left(\beta-\gamma\right)\right)^2\right)\\
		&=&-2\mu^2-\frac12\alpha^2+\frac12h^2+\frac12(\beta-\gamma)^2\\
		R_{42}&=&\mathcal{B}_{145}\mathcal{B}_{521}\varepsilon_{2}\varepsilon_1+\mathcal{B}_{345}\mathcal{B}_{523}\varepsilon_{2}\varepsilon_3+\mathcal{B}_{146}\mathcal{B}_{621}\varepsilon_{3}\varepsilon_1+\mathcal{B}_{346}\mathcal{B}_{623}\varepsilon_{3}\varepsilon_3\\
		&=&0-\mathcal{B}_{345}\mathcal{B}_{523}+\mathcal{B}_{146}\mathcal{B}_{612}+0\\
		&=&\frac12\lambda\left(h-\beta+\gamma-\alpha\right)-\frac12\lambda\left(h+\beta-\gamma+\alpha\right)\\
		&=&-\lambda\left(\beta-\gamma+\alpha\right)\\
		R_{43}&=&\mathcal{B}_{145}\mathcal{B}_{531}\varepsilon_{2}\varepsilon_1+\mathcal{B}_{245}\mathcal{B}_{532}\varepsilon_{2}\varepsilon_2+\mathcal{B}_{146}\mathcal{B}_{631}\varepsilon_{3}\varepsilon_1+\mathcal{B}_{246}\mathcal{B}_{632}\varepsilon_{3}\varepsilon_2\\
		&=&0-\mathcal{B}_{245}\mathcal{B}_{523}+\mathcal{B}_{146}\mathcal{B}_{613}+0\\
		&=&-2\mu\lambda\\
		R_{51}&=&\mathcal{B}_{254}\mathcal{B}_{412}\varepsilon_{1}\varepsilon_2+\mathcal{B}_{354}\mathcal{B}_{413}\varepsilon_{1}\varepsilon_3+\mathcal{B}_{256}\mathcal{B}_{612}\varepsilon_{3}\varepsilon_2+\mathcal{B}_{356}\mathcal{B}_{613}\varepsilon_{3}\varepsilon_3\\
		&=&0+\mathcal{B}_{345}\mathcal{B}_{413}-\mathcal{B}_{256}\mathcal{B}_{612}+0\\
		&=&\frac12\lambda\left(h-\beta+\gamma-\alpha\right)-\frac12\lambda\left(h+\beta-\gamma+\alpha\right)\\
		&=&-\lambda\left(\beta-\gamma+\alpha\right)\\
		R_{52}&=&\mathcal{B}_{154}\mathcal{B}_{421}\varepsilon_{1}\varepsilon_1+\mathcal{B}_{354}\mathcal{B}_{423}\varepsilon_{1}\varepsilon_3+\mathcal{B}_{156}\mathcal{B}_{621}\varepsilon_{3}\varepsilon_1+\mathcal{B}_{356}\mathcal{B}_{623}\varepsilon_{3}\varepsilon_3\\
		&=&0+\mathcal{B}_{345}\mathcal{B}_{423}+\mathcal{B}_{156}\mathcal{B}_{612}+0\\
		&=&\frac14\left(h-\beta+\gamma-\alpha\right)\left(h+\gamma-\alpha+\beta\right)+\frac14\left(h-\gamma+\alpha-\beta\right)\left(h+\beta-\gamma+\alpha\right)\\
		&=&\frac14\left(\left(h+\left(\gamma-\alpha\right)\right)^2-\beta^2+\left(h-\left(\gamma-\alpha\right)\right)^2-\beta^2\right)\\
		&=&-\frac12\beta^2+\frac12h^2+\frac12\left(\gamma-\alpha\right)^2\\
		R_{53}&=&\mathcal{B}_{154}\mathcal{B}_{431}\varepsilon_{1}\varepsilon_1+\mathcal{B}_{254}\mathcal{B}_{432}\varepsilon_{1}\varepsilon_2+\mathcal{B}_{156}\mathcal{B}_{631}\varepsilon_{3}\varepsilon_1+\mathcal{B}_{256}\mathcal{B}_{632}\varepsilon_{3}\varepsilon_2\\
		&=&0+\mathcal{B}_{245}\mathcal{B}_{423}+\mathcal{B}_{156}\mathcal{B}_{613}+0\\
		&=&-\frac12\mu\left(h+\gamma-\alpha+\beta\right)+\frac12\mu\left(h-\gamma+\alpha-\beta\right)\\
		&=&-\mu\left(\gamma-\alpha+\beta\right)\\
		R_{61}&=&\mathcal{B}_{264}\mathcal{B}_{412}\varepsilon_{1}\varepsilon_2+\mathcal{B}_{364}\mathcal{B}_{413}\varepsilon_{1}\varepsilon_3+\mathcal{B}_{265}\mathcal{B}_{512}\varepsilon_{2}\varepsilon_2+\mathcal{B}_{365}\mathcal{B}_{513}\varepsilon_{2}\varepsilon_3\\
		&=&0+\mathcal{B}_{346}\mathcal{B}_{413}-\mathcal{B}_{256}\mathcal{B}_{512}+0\\
		&=&-2\lambda\mu\\
		R_{62}&=&\mathcal{B}_{164}\mathcal{B}_{421}\varepsilon_{1}\varepsilon_1+\mathcal{B}_{364}\mathcal{B}_{423}\varepsilon_{1}\varepsilon_3+\mathcal{B}_{165}\mathcal{B}_{521}\varepsilon_{2}\varepsilon_1+\mathcal{B}_{365}\mathcal{B}_{523}\varepsilon_{2}\varepsilon_3\\
		&=&0+\mathcal{B}_{346}\mathcal{B}_{423}+\mathcal{B}_{156}\mathcal{B}_{512}+0\\
		&=&-\frac12\mu\left(h+\gamma-\alpha+\beta\right)+\frac12\mu\left(h-\gamma+\alpha-\beta\right)\\
		&=&-\mu\left(\gamma-\alpha+\beta\right)\\
		R_{63}&=&\mathcal{B}_{164}\mathcal{B}_{431}\varepsilon_{1}\varepsilon_1+\mathcal{B}_{264}\mathcal{B}_{432}\varepsilon_{1}\varepsilon_2+\mathcal{B}_{165}\mathcal{B}_{531}\varepsilon_{2}\varepsilon_1+\mathcal{B}_{265}\mathcal{B}_{532}\varepsilon_{2}\varepsilon_2\\
		&=&\mathcal{B}_{146}\mathcal{B}_{413}+\mathcal{B}_{246}\mathcal{B}_{423}+\mathcal{B}_{156}\mathcal{B}_{513}+\mathcal{B}_{256}\mathcal{B}_{523}\\
		&=&-\lambda^2+\frac14\left(-h+\gamma-\beta+\alpha\right)\left(h+\gamma-\alpha+\beta\right)\\
		&&\qquad+\frac14\left(h-\gamma+\alpha-\beta\right)\left(-h-\gamma+\beta-\alpha\right)-\lambda^2\\
		&=&-2\lambda^2+\frac14\left(\gamma^2-\left(h+\left(\beta-\alpha\right)\right)^2+\gamma^2-\left(h-\left(\beta-\alpha\right)\right)^2\right)\\
		&=&-2\lambda^2+\frac12\gamma^2-\frac12h^2-\frac12\left(\beta-\alpha\right)^2
	\end{eqnarray*}
	We see that the Einstein equations yield a system of homogeneous quadratic equations in the real variables $ \alpha,\beta,\gamma,\lambda $ and $ \mu $. 
	
	The normal form $ L_5(\alpha) $ is excluded by equation $ R_{43}=0 $ for any $ \alpha\in\R $. 
	
	Equation $ R_{53}=0 $ for the normal form $ L_2(\alpha,\beta,\gamma) $ reads as \[ 0=\beta\left(2\alpha-\gamma\right). \] If $ \beta=0 $, then the matrix is diagonal, so assume that $ \gamma=2\alpha $. Then $ R_{52}=0 $ yields 
	\begin{eqnarray*}
		0&=&-\frac12\alpha^2+\frac12h^2+\frac12\left(\alpha-\gamma\right)^2\\
		&=&\frac12h^2-\alpha\gamma+\frac12\gamma^2\\
		&=&\frac12h^2
	\end{eqnarray*}
	and hence $ h=0 $. Therefore, equation $ R_{41}=0 $ is 
	\begin{equation*}
		0=-2\beta^2-\frac12\gamma^2,
	\end{equation*}
	which gives $ \beta=\gamma=0 $. Hence $ L $ is diagonalizable by an orthonormal basis. 
	
	If we consider the normal form $ L_3(\alpha,\beta) $, the equation $ R_{53}=0 $ is 
	\begin{equation*}
		0=-\frac12\left(\alpha-\frac12-\beta+\alpha+\frac12\right)=-\frac12\left(2\alpha-\beta\right)
	\end{equation*}
	and hence $ 2\alpha=\beta $. Now, the equation for $ R_{41} $ yields 
	\begin{eqnarray*}
		0&=&-2\left(\frac12\right)^2-\frac12\beta^2+\frac12h^2+\frac12\left(\alpha+\frac12-\alpha+\frac12\right)^2\\
		&=&-\frac12-\frac12\beta^2+\frac12h^2+\frac12\\
		&=&-\frac12\beta^2+\frac12h^2,
	\end{eqnarray*}
	hence $ \beta^2=h^2 $. Applying this to the equation $ R_{52}=0 $ gives 
	\begin{eqnarray*}
		0&=&-\frac12\left(\frac12+\alpha\right)^2+\frac12h^2+\frac12\left(\alpha-\frac12-\beta\right)^2\\
		&=&-\frac12\left(\frac12+\alpha\right)^2+\frac12h^2+\frac12\left(-\alpha-\frac12\right)^2\\
		&=&\frac12h^2.
	\end{eqnarray*}
	From that see $ h=0 $ and therefore $ \alpha=\beta=0 $. 
	
	A similar computation for $ L_4(\alpha,\beta) $, shows that the only possibility is $ L_4(0,0) $ with $ h=0 $.

\end{proof}
  \begin{thm}\label{Einstein:thm} Let $(H,\mathcal{G}_g)$ be a divergence-free generalized Einstein structure on an oriented  
  three-dimensional unimodular Lie group $G$. If the endomorphism $L\in \mathrm{End}\, \mathfrak g$ defined in (\ref{Ldef:eq}) is diagonalizable, then there exists an oriented $g$-orthonormal basis $(v_a)$ 
  of   $\mathfrak g = \mathrm{Lie}\, G$ and $\alpha_1,\alpha_2,\alpha_3, h \in \mathbb{R}$ such that 
  \begin{equation}\label{diag:eq} [v_a,v_b] = \alpha_c \varepsilon_cv_c,\quad \forall\quad\mbox{cyclic}\quad(a,b,c)\in \mathfrak{S}_3,\quad H=h\mathrm{vol}_g,\end{equation}
  where $\varepsilon_a=g(v_a,v_a)$ satisfies $\varepsilon_1= \varepsilon_2$. The constants 
  $(\alpha_1,\alpha_2,\alpha_3,h)$ can take the following values.
  \begin{enumerate}
  \item $\alpha_1=\alpha_2=\alpha_3 = \pm h$, in which case $\mathfrak g$ is either abelian and $g$ is flat (the case $h=0$) or 
  $\mathfrak g$ is isomorphic to  
  $\mathfrak{so}(2,1)$ or $\mathfrak{so}(3)$. The case $\mathfrak{so}(3)$ occurs precisely when $g$ is 
  definite (and $h\neq 0$). 
  \item There exists a cyclic 
permutation $\sigma\in \mathfrak{S}_3$ such that 
\[ \alpha_{\sigma (1)}=\alpha_{\sigma (2)} \neq0\quad\mbox{and}\quad h=\alpha_{\sigma (3)}=0.\]
In this case $g$ is flat and $[\mathfrak g, \mathfrak g]$ is abelian of dimension $2$, that is $\mathfrak g$ is metabelian. More precisely, $\mathfrak g$ is isomorphic to $\mathfrak{e}(2)$ ($g$ definite on $ [\mathfrak{g},\mathfrak{g}] $) or $\mathfrak{e}(1,1)$ ($g$ indefinite on $ [\mathfrak{g},\mathfrak{g}] $), where $\mathfrak{e}(p,q)$ denotes the Lie algebra of the 
isometry group of the pseudo-Euclidean space $\mathbb{R}^{p,q}$. 
  \end{enumerate}
	If the endomorphism is not diagonalizable ($ g $ is necessarily indefinite in this case), then $ h=0 $ and the Lie group $ G $ is isomorphic to the Heisenberg group. 
  \end{thm}
  \begin{proof}
  	Assume first that $ L $ is diagonalizable. Note, that the existence of $(\alpha_1,\alpha_2,\alpha_3,h)$ such that (\ref{diag:eq})
  is an immediate consequence of the diagonalizability of $L$. The corresponding 
  structure constants $\kappa_{abc}$ are given by 
\[ \kappa_{abc}= \alpha_c,\quad \forall\quad\mbox{cyclic}\quad(a,b,c)\in \mathfrak{S}_3.\]
In virtue of Proposition~\ref{B:prop}, this implies the following\footnote{The first two formulas are not needed for the proof. 
They are only included for future use.}. 
\begin{enumerate}
\item For all $a,b,c\in \{1,2,3\}$:   
\[ \mathcal{B}_{abc} = \frac12 (h+\alpha_1+\alpha_2+\alpha_3)\varepsilon_{abc},\]
where $\varepsilon_{abc} = \mathrm{vol}_g(v_a,v_b,v_c)$. 
\item For all $i,j,k \in \{4,5,6\}$: 
\[  \mathcal{B}_{ijk}=\frac12 (h-\alpha_1-\alpha_2-\alpha_3)\varepsilon_{i'j'k'},\]
where $i'=i-3$ for all $i\in \{4,5,6\}$. 
\item For $a\in \{1,2,3\}$ and $j,k\in \{4,5,6\}$ the coefficients 
\[ \mathcal{B}_{ajk}
= \frac12 (H_{aj'k'} -\kappa_{aj'k'} +\kappa_{j'k'a} -\kappa_{k'aj'})\] 
are given explicitly by 
\begin{eqnarray*} \mathcal{B}_{156}&=&-\mathcal{B}_{165} = \frac12 (h -\alpha_{3} +\alpha_{1} -\alpha_{2})=:\frac12 X_1\\
\mathcal{B}_{264}&=&-\mathcal{B}_{246}= \frac12 (h -\alpha_{1} +\alpha_{2} -\alpha_{3})=:\frac12 X_2\\
\mathcal{B}_{345}
&=& - \mathcal{B}_{354}= \frac12 (h -\alpha_{2} +\alpha_{3} -\alpha_{1})=:\frac12 X_3,
\end{eqnarray*}
with all other components equal to zero.
\item For $i\in \{ 4,5,6\}$ and $b,c\in \{ 1,2,3\}$ the coefficients 
\[ \mathcal{B}_{ibc}
= \frac12 (H_{i'bc} +\kappa_{i'bc} -\kappa_{bci'} +\kappa_{ci'b})\] 
are given explicitly by 
\begin{eqnarray*} \mathcal{B}_{423} &=& - \mathcal{B}_{432} = \frac12 (h +\alpha_{3} -\alpha_{1} +\alpha_{2}) =: 
\frac12 Y_1\\
\mathcal{B}_{531} &=& - \mathcal{B}_{513} = \frac12 (h+\alpha_{1} -\alpha_{2} +\alpha_{3}) =: \frac12 Y_2\\
\mathcal{B}_{612} &=& - \mathcal{B}_{621} = \frac12 (h+\alpha_{2} -\alpha_{3} +\alpha_{1}) =: \frac12 Y_3,
\end{eqnarray*}
with all other components equal to zero.
\end{enumerate}
From these formulas and equation (\ref{RicCoordinates:eq}), we can now compute the components
\[ R_{ia} = \sum_{j,b} \mathcal{B}_{bij} \mathcal{B}_{jab}\varepsilon_{j'}\varepsilon_b,
\]
of the Ricci curvature (for zero divergence $\delta=0$) with respect to the orthonormal basis $(e_A)=(e_a,e_i)$, $e_a = v_a + gv_a$, $e_i=v_{i'}-gv_{i'}$, of $E=\mathfrak g \oplus \mathfrak g^*$. Explicitly we obtain
\begin{eqnarray*} R_{41} &=& \mathcal{B}_{246}\mathcal{B}_{612}\varepsilon_3\varepsilon_2 + \mathcal{B}_{345}\mathcal{B}_{513}\varepsilon_2\varepsilon_3=-\frac{\varepsilon_2\varepsilon_3}{4}(X_2Y_3+X_3Y_2)\\
R_{52}&=&\mathcal{B}_{156} \mathcal{B}_{621}\varepsilon_{3}\varepsilon_1 + \mathcal{B}_{354} \mathcal{B}_{423}\varepsilon_{1}\varepsilon_3= -\frac{\varepsilon_1\varepsilon_3}{4}(X_1Y_3+X_3Y_1)\\
R_{63}&=& \mathcal{B}_{264} \mathcal{B}_{432}\varepsilon_{1}\varepsilon_2 + \mathcal{B}_{165} \mathcal{B}_{531}\varepsilon_{2}\varepsilon_1 =-\frac{\varepsilon_1\varepsilon_2}{4}(X_1Y_2+X_2Y_1),
\end{eqnarray*}
with all other components equal to zero. We conclude that the generalized Einstein equations reduce to 
a system of three homogeneous quadratic equations in the variables $X_a$ and $Y_a$:
\[ X_1Y_2+X_2Y_1=X_1Y_3+X_3Y_1=X_2Y_3+X_3Y_2=0.\]
A priori, we can distinguish four types of solutions depending on how many components of the vector $(X_1,X_2,X_3)$ 
are equal to zero: 0,1,2 or 3.
 
Solutions of type $0$: $X_1X_2X_3\neq 0$ implies $Y_1=Y_2=Y_3=0$ and finally 
\[ \alpha_1=\alpha_2=\alpha_3=-h\neq0.\] 
In this case the Lie algebra $\mathfrak g$ is isomorphic to $\mathfrak{so}(2,1)$ or $\mathfrak{so}(3)$. 
The latter case happens if and only if the metric $g$ is definite. 

Solutions of type $1$: assume for example that $X_1X_2\neq 0$, $X_3=0$.  
This implies that $Y_3=0$ and, hence, $\alpha_3=\alpha_1+\alpha_2$ and $h=0$. 
But then the equation $X_1Y_2+X_2Y_1=0$ reduces to $\alpha_1\alpha_2=0$,
which is inconsistent with $X_1X_2\neq 0$. This shows that solutions of type $1$
do not exist. 

Solutions of type $2$: assume for example that $X_1\neq 0$, $X_2=X_3=0$.  
This implies $Y_2=Y_3=0$ and finally $h=\alpha_1=0$, $\alpha_2=\alpha_3\neq0$.
So the solutions of type $2$ are of the following form. There exists a cyclic 
permutation $\sigma\in \mathfrak{S}_3$ such that 
\[ \alpha_{\sigma (1)}=\alpha_{\sigma (2)}\neq0\quad\mbox{and}\quad h=\alpha_{\sigma (3)}=0.\]
We conclude, for solutions of type 2,  that $g$ is flat (see Corollary~\ref{flat:cor}) and $\mathfrak g$ is metabelian. 
$[\mathfrak g,\mathfrak g]= \mathrm{span}\{ v_{\sigma (1)},v_{\sigma (2)}\}$ 
is two-dimensional and $\mathrm{ad}_{v_{\sigma (3)}}$ acts on it by a non-zero $g$-skew-symmetric 
endomorphism. This implies that $\mathfrak g$ is isomorphic to $\mathfrak{e}(2)$ or $\mathfrak{e}(1,1)$. 

Solutions of type $3$: assume $X_1=X_2=X_3=0$. This implies 
\[ \alpha_1=\alpha_2=\alpha_3=h.\] 
In this case $\mathfrak g$ is either abelian and $g$ is flat (the case $h=0$ again by Corollary \ref{flat:cor}) or
$\mathfrak g$ is isomorphic to $\mathfrak{so}(2,1)$ or $\mathfrak{so}(3)$, as for type $ 0 $.

If $ L $ is not diagonalizable, then $ g $ is indefinite and there exists an orthonormal basis $ (v_a)_a $ with $g(v_1,v_1)=g(v_2,v_2)=-g(v_3,v_3)$ such that $ L $ is either of the form $ L_3(0,0) $ or $ L_4(0,0) $, and $ h=0 $ by Proposition~\ref{EinsteinNF:prop}. We consider first the case $ L_3(0,0) $. To prove that $ G $ is isomorphic to the Heisenberg group, we show, using equation (\ref{Liebracket:eq}), that the generators $ P\coloneqq v_1, Q\coloneqq v_2+v_3 $ and $ R\coloneqq \varepsilon_2(v_3-v_2) $ of its Lie algebra $ \mathfrak{g} $ satisfy the relations $ [P, Q] = R $ and $[P, R] = [Q, R] = 0 $: 
\begin{eqnarray*}
	\left[P,Q\right]&=&[v_1,v_2+v_3]=[v_1,v_2]-[v_3,v_1]\\
	&=&\varepsilon_3Lv_3-\varepsilon_2Lv_2\\
	&=&\frac12\varepsilon_3v_2-\frac12\varepsilon_3v_3-\frac12\varepsilon_2v_2+\frac12\varepsilon_2v_3\\
	&=&-\frac12\varepsilon_2v_2+\frac12\varepsilon_2v_3-\frac12\varepsilon_2v_2+\frac12\varepsilon_2v_3\\
	&=&\varepsilon_2(v_3-v_2)\\
	&=&R\\
	\left[P,R\right]&=&[v_1,\varepsilon_2(v_3-v_2)]=-\varepsilon_2[v_3,v_1]-\varepsilon_2[v_1,v_2]\\
	&=&-\varepsilon_2\varepsilon_2Lv_2-\varepsilon_2\varepsilon_3Lv_3\\
	&=&-Lv_2+Lv_3\\
	&=&-\frac12v_2+\frac12v_3+\frac12v_2-\frac12v_3\\
	&=&0\\
	\left[Q,R\right]&=&[v_2+v_3,\varepsilon_2(v_3-v_2)]\\
	&=&\varepsilon_2[v_2,v_3]-\varepsilon_2[v_3,v_2]\\
	&=&2\varepsilon_2[v_2,v_3]\\
	&=&2\varepsilon_2\varepsilon_1Lv_1\\
	&=&0.
\end{eqnarray*}
In the case that $ L $ takes the form $ L_4(0,0) $, we see analogously that the generators  $ P=v_1, Q=v_2+v_3 $ and $ R=\varepsilon_2(v_2-v_3) $ satisfy the relations $ [P, Q] = R $ and $[P, R] = [Q, R] = 0 $.
  \end{proof}
\subsubsection{Non-unimodular Lie groups}
We assume now that the Lie group $ G $ is not unimodular. Let $ \mathfrak{u}\coloneqq \{x\in \mathfrak{g}\mid\tr \mathrm{ad}_x=0\} $ be the \emph{unimodular kernel} of $ \mathfrak{g} $. It can be easily checked that $ \mathfrak{u} $ is a two-dimensional abelian ideal of $ \mathfrak{g} $, containing the commutator ideal $ [\mathfrak{g},\mathfrak{g}] $. This means that the Lie algebra $ \mathfrak{g} $ is a semidirect product of $ \R $ and $ \R^2 $, where $ \R $ is acting on $ \R^2 $ by an endomorphism with non-zero trace. For details on the classification of non-unimodular, three-dimensional Lie algebras in terms of the Jordan normal form of this endomorphism we refer to \cite[Ch.\ 7, Theorem 1.4]{GOV}.

 We first treat the case that the restriction $ g|_{\mathfrak{u}\times\mathfrak{u}} $ of the metric $ g $ to $ \mathfrak{u} $ is non-degenerate. 
\begin{prop}\label{non-unimodular,non-degenerate:prop}
	Let $(H,\mathcal G_g,\delta=0)$ be a divergence-free generalized Einstein structure on an three-dimensional non-unimodular Lie group $G$. Let $ \mathfrak{u} $ be the unimodular kernel of the Lie algebra $ \mathfrak{g} $ and assume that $ g|_{\mathfrak{u}\times\mathfrak{u}} $ is non-degenerate. Then $ H=0 $ and $ g $ is indefinite. Furthermore there exists an orthonormal basis $(v_a)$ of $(\mathfrak{g},g)$ such that $ v_1,v_3\in\mathfrak{u} $ and $g(v_1,v_1)=g(v_2,v_2)=-g(v_3,v_3)$ and a positive constant $ \theta>0 $ such that 
	\begin{eqnarray*}
		[v_1,v_3]&=&0\\
		\left[v_2,v_1\right]&=&\theta v_1-\theta v_3\\
		\left[v_2,v_3\right]&=&\theta v_1+\theta v_3.
	\end{eqnarray*}
\end{prop}
\begin{proof}
	A $ g $-orthonormal basis $ (v_a)_a $ of $ \mathfrak g $ such that $ v_1,v_3\in\mathfrak{u} $ 
	exists, because $ g|_{\mathfrak{u}\times\mathfrak{u}} $ is non-degenerate. 
	Since $ \mathfrak{u} $ is an abelian ideal, there are $ \lambda,\mu,\nu,\rho\in\R $ such that 
	\begin{eqnarray*}
		[v_3,v_1]&=&0\\
		\left[v_2,v_1\right]&=&\varepsilon_1\lambda v_1+\varepsilon_3\mu v_3\\
		\left[v_2,v_3\right]&=&\varepsilon_1\nu v_1+\varepsilon_{3}\rho v_3
	\end{eqnarray*}
	with $ 0\neq\tr \mathrm{ad}_{v_2}=\varepsilon_1\lambda +\varepsilon_{3}\rho $. Using $ \lambda=\kappa_{211},\mu=\kappa_{213},\nu=\kappa_{231} $ and $ \rho=\kappa_{233} $, we can compute the Dorfman coefficients 
	\begin{eqnarray*}
		\mathcal B_{145}&=&\frac12\left(H_{112}-\kappa_{112}+\kappa_{121}-\kappa_{211}\right)=-\kappa_{211}=-\lambda\\
		\mathcal B_{146}&=&\frac12\left(H_{113}-\kappa_{113}+\kappa_{131}-\kappa_{311}\right)=-\kappa_{311}=0\\
		\mathcal B_{156}&=&\frac12\left(H_{123}-\kappa_{123}+\kappa_{231}-\kappa_{312}\right)=\frac12\left(h+\kappa_{213}+\kappa_{231}\right)=\frac12\left(h+\mu+\nu\right)\\
		\mathcal B_{245}&=&\frac12\left(H_{212}-\kappa_{212}+\kappa_{122}-\kappa_{221}\right)=-\kappa_{212}=0\\
		\mathcal B_{246}&=&\frac12\left(H_{213}-\kappa_{213}+\kappa_{132}-\kappa_{321}\right)=\frac12\left(-h-\kappa_{213}+\kappa_{231}\right)=-\frac12\left(h+\mu-\nu\right)\\
		\mathcal B_{256}&=&\frac12\left(H_{223}-\kappa_{223}+\kappa_{232}-\kappa_{322}\right)=\kappa_{232}=0\\
		\mathcal B_{345}&=&\frac12\left(H_{312}-\kappa_{312}+\kappa_{123}-\kappa_{231}\right)=\frac12\left(h-\kappa_{213}-\kappa_{231}\right)=\frac12\left(h-\mu-\nu\right)\\
		\mathcal B_{346}&=&\frac12\left(H_{313}-\kappa_{313}+\kappa_{133}-\kappa_{331}\right)=-\kappa_{313}=0\\
		\mathcal B_{356}&=&\frac12\left(H_{323}-\kappa_{323}+\kappa_{233}-\kappa_{332}\right)=\kappa_{233}=\rho\\
		\mathcal B_{412}&=&\frac12\left(H_{112}+\kappa_{112}-\kappa_{121}+\kappa_{211}\right)=\kappa_{211}=\lambda\\
		\mathcal B_{413}&=&\frac12\left(H_{113}+\kappa_{113}-\kappa_{131}+\kappa_{311}\right)=\kappa_{311}=0\\
		\mathcal B_{423}&=&\frac12\left(H_{123}+\kappa_{123}-\kappa_{231}+\kappa_{312}\right)=\frac12\left(h-\kappa_{213}-\kappa_{231}\right)=\frac12\left(h-\mu-\nu\right)\\
		\mathcal B_{512}&=&\frac12\left(H_{212}+\kappa_{212}-\kappa_{122}+\kappa_{221}\right)=\kappa_{212}=0\\
		\mathcal B_{513}&=&\frac12\left(H_{213}+\kappa_{213}-\kappa_{132}+\kappa_{321}\right)=\frac12\left(-h+\kappa_{213}-\kappa_{231}\right)=-\frac12\left(h-\mu+\nu\right)\\
		\mathcal B_{523}&=&\frac12\left(H_{223}+\kappa_{223}-\kappa_{232}+\kappa_{322}\right)=-\kappa_{232}=0\\
		\mathcal B_{612}&=&\frac12\left(H_{312}+\kappa_{312}-\kappa_{123}+\kappa_{231}\right)=\frac12\left(h+\kappa_{213}+\kappa_{231}\right)=\frac12\left(h+\mu+\nu\right)\\
		\mathcal B_{613}&=&\frac12\left(H_{313}+\kappa_{313}-\kappa_{133}+\kappa_{331}\right)=\kappa_{313}=0\\
		\mathcal B_{623}&=&\frac12\left(H_{323}+\kappa_{323}-\kappa_{233}+\kappa_{332}\right)=-\kappa_{233}=-\rho.
	\end{eqnarray*}
	To prove that the case $\varepsilon_1=\varepsilon_3$ cannot occur, we compute using equation (\ref{RicCoordinates:eq})
	\begin{eqnarray*}
		R_{52}&=&\mathcal{B}_{154}\mathcal{B}_{421}\varepsilon_{1}\varepsilon_1+\mathcal{B}_{354}\mathcal{B}_{423}\varepsilon_{1}\varepsilon_3+\mathcal{B}_{156}\mathcal{B}_{621}\varepsilon_{3}\varepsilon_1+\mathcal{B}_{356}\mathcal{B}_{623}\varepsilon_{3}\varepsilon_3\\
		&=&\mathcal{B}_{145}\mathcal{B}_{412}-\mathcal{B}_{345}\mathcal{B}_{423}-\mathcal{B}_{156}\mathcal{B}_{612}+\mathcal{B}_{356}\mathcal{B}_{623}\\
		&=&-\lambda^2-\frac14\left(h-\mu-\nu\right)^2-\frac14\left(h+\mu+\nu\right)^2-\rho^2
	\end{eqnarray*}
	where we have used that $ \varepsilon_{1}=\varepsilon_{3} $. But this can only be zero if $ \lambda=\rho=0 $, which contradicts $ 0\neq\tr \mathrm{ad}_{v_2}=\varepsilon_1\lambda +\varepsilon_{3}\rho $. This proves that 
	$\varepsilon_1=-\varepsilon_3$. Hence, we can assume that the basis is chosen such that $ \varepsilon_{1}=\varepsilon_{2}=-\varepsilon_{3} $.
	
	In this case the components of the Ricci curvature are 
	\begin{eqnarray*}
		R_{41}&=&\mathcal{B}_{245}\mathcal{B}_{512}\varepsilon_{2}\varepsilon_2+\mathcal{B}_{345}\mathcal{B}_{513}\varepsilon_{2}\varepsilon_3+\mathcal{B}_{246}\mathcal{B}_{612}\varepsilon_{3}\varepsilon_2+\mathcal{B}_{346}\mathcal{B}_{613}\varepsilon_{3}\varepsilon_3\\
		&=&0-\mathcal{B}_{345}\mathcal{B}_{513}-\mathcal{B}_{246}\mathcal{B}_{612}+0\\
		&=&\frac14\left(h-\mu-\nu\right)\left(h-\mu+\nu\right)+\frac14\left(h+\mu-\nu\right)\left(h+\mu+\nu\right)\\
		&=&\frac14\left(\left(h-\mu\right)^2-\nu^2+\left(h+\mu\right)^2-\nu^2\right)\\
		&=&\frac12\left(h^2+\mu^2-\nu^2\right)\\
		R_{42}&=&\mathcal{B}_{145}\mathcal{B}_{521}\varepsilon_{2}\varepsilon_1+\mathcal{B}_{345}\mathcal{B}_{523}\varepsilon_{2}\varepsilon_3+\mathcal{B}_{146}\mathcal{B}_{621}\varepsilon_{3}\varepsilon_1+\mathcal{B}_{346}\mathcal{B}_{623}\varepsilon_{3}\varepsilon_3\\
		&=&0\\
		R_{43}&=&\mathcal{B}_{145}\mathcal{B}_{531}\varepsilon_{2}\varepsilon_1+\mathcal{B}_{245}\mathcal{B}_{532}\varepsilon_{2}\varepsilon_2+\mathcal{B}_{146}\mathcal{B}_{631}\varepsilon_{3}\varepsilon_1+\mathcal{B}_{246}\mathcal{B}_{632}\varepsilon_{3}\varepsilon_2\\
		&=&-\mathcal{B}_{145}\mathcal{B}_{513}+0+0+\mathcal{B}_{246}\mathcal{B}_{623}\\
		&=&-\frac12\lambda\left(h-\mu+\nu\right)+\frac12\rho\left(h+\mu-\nu\right)\\
		R_{51}&=&\mathcal{B}_{254}\mathcal{B}_{412}\varepsilon_{1}\varepsilon_2+\mathcal{B}_{354}\mathcal{B}_{413}\varepsilon_{1}\varepsilon_3+\mathcal{B}_{256}\mathcal{B}_{612}\varepsilon_{3}\varepsilon_2+\mathcal{B}_{356}\mathcal{B}_{613}\varepsilon_{3}\varepsilon_3\\
		&=&0\\
		R_{52}&=&\mathcal{B}_{154}\mathcal{B}_{421}\varepsilon_{1}\varepsilon_1+\mathcal{B}_{354}\mathcal{B}_{423}\varepsilon_{1}\varepsilon_3+\mathcal{B}_{156}\mathcal{B}_{621}\varepsilon_{3}\varepsilon_1+\mathcal{B}_{356}\mathcal{B}_{623}\varepsilon_{3}\varepsilon_3\\
		&=&\mathcal{B}_{145}\mathcal{B}_{412}+\mathcal{B}_{345}\mathcal{B}_{423}+\mathcal{B}_{156}\mathcal{B}_{612}+\mathcal{B}_{356}\mathcal{B}_{623}\\
		&=&-\lambda^2+\frac14\left(h-\mu-\nu\right)^2+\frac14\left(h+\mu+\nu\right)^2-\rho^2\\
		&=&-\lambda^2+\frac12h^2+\frac12\left(\mu+\nu\right)^2-\rho^2\\
		R_{53}&=&\mathcal{B}_{154}\mathcal{B}_{431}\varepsilon_{1}\varepsilon_1+\mathcal{B}_{254}\mathcal{B}_{432}\varepsilon_{1}\varepsilon_2+\mathcal{B}_{156}\mathcal{B}_{631}\varepsilon_{3}\varepsilon_1+\mathcal{B}_{256}\mathcal{B}_{632}\varepsilon_{3}\varepsilon_2\\
		&=&0\\
		R_{61}&=&\mathcal{B}_{264}\mathcal{B}_{412}\varepsilon_{1}\varepsilon_2+\mathcal{B}_{364}\mathcal{B}_{413}\varepsilon_{1}\varepsilon_3+\mathcal{B}_{265}\mathcal{B}_{512}\varepsilon_{2}\varepsilon_2+\mathcal{B}_{365}\mathcal{B}_{513}\varepsilon_{2}\varepsilon_3\\
		&=&-\mathcal{B}_{246}\mathcal{B}_{412}+0+0+\mathcal{B}_{356}\mathcal{B}_{513}\\
		&=&\frac12\lambda\left(h+\mu-\nu\right)-\frac12\rho\left(h-\mu+\nu\right)\\
		R_{62}&=&\mathcal{B}_{164}\mathcal{B}_{421}\varepsilon_{1}\varepsilon_1+\mathcal{B}_{364}\mathcal{B}_{423}\varepsilon_{1}\varepsilon_3+\mathcal{B}_{165}\mathcal{B}_{521}\varepsilon_{2}\varepsilon_1+\mathcal{B}_{365}\mathcal{B}_{523}\varepsilon_{2}\varepsilon_3\\
		&=&0\\
		R_{63}&=&\mathcal{B}_{164}\mathcal{B}_{431}\varepsilon_{1}\varepsilon_1+\mathcal{B}_{264}\mathcal{B}_{432}\varepsilon_{1}\varepsilon_2+\mathcal{B}_{165}\mathcal{B}_{531}\varepsilon_{2}\varepsilon_1+\mathcal{B}_{265}\mathcal{B}_{532}\varepsilon_{2}\varepsilon_2\\
		&=&0+\mathcal{B}_{246}\mathcal{B}_{423}+\mathcal{B}_{156}\mathcal{B}_{513}+0\\
		&=&-\frac14\left(h+\mu-\nu\right)\left(h-\mu-\nu\right)-\frac14\left(h+\mu+\nu\right)\left(h-\mu+\nu\right)\\
		&=&-\frac14\left(\left(h-\nu\right)^2-\mu^2+\left(h+\nu\right)^2-\mu^2\right)\\
		&=&-\frac12\left(h^2+\nu^2-\mu^2\right)
	\end{eqnarray*}
	Imposing the Einstein condition, we see from the equations $ R_{41}+R_{63}=0 $ and $ R_{41}-R_{63}=0 $, that $ h^2=0 $ and $ \mu^2=\nu^2 $. If $ \mu=-\nu $, then $ R_{52}=0 $ reads as $ 0=-\lambda^2-\rho^2 $, hence $ \lambda=\rho=0 $, which contradicts $ 0\neq\tr \mathrm{ad}_{v_2}=\varepsilon_1\lambda +\varepsilon_{3}\rho $. Therefore $ \mu=\nu $ and, from $ R_{52}=0 $, \[ 2\mu^2=\lambda^2+\rho^2. \] In particular $ \mu\neq0 $, due to $ 0\neq\tr \mathrm{ad}_{v_2}=\varepsilon_1\lambda +\varepsilon_{3}\rho $. Note now that $ \mu=\nu $ implies that the endomorphism $ M\in\End(\mathfrak{u}) $, defined as the restriction of $ \mathrm{ad}_{v_2} $ to $ \mathfrak{u} $, is symmetric. A simple consequence of \cite[Lemma~2.2]{CEHL} (compare Proposition~\ref{NF:prop}) is that there exists an orthonormal basis of $ \mathfrak{u} $ such that $ M $ is represented by one of the matrices 
	\begin{eqnarray*}
		M_1(\theta,\eta )  &=& \left( \begin{array}{cc}
			\theta& 0\\
			0&\eta
		\end{array}\right),\quad 
		M_2(\theta,\eta )=\left( \begin{array}{cc}
			\theta& -\eta\\
			\eta&\theta
		\end{array}\right),\\
		M_3(\theta)&=&\left( \begin{array}{cc}\frac12+\theta &\frac12\\
			-\frac12&-\frac12 +\theta
		\end{array}\right),\quad 
		M_4(\theta)= \left( \begin{array}{cc} -\frac12+\theta &-\frac12\\
			\frac12&\frac12 +\theta\\
		\end{array}\right)
	\end{eqnarray*}
	in this basis. We may assume that the basis $ v_1,v_3 $ of $ \mathfrak{u} $ is chosen such that $ M $ takes one of these normal forms with respect to $ v_1,v_3 $. We see that $ M_1(\theta,\eta) $ is excluded by the condition $ \mu\neq0 $. Applying $ 2\mu^2=\lambda^2+\rho^2 $ to the normal form $ M_3(\theta) $ yields \[ 2\left(\frac12\right)^2=\left(\frac12+\theta\right)^2+\left(-\frac12+\theta\right)^2=2\left(\frac12\right)^2+\theta^2. \] Hence $ \theta=0 $, which contradicts  $ \tr \mathrm{ad}_{v_2}\neq0 $. For the same reason $ M $ also cannot have the normal form $ M_4(\theta) $. In the remaining case $ M_2(\theta,\eta ) $ the equation $ 2\mu^2=\lambda^2+\rho^2 $ reads as $ 2\left(-\eta\right)^2=\theta^2+\left(-\theta\right)^2 $. Therefore $ \eta=\pm\theta $. Furthermore $ \eta\neq0 $ because $ \mu\neq0 $. Hence the only two normal forms are 
	\begin{equation*}
		M_2(\theta,\theta )=\left( \begin{array}{cc}
			\theta& -\theta\\
			\theta&\theta
		\end{array}\right),\quad
	M_2(\theta,-\theta )=\left( \begin{array}{cc}
		\theta& \theta\\
		-\theta&\theta
	\end{array}\right),\quad \theta\neq0.
	\end{equation*}
	Replacing $ v_1 $ by $ -v_1 $ (exchanging $ M_2(\theta,\theta ) $ with $ M_2(\theta,-\theta ) $) and $ v_2 $ by $ -v_2 $ (replacing $ \theta $ by $ -\theta $), if necessary, we obtain the claimed equations for $ \theta>0 $. 
	\end{proof}
\begin{rem}
\label{r3,1':rem}
	Note that, while all the occurring Lie algebras in the previous Proposition are non-isomorphic as metric Lie algebras, they are isomorphic as Lie algebras. They are a semidirect product of $ \R^2 $ and $ \R $, where $ \R $ acts on $ \R^2 $ by the endomorphism $ \mathrm{ad}_{v_2}|_\mathfrak{u} $, which has non-real and non-imaginary eigenvalues $ (1+i)\theta $ and $ (1-i)\theta $. This corresponds to the Lie algebra $ \mathfrak{r}'_{3,1}(\R) $ in the notation of \cite[Ch.\ 7, Theorem 1.4]{GOV}.
\end{rem}
\begin{prop}\label{non-unimodular,degenerate:prop}
	There is no divergence-free generalized Einstein structure $(H,\mathcal G_g,\delta=0)$ on a  three-dimensional non-unimodular Lie group $G$ such that $ g $ is degenerate on the unimodular kernel $ \mathfrak{u} $ of $ \mathfrak{g} $.
\end{prop}
\begin{proof}
	Note first that the metric $ g $ necessarily has to be indefinite. We define $ \varepsilon\coloneqq 1 $ if the signature of $ g $ is $ (2,1) $ and $ \varepsilon\coloneqq -1 $ if it is $ (1,2) $. Note that in both cases there is a 
	two-dimensional subspace of $\mathfrak{g}$ on which $\varepsilon g$ is positive definite. Taking the intersection with $\mathfrak{u}$ we obtain a one-dimensional subspace generated by a vector $w_1$ such that $g(w_1,w_1)=\varepsilon$. 
	Next we choose a generator $w_2$ of the kernel of $g|_{\mathfrak{u} \times \mathfrak{u}}$ and a 
	null vector $w_3$ orthogonal to $w_1$ such that $g(w_2,w_3)=\frac\varepsilon2$. Summarizing, we obtain a 
	basis $(w_a)$ of $\mathfrak{g}$ such that:  
	\begin{equation}\label{indef basis:eq}
		g(w_1,w_1)=\varepsilon,\, g(w_1,w_2)=g(w_1,w_3)=g(w_2,w_2)=g(w_3,w_3)=0,\,g(w_2,w_3)=\frac\varepsilon 2 
	\end{equation}	
	and $ w_1,w_2\in\mathfrak{u} $. Denote by $ \theta_{ab}^c $ the structure constants of $ \mathfrak{g} $ in the basis $(w_a)$, $ [w_a,w_b]=\theta_{ab}^cw_c $. Then
	\begin{eqnarray*}
		\left[w_1,w_2\right]&=&0\\
		\left[w_3,w_1\right]&=&\theta_{31}^1w_1+\theta_{31}^2w_2\\
		\left[w_3,w_2\right]&=&\theta_{32}^1w_1+\theta_{32}^2w_2
	\end{eqnarray*}
	with $ 0\neq\tr \mathrm{ad}_{w_3}=\theta_{31}^1+\theta_{32}^2. $ The basis $ v_1\coloneqq w_1,v_2\coloneqq w_2+w_3,v_3\coloneqq w_2-w_3 $ of $ \mathfrak{g} $ is orthonormal with respect to $ g $ satisfying $ g(v_1,v_1)=g(v_2,v_2)=-g(v_3,v_3) $. If we define $ \lambda\coloneqq -\varepsilon_1\theta_{31}^1,\mu\coloneqq -\varepsilon_2\frac12\theta_{31}^2, \nu\coloneqq -\varepsilon_12\theta_{32}^1 $ and $ \rho\coloneqq -\varepsilon_2\theta_{32}^2$, where $ \varepsilon_a=g(v_a,v_a) $, then 
	\begin{eqnarray*}
		\kappa_{12}^cv_c=\left[v_1,v_2\right]&=&\left[w_1,w_2+w_3\right]=-\left[w_3,w_1\right]\\
		&=&-\theta_{31}^1w_1-\theta_{31}^2w_2\\
		&=&-\theta_{31}^1v_1-\frac12\theta_{31}^2v_2-\frac12\theta_{31}^2v_3\\
		&=&\varepsilon_{1}\lambda v_1+\varepsilon_2\mu v_2-\varepsilon_3\mu v_3\\
		\kappa_{23}^cv_c=\left[v_2,v_3\right]&=&\left[w_2+w_3,w_2-w_3\right]=-2\left[w_3,w_2\right]\\
		&=&-2\theta_{32}^1w_1-2\theta_{32}^2w_2\\
		&=&-2\theta_{32}^1v_1-\theta_{32}^2v_2-\theta_{32}^2v_3\\
		&=&\varepsilon_{1}\nu v_1+\varepsilon_2\rho v_2-\varepsilon_3\rho v_3\\
		\kappa_{31}^cv_c=\left[v_3,v_1\right]&=&\left[w_2-w_3,w_1\right]=-\left[w_3,w_1\right]\\
		&=&-\theta_{31}^1w_1-\theta_{31}^2w_2\\
		&=&-\theta_{31}^1v_1-\frac12\theta_{31}^2v_2-\frac12\theta_{31}^2v_3\\
		&=&\varepsilon_{1}\lambda v_1+\varepsilon_2\mu v_2-\varepsilon_3\mu v_3
	\end{eqnarray*}
	with $ \lambda+\rho\neq0 $.
	Hence the structure constants $ \kappa_{abc} $ of $ \mathfrak{g} $ with respect to $ (v_a)_a $ are 
	\[ \begin{array}{lll}	\kappa_{121}=\lambda,&\kappa_{122}=\mu,&\kappa_{123}=-\mu\\
		\kappa_{231}=\nu,&\kappa_{232}=\rho,&\kappa_{233}=-\rho\\
		\kappa_{311}=\lambda,&\kappa_{312}=\mu,&\kappa_{313}=-\mu.\end{array} \]
	Now the Dorfman coefficients are 
	\begin{eqnarray*}
		\mathcal B_{145}&=&\frac12\left(H_{112}-\kappa_{112}+\kappa_{121}-\kappa_{211}\right)=\kappa_{121}=\lambda\\
		\mathcal B_{146}&=&\frac12\left(H_{113}-\kappa_{113}+\kappa_{131}-\kappa_{311}\right)=-\kappa_{311}=-\lambda\\
		\mathcal B_{156}&=&\frac12\left(H_{123}-\kappa_{123}+\kappa_{231}-\kappa_{312}\right)=\frac12\left(h+\mu+\nu-\mu\right)=\frac12\left(h+\nu\right)\\
		\mathcal B_{245}&=&\frac12\left(H_{212}-\kappa_{212}+\kappa_{122}-\kappa_{221}\right)=\kappa_{122}=\mu\\
		\mathcal B_{246}&=&\frac12\left(H_{213}-\kappa_{213}+\kappa_{132}-\kappa_{321}\right)=\frac12\left(-h-\mu-\mu+\nu\right)=-\frac12\left(h+2\mu-\nu\right)\\
		\mathcal B_{256}&=&\frac12\left(H_{223}-\kappa_{223}+\kappa_{232}-\kappa_{322}\right)=\kappa_{232}=\rho\\
		\mathcal B_{345}&=&\frac12\left(H_{312}-\kappa_{312}+\kappa_{123}-\kappa_{231}\right)=\frac12\left(h-\mu-\mu-\nu\right)=\frac12\left(h-2\mu-\nu\right)\\
		\mathcal B_{346}&=&\frac12\left(H_{313}-\kappa_{313}+\kappa_{133}-\kappa_{331}\right)=-\kappa_{313}=\mu\\
		\mathcal B_{356}&=&\frac12\left(H_{323}-\kappa_{323}+\kappa_{233}-\kappa_{332}\right)=\kappa_{233}=-\rho\\
		\mathcal B_{412}&=&\frac12\left(H_{112}+\kappa_{112}-\kappa_{121}+\kappa_{211}\right)=-\kappa_{121}=-\lambda\\
		\mathcal B_{413}&=&\frac12\left(H_{113}+\kappa_{113}-\kappa_{131}+\kappa_{311}\right)=\kappa_{311}=\lambda\\
		\mathcal B_{423}&=&\frac12\left(H_{123}+\kappa_{123}-\kappa_{231}+\kappa_{312}\right)=\frac12\left(h-\mu-\nu+\mu\right)=\frac12\left(h-\nu\right)\\
		\mathcal B_{512}&=&\frac12\left(H_{212}+\kappa_{212}-\kappa_{122}+\kappa_{221}\right)=-\kappa_{122}=-\mu\\
		\mathcal B_{513}&=&\frac12\left(H_{213}+\kappa_{213}-\kappa_{132}+\kappa_{321}\right)=\frac12\left(-h+\mu+\mu-\nu\right)=-\frac12\left(h-2\mu+\nu\right)\\
		\mathcal B_{523}&=&\frac12\left(H_{223}+\kappa_{223}-\kappa_{232}+\kappa_{322}\right)=-\kappa_{232}=-\rho\\
		\mathcal B_{612}&=&\frac12\left(H_{312}+\kappa_{312}-\kappa_{123}+\kappa_{231}\right)=\frac12\left(h+\mu+\mu+\nu\right)=\frac12\left(h+2\mu+\nu\right)\\
		\mathcal B_{613}&=&\frac12\left(H_{313}+\kappa_{313}-\kappa_{133}+\kappa_{331}\right)=\kappa_{313}=-\mu\\
		\mathcal B_{623}&=&\frac12\left(H_{323}+\kappa_{323}-\kappa_{233}+\kappa_{332}\right)=-\kappa_{233}=\rho.
	\end{eqnarray*}
	By equation (\ref{RicCoordinates:eq}) the components of the generalized Ricci curvature are
	\begin{eqnarray*}
		R_{41}&=&\mathcal{B}_{245}\mathcal{B}_{512}\varepsilon_{2}\varepsilon_2+\mathcal{B}_{345}\mathcal{B}_{513}\varepsilon_{2}\varepsilon_3+\mathcal{B}_{246}\mathcal{B}_{612}\varepsilon_{3}\varepsilon_2+\mathcal{B}_{346}\mathcal{B}_{613}\varepsilon_{3}\varepsilon_3\\
		&=&\mathcal{B}_{245}\mathcal{B}_{512}-\mathcal{B}_{345}\mathcal{B}_{513}-\mathcal{B}_{246}\mathcal{B}_{612}+\mathcal{B}_{346}\mathcal{B}_{613}\\
		&=&-\mu^2+\frac14\left(h-2\mu-\nu\right)\left(h-2\mu+\nu\right)+\frac14\left(h+2\mu-\nu\right)\left(h+2\mu-\nu\right)-\mu^2\\
		&=&-2\mu^2+\frac14\left(\left(h-2\mu\right)^2-\nu^2+\left(h+2\mu\right)^2-\nu^2\right)\\
		&=&-2\mu^2+\frac12h^2+\frac12\left(2\mu\right)^2-\frac12\nu^2\\
		&=&\frac12h^2-\frac12\nu^2\\
		R_{42}&=&\mathcal{B}_{145}\mathcal{B}_{521}\varepsilon_{2}\varepsilon_1+\mathcal{B}_{345}\mathcal{B}_{523}\varepsilon_{2}\varepsilon_3+\mathcal{B}_{146}\mathcal{B}_{621}\varepsilon_{3}\varepsilon_1+\mathcal{B}_{346}\mathcal{B}_{623}\varepsilon_{3}\varepsilon_3\\
		&=&-\mathcal{B}_{145}\mathcal{B}_{512}-\mathcal{B}_{345}\mathcal{B}_{523}+\mathcal{B}_{146}\mathcal{B}_{612}+\mathcal{B}_{346}\mathcal{B}_{623}\\
		&=&\lambda\mu+\frac12\rho\left(h-2\mu-\nu\right)-\frac12\lambda\left(h+2\mu+\nu\right)+\rho\mu\\
		&=&\frac12\rho\left(h-\nu\right)-\frac12\lambda\left(h+\nu\right)\\
		R_{43}&=&\mathcal{B}_{145}\mathcal{B}_{531}\varepsilon_{2}\varepsilon_1+\mathcal{B}_{245}\mathcal{B}_{532}\varepsilon_{2}\varepsilon_2+\mathcal{B}_{146}\mathcal{B}_{631}\varepsilon_{3}\varepsilon_1+\mathcal{B}_{246}\mathcal{B}_{632}\varepsilon_{3}\varepsilon_2\\
		&=&-\mathcal{B}_{145}\mathcal{B}_{513}-\mathcal{B}_{245}\mathcal{B}_{523}+\mathcal{B}_{146}\mathcal{B}_{613}+\mathcal{B}_{246}\mathcal{B}_{623}\\
		&=&\frac12\lambda\left(h-2\mu+\nu\right)+\mu\rho+\lambda\mu-\frac12\rho\left(h+2\mu-\nu\right)\\
		&=&\frac12\lambda\left(h+\nu\right)-\frac12\rho\left(h-\nu\right)\\
		R_{51}&=&\mathcal{B}_{254}\mathcal{B}_{412}\varepsilon_{1}\varepsilon_2+\mathcal{B}_{354}\mathcal{B}_{413}\varepsilon_{1}\varepsilon_3+\mathcal{B}_{256}\mathcal{B}_{612}\varepsilon_{3}\varepsilon_2+\mathcal{B}_{356}\mathcal{B}_{613}\varepsilon_{3}\varepsilon_3\\
		&=&-\mathcal{B}_{245}\mathcal{B}_{412}+\mathcal{B}_{345}\mathcal{B}_{413}-\mathcal{B}_{256}\mathcal{B}_{612}+\mathcal{B}_{356}\mathcal{B}_{613}\\
		&=&\mu\lambda+\frac12\lambda\left(h-2\mu-\nu\right)-\frac12\rho\left(h+2\mu+\nu\right)+\rho\mu\\
		&=&\frac12\left(h-\nu\right)-\frac12\left(h+\nu\right)\\
		R_{52}&=&\mathcal{B}_{154}\mathcal{B}_{421}\varepsilon_{1}\varepsilon_1+\mathcal{B}_{354}\mathcal{B}_{423}\varepsilon_{1}\varepsilon_3+\mathcal{B}_{156}\mathcal{B}_{621}\varepsilon_{3}\varepsilon_1+\mathcal{B}_{356}\mathcal{B}_{623}\varepsilon_{3}\varepsilon_3\\
		&=&\mathcal{B}_{145}\mathcal{B}_{412}+\mathcal{B}_{345}\mathcal{B}_{423}+\mathcal{B}_{156}\mathcal{B}_{612}+\mathcal{B}_{356}\mathcal{B}_{623}\\
		&=&-\lambda^2+\frac14\left(h-2\mu-\nu\right)\left(h-\nu\right)+\frac14\left(h+\nu\right)\left(h+2\mu+\nu\right)-\rho^2\\
		&=&-\lambda^2+\frac14\left(h-\nu\right)^2-\frac12\mu\left(h-\nu\right)+\frac14\left(h+\nu\right)^2+\frac12\mu\left(h+\nu\right)-\rho^2\\
		&=&-\lambda^2+\frac12h^2+\frac12\nu^2+\mu\nu-\rho^2\\
		R_{53}&=&\mathcal{B}_{154}\mathcal{B}_{431}\varepsilon_{1}\varepsilon_1+\mathcal{B}_{254}\mathcal{B}_{432}\varepsilon_{1}\varepsilon_2+\mathcal{B}_{156}\mathcal{B}_{631}\varepsilon_{3}\varepsilon_1+\mathcal{B}_{256}\mathcal{B}_{632}\varepsilon_{3}\varepsilon_2\\
		&=&\mathcal{B}_{145}\mathcal{B}_{413}+\mathcal{B}_{245}\mathcal{B}_{423}+\mathcal{B}_{156}\mathcal{B}_{613}+\mathcal{B}_{256}\mathcal{B}_{623}\\
		&=&\lambda^2+\frac12\mu\left(h-\nu\right)-\frac12\mu\left(h+\nu\right)+\rho^2\\
		&=&\lambda^2-\mu\nu+\rho^2\\
		R_{61}&=&\mathcal{B}_{264}\mathcal{B}_{412}\varepsilon_{1}\varepsilon_2+\mathcal{B}_{364}\mathcal{B}_{413}\varepsilon_{1}\varepsilon_3+\mathcal{B}_{265}\mathcal{B}_{512}\varepsilon_{2}\varepsilon_2+\mathcal{B}_{365}\mathcal{B}_{513}\varepsilon_{2}\varepsilon_3\\
		&=&-\mathcal{B}_{246}\mathcal{B}_{412}+\mathcal{B}_{346}\mathcal{B}_{413}-\mathcal{B}_{256}\mathcal{B}_{512}+\mathcal{B}_{356}\mathcal{B}_{513}\\
		&=&-\frac12\lambda\left(h+2\mu-\nu\right)+\mu\lambda+\mu\rho+\frac12\rho\left(h-2\mu+\nu\right)\\
		&=&-\frac12\lambda\left(h-\nu\right)+\frac12\rho\left(h+\nu\right)\\
		R_{62}&=&\mathcal{B}_{164}\mathcal{B}_{421}\varepsilon_{1}\varepsilon_1+\mathcal{B}_{364}\mathcal{B}_{423}\varepsilon_{1}\varepsilon_3+\mathcal{B}_{165}\mathcal{B}_{521}\varepsilon_{2}\varepsilon_1+\mathcal{B}_{365}\mathcal{B}_{523}\varepsilon_{2}\varepsilon_3\\
		&=&\mathcal{B}_{146}\mathcal{B}_{412}+\mathcal{B}_{346}\mathcal{B}_{423}+\mathcal{B}_{156}\mathcal{B}_{512}+\mathcal{B}_{356}\mathcal{B}_{523}\\
		&=&\lambda^2+\frac12\mu\left(h-\nu\right)-\frac12\mu\left(h+\nu\right)+\rho^2\\
		&=&\lambda^2-\mu\nu+\rho^2\\
		R_{63}&=&\mathcal{B}_{164}\mathcal{B}_{431}\varepsilon_{1}\varepsilon_1+\mathcal{B}_{264}\mathcal{B}_{432}\varepsilon_{1}\varepsilon_2+\mathcal{B}_{165}\mathcal{B}_{531}\varepsilon_{2}\varepsilon_1+\mathcal{B}_{265}\mathcal{B}_{532}\varepsilon_{2}\varepsilon_2\\
		&=&\mathcal{B}_{146}\mathcal{B}_{413}+\mathcal{B}_{246}\mathcal{B}_{423}+\mathcal{B}_{156}\mathcal{B}_{513}+\mathcal{B}_{256}\mathcal{B}_{523}\\
		&=&-\lambda^2-\frac14\left(h+2\mu-\nu\right)\left(h-\nu\right)-\frac14\left(h+\nu\right)\left(h-2\mu+\nu\right)-\rho^2\\
		&=&-\lambda^2-\frac14\left(h-\nu\right)^2-\frac12\mu\left(h-\nu\right)-\frac14\left(h+\nu\right)^2+\frac12\mu\left(h+\nu\right)-\rho^2\\
		&=&-\lambda^2-\frac12h^2-\frac12\nu^2+\mu\nu-\rho^2.
	\end{eqnarray*}
	If we impose the Einstein condition, we see that $ 0=R_{52}-R_{63}=h^2+\nu^2 $, hence $ h=\nu=0 $. Therefore the equation $ R_{63}=0 $ reads as $ 0=-\lambda^2-\rho^2 $. This implies $ \lambda=\rho=0 $, which is a contradiction to $ \lambda+\rho\neq0 $.
\end{proof}
We summarize this by the following theorem. 
\begin{thm}
\label{nonunimod_divzero:thm}
	Let $(H,\mathcal G_g,\delta=0)$ be a divergence-free generalized Einstein structure on an  three-dimensional non-unimodular Lie group $G$. Then $ H=0 $ and $ g $ is indefinite. Furthermore there exists an orthonormal basis $(v_a)$ of $(\mathfrak{g},g)$ such that $ v_1,v_3\in\mathfrak{u} $ and $g(v_1,v_1)=g(v_2,v_2)=-g(v_3,v_3)$ as well as a positive constant $ \theta>0 $ such that 
	\begin{eqnarray*}
		[v_1,v_3]&=&0\\
		\left[v_2,v_1\right]&=&\theta v_1-\theta v_3\\
		\left[v_2,v_3\right]&=&\theta v_1+\theta v_3.
	\end{eqnarray*}
	The metric $g$ is a Ricci-soliton which is not of constant curvature.
\end{thm}
\begin{proof} It remains to prove the last statement. The fact that $g$ is a Ricci soliton is a 
direct consequence of Corollary~\ref{soliton:cor}. To see that the metric is non-flat, it suffices to check that $\nabla \tau \neq 0$. Since $\tau = 2\theta v_2^*$, where $(v_a^*)$ denotes the dual basis, it suffices to compute $\nabla v_2$:
\[ g(\nabla_{v_1}v_2,v_1) = g([v_1,v_2],v_1) = -\theta \varepsilon_1\neq 0.\]
Similarly,  $\nabla_{v_2}v_2=0$ shows that $g$ is neither of non-zero constant curvature.
\end{proof}
\begin{cor}
	If the metric is definite there are no solutions to the Ricci soliton equation (\ref{soliton:Eq}) in the non-unimodular case. 
\end{cor}
\begin{rem}
	Note that in all our proofs in the unimodular and in the non-unimodular case we only used that the diagonal components $ R_{ii'} $ are zero. In particular, the Ricci tensor is zero, if $ R_{ii'}=0 $ for all $ i\in\{4,5,6\} $, in the divergence free case. 
\end{rem}
\subsection{Arbitrary divergence}
\label{arbdiv:sec}
Recall that $R_{ia}^\delta=Ric^+_\delta (e_i,e_a)$ and $R_{ai}^\delta=Ric^-_\delta (e_a,e_i)$ denote the components of the Ricci curvature tensors $Ric^\pm_\delta$ of a generalized pseudo-Riemannian Lie group $(G,H,\mathcal G_g,\delta)$ with arbitrary divergence $\delta\in E^*$. If $ \delta=0 $ we often write $ R_{ia}=R_{ia}^0 $ and $ R_{ai}=R_{ai}^0 $. By Theorem \ref{Ric:thm} we have 
\begin{eqnarray*}
	R_{ia}^\delta&=&R_{ia}+\sum_cB_{ia}^c\delta_c=R_{ia}+\sum_c\varepsilon_cB_{iac}\delta_c\\
	R_{ai}^\delta&=&R_{ai}+\sum_jB_{ai}^j\delta_j=R_{ia}-\sum_{j}\varepsilon_{j'}B_{aij}\delta_j.
\end{eqnarray*}
\subsubsection{Unimodular Lie groups}
\label{unimoddivnotzero:sec}
\begin{prop}\label{DivergenceNF:prop}
	If $(H,\mathcal G_g,\delta)$ is a generalized Einstein structure on an oriented three-dimensional unimodular Lie group $G$, then there exists a $ g $-orthonormal basis $(v_a)$ of $\mathfrak{g}$ such that $g(v_1,v_1)=g(v_2,v_2)$ and such that the symmetric endomorphism $L$ defined in equation (\ref{defL:eq}) takes one of the following forms
	\begin{enumerate}
		\item $ L_1(\alpha,\beta,\gamma) $, that is $ L $ is diagonalizable by an orthonormal basis;
		\item $ L_3(\alpha,0) $ or $ L_4(\alpha,0) $, in both cases $ -\varepsilon_1\frac12\delta_1=-\varepsilon_1\frac12\delta_4=\alpha $ as well as $ \delta_2=\delta_3 $ and $ \delta_5=\delta_6 $. If $ \alpha\ne0 $, then $ \delta_2=\delta_3=\delta_5=\delta_6=0 $;	
		\item $ L_5(0) $ with $ \delta_2=\delta_5=0 $ and $ \delta_1=\delta_3=\delta_4=\delta_6=-\varepsilon_1\sqrt2 $, where $ \delta_a=\delta(v_a+g(v_a)) $ and $ \delta_i=\delta(v_{i'}-g(v_{i'})) $.
	\end{enumerate}
	
%
	
	Furthermore, in the non-diagonalizable case the three-form $ H $ is always zero (see Proposition~\ref{NF:prop} for the notation of the normal forms of $ L $). 
\end{prop}
\begin{proof}
	Since in the Euclidean case any symmetric endomorphism is always diagonalizable by an orthonormal basis, we may assume that the scalar product is indefinite. By Proposition~\ref{NF:prop}, there is an orthonormal basis $ (v_a) $, such that the endomorphism $ L $ takes one of the normal forms $L_1(\alpha , \beta, \gamma ), L_2(\alpha , \beta, \gamma ), L_3(\alpha , \beta),
	L_4(\alpha, \beta )$ or $L_5(\alpha )$ from said Proposition. As in the proof of Proposition~\ref{EinsteinNF:prop}, we can treat all these cases at once by considering the matrix 
	\[ \left(\begin{array}{ccc}\alpha &\lambda &0\\
		\lambda &\beta &\mu \\
		0&-\mu&\gamma\end{array}
	\right). 
	\]
	Recall that we assume $ \varepsilon_{1}=\varepsilon_2=-\varepsilon_{3} $, where $ \varepsilon_a=g(v_a,v_a) $. Using the Dorfman coefficients and the coefficients of the Ricci curvature with divergence zero from the proof of Proposition~\ref{EinsteinNF:prop}, we can compute the components of the Ricci curvature with divergence $ \delta $ as
	\begin{eqnarray*}
		R_{41}^\delta&=&R_{41}+\varepsilon_2\mathcal{B}_{412}\delta_2+\varepsilon_3\mathcal{B}_{413}\delta_3\\
		&=&-2\mu^2-\frac12\alpha^2+\frac12h^2+\frac12(\beta-\gamma)^2+\varepsilon_3\lambda\delta_3\\
		R_{42}^\delta&=&R_{42}+\varepsilon_1\mathcal{B}_{421}\delta_1+\varepsilon_3\mathcal{B}_{423}\delta_3\\
		&=&-\lambda\left(\beta-\gamma+\alpha\right)+\varepsilon_3\frac12\left(h+\gamma-\alpha+\beta\right)\delta_3\\
		R_{43}^\delta&=&R_{43}+\varepsilon_1\mathcal{B}_{431}\delta_1+\varepsilon_2\mathcal{B}_{432}\delta_2\\
		&=&-2\mu\lambda-\varepsilon_{1}\lambda\delta_1-\varepsilon_{2}\frac12\left(h+\gamma-\alpha+\beta\right)\delta_2\\
		R_{51}^\delta&=&R_{51}+\varepsilon_2\mathcal{B}_{512}\delta_2+\varepsilon_3\mathcal{B}_{513}\delta_3\\
		&=&-\lambda\left(\beta-\gamma+\alpha\right)+\varepsilon_{2}\mu\delta_2+\varepsilon_{3}\frac12\left(-h-\gamma+\beta-\alpha\right)\delta_3\\
		R_{52}^\delta&=&R_{52}+\varepsilon_1\mathcal{B}_{521}\delta_1+\varepsilon_3\mathcal{B}_{523}\delta_3\\
		&=&-\frac12\beta^2+\frac12h^2+\frac12\left(\gamma-\alpha\right)^2-\varepsilon_{1}\mu\delta_1-\varepsilon_{3}\lambda\delta_3\\
		R_{53}^\delta&=&R_{53}+\varepsilon_1\mathcal{B}_{531}\delta_1+\varepsilon_2\mathcal{B}_{532}\delta_2\\
		&=&-\mu\left(\gamma-\alpha+\beta\right)-\varepsilon_1\frac12\left(-h-\gamma+\beta-\alpha\right)\delta_1+\varepsilon_{2}\lambda\delta_2\\
		R_{61}^\delta&=&R_{61}+\varepsilon_2\mathcal{B}_{612}\delta_2+\varepsilon_3\mathcal{B}_{613}\delta_3\\
		&=&-2\lambda\mu+\varepsilon_{2}\frac12\left(h+\beta-\gamma+\alpha\right)\delta_2+\varepsilon_3\mu\delta_3\\
		R_{62}^\delta&=&R_{62}+\varepsilon_1\mathcal{B}_{621}\delta_1+\varepsilon_3\mathcal{B}_{623}\delta_3\\
		&=&-\mu\left(\gamma-\alpha+\beta\right)-\varepsilon_1\frac12\left(h+\beta-\gamma+\alpha\right)\delta_1\\
		R_{63}^\delta&=&R_{63}+\varepsilon_1\mathcal{B}_{631}\delta_1+\varepsilon_2\mathcal{B}_{632}\delta_2\\
		&=&-2\lambda^2+\frac12\gamma^2-\frac12h^2-\frac12\left(\beta-\alpha\right)^2-\varepsilon_{1}\mu\delta_1\\
		R_{14}^\delta&=&R_{41}-\varepsilon_2\mathcal{B}_{145}\delta_5-\varepsilon_3\mathcal{B}_{146}\delta_6\\
		&=&-2\mu^2-\frac12\alpha^2+\frac12h^2+\frac12(\beta-\gamma)^2+\varepsilon_3\lambda\delta_6\\
		R_{24}^\delta&=&R_{42}-\varepsilon_2\mathcal{B}_{245}\delta_5-\varepsilon_3\mathcal{B}_{246}\delta_6\\
		&=&-\lambda\left(\beta-\gamma+\alpha\right)+\varepsilon_2\mu\delta_5-\varepsilon_3\frac12\left(-h+\gamma-\beta+\alpha\right)\delta_6\\
		R_{34}^\delta&=&R_{43}-\varepsilon_2\mathcal{B}_{345}\delta_5-\varepsilon_3\mathcal{B}_{346}\delta_6\\
		&=&-2\mu\lambda-\varepsilon_{2}\frac12\left(h-\beta+\gamma-\alpha\right)\delta_5+\varepsilon_{3}\mu\delta_6\\
		R_{15}^\delta&=&R_{51}-\varepsilon_1\mathcal{B}_{154}\delta_4-\varepsilon_3\mathcal{B}_{156}\delta_6\\
		&=&-\lambda\left(\beta-\gamma+\alpha\right)-\varepsilon_{3}\frac12\left(h-\gamma+\alpha-\beta\right)\delta_6\\
		R_{25}^\delta&=&R_{52}-\varepsilon_1\mathcal{B}_{254}\delta_4-\varepsilon_3\mathcal{B}_{256}\delta_6\\
		&=&-\frac12\beta^2+\frac12h^2+\frac12\left(\gamma-\alpha\right)^2-\varepsilon_{1}\mu\delta_4-\varepsilon_3\lambda\delta_6\\
		R_{35}^\delta&=&R_{53}-\varepsilon_1\mathcal{B}_{354}\delta_4-\varepsilon_3\mathcal{B}_{356}\delta_6\\
		&=&-\mu\left(\gamma-\alpha+\beta\right)+\varepsilon_{1}\frac12\left(h-\beta+\gamma-\alpha\right)\delta_4\\
		R_{16}^\delta&=&R_{61}-\varepsilon_1\mathcal{B}_{164}\delta_4-\varepsilon_2\mathcal{B}_{165}\delta_5\\
		&=&-2\lambda\mu-\varepsilon_{1}\lambda\delta_4+\varepsilon_2\frac12\left(h-\gamma+\alpha-\beta\right)\delta_5\\
		R_{26}^\delta&=&R_{62}-\varepsilon_1\mathcal{B}_{264}\delta_4-\varepsilon_2\mathcal{B}_{265}\delta_5\\
		&=&-\mu\left(\gamma-\alpha+\beta\right)+\varepsilon_1\frac12\left(-h+\gamma-\beta+\alpha\right)\delta_4+\varepsilon_2\lambda\delta_5\\
		R_{36}^\delta&=&R_{63}-\varepsilon_1\mathcal{B}_{364}\delta_4-\varepsilon_2\mathcal{B}_{365}\delta_5\\
		&=&-2\lambda^2+\frac12\gamma^2-\frac12h^2-\frac12\left(\beta-\alpha\right)^2-\varepsilon_{1}\mu\delta_4.
	\end{eqnarray*}
	For the normal form $ L_2(\alpha,\beta,\gamma) $ the equations for $ Ric^+_\delta $ read  
	\begin{eqnarray*}
		R_{41}^\delta&=&-2\beta^2-\frac12\gamma^2+\frac12h^2\\
		R_{42}^\delta&=&\varepsilon_3\frac12\left(h+2\alpha-\gamma\right)\delta_3\\
		R_{43}^\delta&=&-\varepsilon_2\frac12\left(h+2\alpha-\gamma\right)\delta_2\\
		R_{51}^\delta&=&-\varepsilon_2\beta\delta_2+\varepsilon_3\frac12\left(-h-\gamma\right)\delta_3\\
		R_{52}^\delta&=&-\frac12\alpha^2+\frac12h^2+\frac12\left(\alpha-\gamma\right)^2+\varepsilon_1\beta\delta_1\\
		R_{53}^\delta&=&\beta\left(2\alpha-\gamma\right)-\varepsilon_{1}\frac12\left(-h-\gamma\right)\delta_1\\
		R_{61}^\delta&=&\varepsilon_2\frac12\left(h+\gamma\right)\delta_2-\varepsilon_3\beta\delta_3\\
		R_{62}^\delta&=&\beta\left(2\alpha-\gamma\right)-\varepsilon_1\frac12\left(h+\gamma\right)\delta_1\\
		R_{63}^\delta&=&\frac12\alpha^2-\frac12h^2-\frac12\left(\alpha-\gamma\right)^2+\varepsilon_1\beta\delta_1.
	\end{eqnarray*}
	Imposing now the Einstein condition, we get $ 0=R_{53}^\delta+R_{62}^\delta=2\beta\left(2\alpha-\gamma\right) $. So either $ L $ is diagonalizable, if $ \beta=0 $, or $ 2\alpha=\gamma $. But then the equation $ 0=R_{52}^\delta-R_{63}^\delta $ is \[ 0=-\alpha^2+h^2+\left(\alpha-\gamma\right)^2=h^2, \] and hence $ h=0 $. Applying this to the equation for $ R_{41}^\delta $ yields $ 0=-2\beta^2-\frac12\gamma^2 $. Therefore $ \beta=0 $ and the endomorphism $ L $ is diagonalizable by an orthonormal basis. 
	
	If $ L $ takes the normal form $ L_3(\alpha,\beta) $, the components of the Ricci tensor are 
	\begin{eqnarray*}
		R_{41}^\delta&=&-\frac12\beta^2+\frac12h^2\\
		R_{42}^\delta&=&\varepsilon_3\frac12\left(h+2\alpha-\beta\right)\delta_3\\
		R_{43}^\delta&=&-\varepsilon_2\frac12\left(h+2\alpha-\beta\right)\delta_2\\
		R_{51}^\delta&=&\varepsilon_{2}\frac12\delta_2+\varepsilon_3\frac12\left(-h+1-\beta\right)\delta_3\\
		R_{52}^\delta&=&-\frac12\left(\frac12+\alpha\right)^2+\frac12h^2+\frac12\left(-\frac12+\alpha-\beta\right)^2-\varepsilon_1\frac12\delta_1\\
		R_{53}^\delta&=&-\frac12\left(2\alpha-\beta\right)-\varepsilon_{1}\frac12\left(-h+1-\beta\right)\delta_1\\
		R_{61}^\delta&=&\varepsilon_2\frac12\left(h+1+\beta\right)\delta_2+\varepsilon_3\frac12\delta_3\\
		R_{62}^\delta&=&-\frac12\left(2\alpha-\beta\right)-\varepsilon_1\frac12\left(h+1+\beta\right)\delta_1\\
		R_{63}^\delta&=&\frac12\left(-\frac12+\alpha\right)^2-\frac12h^2-\frac12\left(\frac12+\alpha-\beta\right)^2-\varepsilon_1\frac12\delta_1\\
		R_{14}^\delta&=&-\frac12\beta^2+\frac12h^2\\
		R_{24}^\delta&=&\varepsilon_2\frac12\delta_5-\varepsilon_3\frac12\left(-h-1+\beta\right)\delta_6\\
		R_{34}^\delta&=&-\varepsilon_2\frac12(h-1-\beta)\delta_5+\varepsilon_3\frac12\delta_6\\
		R_{15}^\delta&=&-\varepsilon_3\frac12\left(h-2\alpha+\beta\right)\delta_6\\
		R_{25}^\delta&=&-\frac12\left(\frac12+\alpha^2\right)^2+\frac12h^2+\frac12\left(-\frac12+\alpha-\beta\right)^2-\varepsilon_1\frac12\delta_4\\
		R_{35}^\delta&=&-\frac12\left(2\alpha-\beta\right)+\varepsilon_1\frac12\left(h-1-\beta\right)\delta_4\\
		R_{16}^\delta&=&\varepsilon_2\frac12\left(h-2\alpha+\beta\right)\delta_5\\
		R_{26}^\delta&=&-\frac12\left(2\alpha-\beta\right)+\varepsilon_1\frac12\left(-h-1+\beta\right)\delta_4\\
		R_{36}^\delta&=&\frac12\left(-\frac12+\alpha\right)^2-\frac12h^2-\frac12\left(\frac12+\alpha-\beta\right)^2-\varepsilon_1\frac12\delta_4.
	\end{eqnarray*}
	First, equation $ R_{63}^\delta-R_{36}^\delta=0 $ yields $ \delta_1=\delta_4 $. Furthermore, due to $ 0=R_{53}^\delta-R_{62}^\delta=\varepsilon_1\left(h+\beta\right)\delta_1 $ and $ 0=R_{35}^\delta-R_{26}^\delta=\varepsilon_1\left(h-\beta\right)\delta_4=\varepsilon_1\left(h-\beta\right)\delta_1 $, we have $ \varepsilon_1\beta\delta_1=0 $. If $ \delta_1=0 $, then we see from $ 0=R_{53}^\delta=-\frac12\left(2\alpha-\beta\right) $ that $ 2\alpha=\beta $. Then $ 0=R_{63}^\delta=\frac12\left(-\frac12+\alpha\right)^2-\frac12h^2-\frac12\left(\frac12-\alpha\right)^2=-\frac12h^2 $ and $ h=0 $. Equation $ R_{41}^\delta=0 $ shows $ \beta=0 $ and therefore $ \alpha=0 $. Furthermore, $ R_{51}^\delta=0 $ shows $ \delta_2=\delta_3 $ and $ R_{15}^\delta=0 $ shows $ \delta_5=\delta_6 $, because $ \varepsilon_2=-\varepsilon_3 $. If otherwise $ \beta=0 $, we see again from $ R_{41}^\delta=0 $ that $ h=0 $ and also $ \delta_2=\delta_3 $ and $ \delta_5=\delta_6 $, because of  $ R_{51}^\delta=0 $ and $ R_{15}^\delta=0 $, respectively. Finally, $ 0=\alpha\delta_2=\alpha\delta_3=\alpha\delta_2=\alpha\delta_3 $ due to $ 0=R_{42}^\delta=R_{43}^\delta=R_{15}^\delta=R_{16}^\delta $, as well as $ \alpha=-\varepsilon_1\frac12\delta_1=-\varepsilon_1\frac12\delta_4 $ due to $ R_{62}^\delta=R_{26}^\delta=0 $.
	
	In a similar way we obtain the same equations for the normal form $ L_4(\alpha, \beta) $. 
	
	Finally the equations for $ Ric^+_\delta $ for the normal form $ L_5(\alpha) $ are 
	\begin{eqnarray*}
		R_{41}^\delta&=&-1-\frac12\alpha^2+\frac12h^2+\varepsilon_3\frac{1}{\sqrt2}\delta_3\\
		R_{42}^\delta&=&-\frac{1}{\sqrt2}\alpha+\varepsilon_3\frac12\left(h+\alpha\right)\delta_3\\
		R_{43}^\delta&=&-1-\varepsilon_1\frac{1}{\sqrt2}\delta_1-\varepsilon_2\frac12\left(h+\alpha\right)\delta_2\\
		R_{51}^\delta&=&-\frac{1}{\sqrt2}\alpha+\varepsilon_2\frac{1}{\sqrt2}\delta_2+\varepsilon_3\frac12\left(-h-\alpha\right)\delta_3\\
		R_{52}^\delta&=&-\frac12\alpha^2+\frac12h^2-\varepsilon_{1}\frac{1}{\sqrt2}\delta_1-\varepsilon_3\frac{1}{\sqrt2}\delta_3\\
		R_{53}^\delta&=&-\frac{1}{\sqrt2}\alpha-\varepsilon_1\frac12\left(-h-\alpha\right)\delta_1+\varepsilon_2\frac{1}{\sqrt2}\delta_2\\
		R_{61}^\delta&=&-1+\varepsilon_2\frac12\left(h+\alpha\right)\delta_2+\varepsilon_3\frac{1}{\sqrt2}\delta_3\\
		R_{62}^\delta&=&-\frac{1}{\sqrt2}\alpha-\varepsilon_1\frac12\left(h+\alpha\right)\delta_1\\
		R_{63}^\delta&=&-1+\frac12\alpha^2-\frac12h^2-\varepsilon_1\frac{1}{\sqrt2}\delta_1.
	\end{eqnarray*}
	From $ 0=R_{41}^\delta-R_{52}^\delta-R_{63}^\delta=-\frac12\left(\alpha^2-h^2\right) $ we see $ \alpha^2=h^2 $. Therefore $ \varepsilon_3\delta_3=\sqrt2 $ by $ 0=R_{41}^\delta=-1+\varepsilon_3\frac{1}{\sqrt2}\delta_3 $ as well as $ \varepsilon_1\delta_1=-\sqrt2 $ by $ 0=R_{63}^\delta=-1-\varepsilon_1\frac{1}{\sqrt2}\delta_1 $. If now $ \alpha=-h $, then $ 0=R_{42}^\delta=-\frac{1}{\sqrt2}\alpha $ and $ \alpha=h=0 $. By $ 0=R_{53}^\delta=\varepsilon_2\frac{1}{\sqrt2}\delta_2 $ also $ \delta_2=0 $. If otherwise $ \alpha=h $, we have $ 0=R_{42}^\delta+R_{51}^\delta=-\sqrt2\alpha+\varepsilon_2\frac{1}{\sqrt2}\delta_2 $ and thus $ 2\alpha=\varepsilon_2\delta_2 $. But at the same time $ 0=R_{43}^\delta=-\varepsilon_2\alpha\delta_2 $. This is only possible, if $ \alpha=\delta_2=0 $. The equations for $ Ric^-_\delta $ are now 
	\begin{eqnarray*}
		R_{14}^\delta&=&-1+\varepsilon_3\frac{1}{\sqrt2}\delta_6\\
		R_{24}^\delta&=&\varepsilon_2\frac{1}{\sqrt2}\delta_5\\
		R_{34}^\delta&=&-1+\varepsilon_3\frac{1}{\sqrt2}\delta_6\\
		R_{15}^\delta&=&0\\
		R_{25}^\delta&=&-\varepsilon_1\frac{1}{\sqrt2}\delta_4-\varepsilon_3\frac{1}{\sqrt2}\delta_6\\
		R_{35}^\delta&=&0\\
		R_{16}^\delta&=&-1-\varepsilon_1\frac{1}{\sqrt2}\delta_4\\
		R_{26}^\delta&=&\varepsilon_2\frac{1}{\sqrt2}\delta_5\\
		R_{36}^\delta&=&-1-\varepsilon_{1}\frac{1}{\sqrt2}\delta_4. 
	\end{eqnarray*}
	This finally yields $ \delta_5=0 $ and $ \varepsilon_1\delta_4=-\varepsilon_3\delta_6=-\sqrt2 $.
\end{proof}
\begin{thm} 
\label{divnotzerounimod:thm}
	Let $(H,\mathcal{G}_g,\delta)$ be a generalized Einstein structure on an oriented three-dimensional unimodular Lie group $G$. If the endomorphism $L\in \mathrm{End}\, \mathfrak g$ defined in (\ref{Ldef:eq}) is diagonalizable, then there exists an oriented $g$-orthonormal basis $(v_a)$ of $\mathfrak g = \mathrm{Lie}\, G$ and $\alpha_1,\alpha_2,\alpha_3, h \in \mathbb{R}$ such that 
	\begin{equation*}
		[v_a,v_b] = \alpha_c \varepsilon_cv_c,\quad \forall\quad\mbox{cyclic}\quad(a,b,c)\in \mathfrak{S}_3,\quad H=h\mathrm{vol}_g,
	\end{equation*}
	where $\varepsilon_a=g(v_a,v_a)$ satisfies $\varepsilon_1= \varepsilon_2$. The constants $(\alpha_1,\alpha_2,\alpha_3,h)$ can take the following values.
	\begin{enumerate}
		\item $\alpha_1=\alpha_2=\alpha_3=h=0$, in which case $\mathfrak g$ is abelian. The divergence can take an arbitrary value in $ E^* $.
		\item $\alpha_1=\alpha_2=\alpha_3=\pm h\neq0$, and $\mathfrak g$ is isomorphic to $\mathfrak{so}(2,1)$ or $\mathfrak{so}(3)$. The case $\mathfrak{so}(3)$ occurs precisely when $g$ is definite. Furthermore $ \delta|_{E_\pm}=0 $.
		\item There exists a cyclic permutation $\sigma\in \mathfrak{S}_3$ such that such that 
		\[ \alpha_{\sigma (1)}=\alpha_{\sigma (2)}\neq0\quad\mbox{and}\quad h=\alpha_{\sigma (3)}=0.\]
		In this case $[\mathfrak g, \mathfrak g]$ is abelian of dimension $2$, that is $\mathfrak g$ is metabelian. More precisely, $\mathfrak g$ is isomorphic to $\mathfrak{e}(2)$ ($g$ definite  on $ [\mathfrak{g},\mathfrak{g}] $) or $\mathfrak{e}(1,1)$ ($g$ indefinite on $ [\mathfrak{g},\mathfrak{g}] $). The components of the divergence $ \delta $ satisfy $ \delta_{\sigma(1)}=\delta_{\sigma(2)}=\delta_{\sigma(1)+3}=\delta_{\sigma(2)+3}=0 $
	\end{enumerate}
	If $ L $ is not diagonalizable, then $ h=0 $.
	\begin{enumerate}
		\item If $ L $ takes the normal form $ L_3(0,0) $ or $ L_4(0,0) $, then the Lie algebra $ \mathfrak{g} $ is isomorphic to the Heisenberg algebra $ \mathfrak{heis} $. In this case $ \delta_1=\delta_4=0 $, $ \delta_2=\delta_3 $ and $ \delta_5=\delta_6 $.
		\item If $ L $ takes the normal form $ L_3(\alpha,0) $ or $ L_4(\alpha,0) $, $ \alpha\ne0 $, then $ \mathfrak{g} $ is isomorphic to $ \mathfrak{e}(1,1) $. In these cases $ -\varepsilon_1\frac12\delta_1=-\varepsilon_1\frac12\delta_4=\alpha $ as well as $ \delta_2=\delta_3=\delta_5=\delta_6=0 $.
		\item If $ L $ takes the normal form $ L_5(0) $, then $ \mathfrak{g} $ is isomorphic to $ \mathfrak{e}(1,1) $. In this case  $ \varepsilon_1\delta_1=-\varepsilon_3\delta_3=\varepsilon_1\delta_4=-\varepsilon_3\delta_6=-\sqrt2 $ and $ \delta_2=\delta_5=0 $.
	\end{enumerate}
%
\end{thm}
\begin{proof}
	Assume first $ L $ is diagonalizable. To compute the components of the Ricci curvature, we use the formulas for the Dorfman coefficients and the notation for variables $ X_a $ and $ Y_a $ from the proof of Theorem~\ref{Einstein:thm}. 
	\begin{eqnarray*}
		R_{41}^\delta&=&R_{41}+\varepsilon_2\mathcal{B}_{412}\delta_2+\varepsilon_3\mathcal{B}_{413}\delta_3\\
		&=&R_{41}\\
		R_{42}^\delta&=&R_{42}+\varepsilon_1\mathcal{B}_{421}\delta_1+\varepsilon_3\mathcal{B}_{423}\delta_3\\
		&=&\frac12\varepsilon_3\delta_3 Y_1\\
		R_{43}^\delta&=&R_{43}+\varepsilon_1\mathcal{B}_{431}\delta_1+\varepsilon_2\mathcal{B}_{432}\delta_2\\
		&=&-\frac12\varepsilon_2\delta_2 Y_1\\
		R_{51}^\delta&=&R_{51}+\varepsilon_2\mathcal{B}_{512}\delta_2+\varepsilon_3\mathcal{B}_{513}\delta_3\\
		&=&-\frac12\varepsilon_3\delta_3 Y_2\\
		R_{52}^\delta&=&R_{52}+\varepsilon_1\mathcal{B}_{521}\delta_1+\varepsilon_3\mathcal{B}_{523}\delta_3\\
		&=&R_{52}\\
		R_{53}^\delta&=&R_{53}+\varepsilon_1\mathcal{B}_{531}\delta_1+\varepsilon_2\mathcal{B}_{532}\delta_2\\
		&=&\frac12\varepsilon_1\delta_1 Y_2\\
		R_{61}^\delta&=&R_{61}+\varepsilon_2\mathcal{B}_{612}\delta_2+\varepsilon_3\mathcal{B}_{613}\delta_3\\
		&=&\frac12\varepsilon_2\delta_2 Y_3\\
		R_{62}^\delta&=&R_{62}+\varepsilon_1\mathcal{B}_{621}\delta_1+\varepsilon_3\mathcal{B}_{623}\delta_3\\
		&=&-\frac12\varepsilon_1\delta_1 Y_3\\
		R_{63}^\delta&=&R_{63}+\varepsilon_1\mathcal{B}_{631}\delta_1+\varepsilon_2\mathcal{B}_{632}\delta_2\\
		&=&R_{63}\\
		R_{14}^\delta&=&R_{41}-\varepsilon_2\mathcal{B}_{145}\delta_5-\varepsilon_3\mathcal{B}_{146}\delta_6\\
		&=&R_{41}\\
		R_{24}^\delta&=&R_{42}-\varepsilon_2\mathcal{B}_{245}\delta_5-\varepsilon_3\mathcal{B}_{246}\delta_6\\
		&=&\frac12\varepsilon_3\delta_6 X_2\\
		R_{34}^\delta&=&R_{43}-\varepsilon_2\mathcal{B}_{345}\delta_5-\varepsilon_3\mathcal{B}_{346}\delta_6\\
		&=&-\frac12\varepsilon_2\delta_5 X_3\\
		R_{15}^\delta&=&R_{51}-\varepsilon_1\mathcal{B}_{154}\delta_4-\varepsilon_3\mathcal{B}_{156}\delta_6\\
		&=&-\frac12\varepsilon_3\delta_6 X_1\\
		R_{25}^\delta&=&R_{52}-\varepsilon_1\mathcal{B}_{254}\delta_4-\varepsilon_3\mathcal{B}_{256}\delta_6\\
		&=&R_{52}\\
		R_{35}^\delta&=&R_{53}-\varepsilon_1\mathcal{B}_{354}\delta_4-\varepsilon_3\mathcal{B}_{356}\delta_6\\
		&=&\frac12\varepsilon_1\delta_4 X_3\\
		R_{16}^\delta&=&R_{61}-\varepsilon_1\mathcal{B}_{164}\delta_4-\varepsilon_2\mathcal{B}_{165}\delta_5\\
		&=&\frac12\varepsilon_2\delta_5 X_1\\
		R_{26}^\delta&=&R_{62}-\varepsilon_1\mathcal{B}_{264}\delta_4-\varepsilon_2\mathcal{B}_{265}\delta_5\\
		&=&-\frac12\varepsilon_1\delta_4 X_2\\
		R_{36}^\delta&=&R_{63}-\varepsilon_1\mathcal{B}_{364}\delta_4-\varepsilon_2\mathcal{B}_{365}\delta_5\\
		&=&R_{63}.
	\end{eqnarray*}
	Note that if $(H,\mathcal G_g,\delta)$ is a generalized Einstein structure, also $(H,\mathcal G_g, 0)$ is. Therefore, as in the proof of Theorem~\ref{Einstein:thm}, we can distinguish cases depending on how many components of the vector $(X_1,X_2,X_3)$ are equal to zero. 
	
	Solutions of type $0$: $X_1X_2X_3\neq 0$ implies $ \delta_4=\delta_5=\delta_6=0 $. Furthermore recall that $Y_1=Y_2=Y_3=0$ and \[ \alpha_1=\alpha_2=\alpha_3=-h\neq0.\] In this case the Lie algebra $\mathfrak g$ is isomorphic to $\mathfrak{so}(2,1)$ ($ g $ indefinite) or $\mathfrak{so}(3)$ ($ g $ definite). 
	
	We have seen that solutions of type $1$ do not exist.
	
	Solutions of type $2$: assume for example that $X_1\neq 0$, $X_2=X_3=0$. This implies $ \delta_5=\delta_6=0 $. Moreover we have seen that $Y_2=Y_3=0$, $h=\alpha_1=0$ and $\alpha_2=\alpha_3\neq0$. This shows $ Y_1\neq 0 $ and thus $ \delta_2=\delta_3=0 $. So the solutions of type $2$ are of the following form. There exists a cyclic permutation $\sigma\in \mathfrak{S}_3$ such that \[ \alpha_{\sigma (1)}=\alpha_{\sigma (2)} \neq0\quad\mbox{and}\quad h=\alpha_{\sigma (3)}=\delta_{\sigma(1)}=\delta_{\sigma(2)}=\delta_{\sigma(1)+3}=\delta_{\sigma(2)+3}=0.\]
	As in the divergence-free case, we conclude that $\mathfrak g$ is metabelian. The commutator ideal $[\mathfrak g,\mathfrak g]= \mathrm{span}\{ v_{\sigma (1)},v_{\sigma (2)}\}$ is two-dimensional and $\mathrm{ad}_{v_{\sigma (3)}}$ acts on it by a non-zero $g$-skew-symmetric endomorphism. This implies that $\mathfrak g$ is isomorphic to $\mathfrak{e}(2)$ or $\mathfrak{e}(1,1)$. 
	
	Solutions of type $3$: assume $X_1=X_2=X_3=0$. This implies 
	\[ \alpha_1=\alpha_2=\alpha_3=h.\] 
	If $ h=0 $, then $ Y_1=Y_2=Y_3=0 $ and $ \delta\in E^* $ arbitrary, and if $ h\neq0 $, then $ Y_1=Y_2=Y_3=2h\neq0 $ and therefore $ \delta_1=\delta_2=\delta_3=0 $. 
	
	By Proposition~\ref{DivergenceNF:prop}, if $ L $ is not diagonalizable, it is of the forms $ L_3(\alpha,0) $, $ L_4(\alpha,0) $ or $ L_5(0) $ and the divergence has the claimed properties. From Theorem~\ref{Einstein:thm} we know that $ G $ is the Heisenberg group if $ L $ takes the normal for $ L_3(0,0) $ or $ L_4(0,0) $. If $ \alpha\ne0 $,  $\mathrm{ad}_{v_1}$ acts on $[\mathfrak g,\mathfrak g]= \mathrm{span}\{ v_2,v_3\}$ by a symmetric endomorphism with eigenvalues $ \alpha $ and $ -\alpha $. Therefore $ \mathfrak{g}\cong\mathfrak{e}(1,1) $.
	
	If $ L $ takes the normal form $ L_5(0) $, one can show that the only unimodular Lie algebra whose Killing form has the same signature as the one of $ \mathfrak{g} $, is the Lie algebra $ \mathfrak{e}(1,1) $. Alternatively, one can check that $\mathrm{ad}_{v_1+v_3}$ acts on $ \mathrm{span}\{ v_2,v_1-v_3\} $ a symmetric endomorphism with eigenvalues $ \sqrt2 $ and $ -\sqrt2 $. Therefore again $ \mathfrak{g}\cong\mathfrak{e}(1,1) $.
\end{proof}
\begin{rem}
	Except for the cases that the endomorphism $ L $ takes the normal form $ L_3(\alpha,0)$, $L_4(\alpha,0) $ ($ \alpha\ne0 $) and $ L_5(0) $, the solutions are such that the Ricci tensor for zero divergence and the contribution of the divergence to the Ricci tensor vanish simultaneously. 
\end{rem}
\begin{cor}Let $(H,\mathcal{G}_g,\delta)$ be a generalized Einstein structure on an  three-dimensional Lie group $G$. 
Then the left-invariant metric defined by $g$ is bi-invariant if and only if $\mathfrak{g}$ is isomorphic to $\mathfrak{so}(3)$, 
$\mathfrak{so}(2,1)$ or $\mathbb{R}^3$. 
\end{cor}
\begin{proof} This follows from the fact that the only three-dimensional Lie algebras admitting 
an ad-invariant scalar product are the above three Lie algebras together with our classification 
of generalized Einstein structures on these Lie algebras. 
\end{proof}

\subsubsection{Non-unimodular Lie groups}
\label{unimoddivnotzero_notunimod:sec}
\begin{prop}
\label{divnotzero_u_nondeg:prop}
	Let $(H,\mathcal G_g,\delta)$ be a generalized Einstein structure on an three-dimensional non-unimodular Lie group $G$. Let $ \mathfrak{u} $ be the unimodular kernel of the Lie algebra $ \mathfrak{g} $ and assume that $ g|_{\mathfrak{u}\times\mathfrak{u}} $ is non-degenerate. Then there exists an orthonormal basis $(v_a)$ of $(\mathfrak{g},g)$ such that $ v_1,v_3\in\mathfrak{u} $ and $g(v_1,v_1)=g(v_2,v_2)=-g(v_3,v_3)$.  Furthermore $\delta_2=\delta_5$. If $ \delta_2=\delta_5=0 $, then $\delta=0$, $ h=0 $ and one can choose $ v_1 $ and $ v_3 $ such that there is a positive constant $ \theta>0 $ such that 
	\begin{eqnarray*}
		\left[v_2,v_1\right]&=&\theta v_1-\theta v_3\\
		\left[v_2,v_3\right]&=&\theta v_1+\theta v_3.
	\end{eqnarray*}
	If $ \delta_2=\delta_5\neq0 $, $ M\coloneqq \mathrm{ad}_{v_2}|_\mathfrak{u} $ is diagonalizable. We have $ h^2=\left(\tr M\right)^2\ne0 $ and $ \delta_2=\delta_5=-\tr M \ne0 $. In the special case that $ M $ has a double eigenvalue, it is diagonalizable by an orthonormal basis. That is, one can choose $ v_1 $ and $ v_3 $ such that there exists a positive constant $ \theta>0 $ such that
	\begin{eqnarray*}
		\left[v_2,v_1\right]&=&\theta v_1\\
		\left[v_2,v_3\right]&=&\theta v_3.
	\end{eqnarray*}
	In this case $ h^2=\left(2\theta\right)^2\ne0 $ and $ \delta_2=\delta_5=-2\theta\ne0 $. Furthermore $ \delta_1=\delta_3=\delta_4=\delta_6=0 $.
In the case  $ \delta_2=\delta_5\neq 0 $ and two distinct real eigenvalues of $M$ there are the following families of solutions of the 
generalized Einstein equation:
\begin{enumerate}
\item $h=\pm 2\lambda$,  $\delta_2=\delta_5= -2\varepsilon_2\lambda$ and 
$\delta_1=\delta_3=\delta_4=\delta_6=0$, 
\[
M=\left( 
\begin{array}{cc}
		\varepsilon_1\lambda &-\varepsilon_1\mu\\
		\varepsilon_3\mu&-\varepsilon_{3}\lambda
	\end{array}
	\right) 
	\]
where $\lambda,\mu \in \mathbb{R}\setminus \{0\}$ and $|\mu |\neq |\lambda|$. 
\item[2A.] $h=\mu -\nu$,  $\delta_2=\delta_5=\varepsilon_2(-\mu +\nu)$, $\delta_4=\delta_6=0$, $\delta_1$ and $\delta_3$ are related by  $\mu \delta_1-\nu \delta_3=0$
and 
\[
M=\left( 
\begin{array}{cc}
		\varepsilon_1\mu&\varepsilon_1\nu\\
		\varepsilon_3\mu&\varepsilon_{3}\nu
	\end{array}
	\right) 
	\]
where $\mu, \nu \in \mathbb{R}$ are such that $\mu -\nu\neq 0$. 
\item[2B.] $h=\mu -\nu$,  $\delta_2=\delta_5=\varepsilon_2(\mu -\nu)$, $\delta_4=\delta_6=0$, $\delta_1$ and $\delta_3$ are related by  $\mu \delta_1+\nu \delta_3=0$ and 
\[
M=\left( 
\begin{array}{cc}-\varepsilon_1\mu &\varepsilon_1\nu\\		
\varepsilon_3\mu&-\varepsilon_{3}\nu
\end{array}
	\right) 
	\]	
where $\mu, \nu \in \mathbb{R}$ are such that $\mu -\nu\neq 0$. 
\item[2C.] $h=2\mu$,  $\delta_2=\delta_5=2\varepsilon_2 \mu$, $\delta_1=\delta_3$, $\delta_4=\delta_6=0$ and 
\[
M=\left( 
\begin{array}{cc}-\varepsilon_1\mu &-\varepsilon_1\mu\\
		\varepsilon_3\mu &\varepsilon_{3}\mu
		\end{array}
	\right) 
	\]		
where $\mu\in\mathbb{R}\setminus \{0\}$. 
\item[2D.] $h=2\mu$,  $\delta_2=\delta_5= -2\varepsilon_2\mu$, $\delta_1=-\delta_3$, $\delta_4=\delta_6=0$ and 
\[
M=\left( 
\begin{array}{cc}\varepsilon_1\mu&-\varepsilon_1\mu\\
		\varepsilon_3\mu&-\varepsilon_{3}\mu
	\end{array}
	\right) 
	\]	
where $\mu\in\mathbb{R}\setminus \{ 0\}$. 
\item[3A.] $h=\nu-\mu$, $\delta_2=\delta_5= \varepsilon_2(\nu -\mu)$, $\delta_1=\delta_3=0$, $\delta_4$ and $\delta_6$ are related by $\mu\delta_4-\nu\delta_6=0$ and 
\[
M=\left( 
\begin{array}{cc}\varepsilon_1\mu& \varepsilon_1\nu \\
		\varepsilon_3\mu&\varepsilon_{3}\nu 
	\end{array}
	\right) 
	\]	
where $\mu, \nu \in \mathbb{R}$ are such that $\mu -\nu\neq 0$. 
\item[3B.] $h=\nu-\mu$, $\delta_2=\delta_5= \varepsilon_2(\mu -\nu)$, $\delta_1=\delta_3=0$, $\delta_4$ and $\delta_6$ are related by $\mu \delta_4+\nu\delta_6=0$ and 
\[
M=\left( 
\begin{array}{cc}
-\varepsilon_1\mu&\varepsilon_1\nu\\
		 \varepsilon_3\mu&-\varepsilon_{3}\nu
	\end{array}
	\right) 
	\]
where $\mu, \nu \in \mathbb{R}$ are such that $\mu -\nu\neq 0$. 
\item[3C.] $h=-2\mu$, $\delta_2=\delta_5= 2\varepsilon_2\mu$, $\delta_1=\delta_3=0$, $\delta_4=\delta_6$ and 
\[ M=\left( 
\begin{array}{cc}-\varepsilon_1\mu &-\varepsilon_1\mu\\
		\varepsilon_3\mu &\varepsilon_{3}\mu	
	\end{array}
	\right) 
	\]
where $\mu\in \mathbb{R}\setminus \{ 0\}$. 
\item[3D.] $h=-2\mu$, $\delta_2=\delta_5= -2\varepsilon_2\mu$, $\delta_1=\delta_3=0$, $\delta_4=-\delta_6$ and 
\[ M=\left( 
\begin{array}{cc}
	\varepsilon_1\mu &-\varepsilon_1\mu \\
		\varepsilon_3\mu &-\varepsilon_{3}\mu 
\end{array}
	\right) 
	\]	
where $\mu\in \mathbb{R}\setminus \{ 0\}$. 
\end{enumerate} 
\end{prop}
\begin{proof}
	As in the proof of Proposition~\ref{non-unimodular,non-degenerate:prop} there exists a $ g $-orthonormal basis $ (v_a)_a $ of $ \mathfrak g $ such that $ v_1,v_3\in\mathfrak{u} $ and $ \lambda,\mu,\nu,\rho\in\R $ such that 
	\begin{eqnarray*}
		[v_3,v_1]&=&0\\
		\left[v_2,v_1\right]&=&\varepsilon_1\lambda v_1+\varepsilon_3\mu v_3\\
		\left[v_2,v_3\right]&=&\varepsilon_1\nu v_1+\varepsilon_{3}\rho v_3
	\end{eqnarray*}
	with $ 0\neq\tr \mathrm{ad}_{v_2}=\varepsilon_1\lambda +\varepsilon_{3}\rho $. Using the Dorfman coefficients, that were computed in the proof of Proposition~\ref{non-unimodular,non-degenerate:prop}, we obtain the components of the Ricci tensor. 
	
	In the case $ \varepsilon_1=\varepsilon_3 $, we have 
	\begin{eqnarray*}
		R^\delta_{52}&=&R_{52}+\varepsilon_1\mathcal{B}_{521}\delta_1+\varepsilon_3\mathcal{B}_{523}\delta_3\\
		&=&-\lambda^2-\frac14\left(h-\mu-\nu\right)^2-\frac14\left(h+\mu+\nu\right)^2-\rho^2,
	\end{eqnarray*}
	which is always non-zero due to $ 0\neq\varepsilon_1\lambda +\varepsilon_{3}\rho $. Hence, we can assume that the basis is chosen such that $ \varepsilon_{1}=\varepsilon_{2}=-\varepsilon_{3} $.
	
	In this case the components of the Ricci tensor are 
	\begin{eqnarray*}
		R_{41}^\delta&=&R_{41}+\varepsilon_2\mathcal{B}_{412}\delta_2+\varepsilon_3\mathcal{B}_{413}\delta_3\\
		&=&\frac12\left(h^2+\mu^2-\nu^2\right)+\varepsilon_2\lambda\delta_2\\
		R_{42}^\delta&=&R_{42}+\varepsilon_1\mathcal{B}_{421}\delta_1+\varepsilon_3\mathcal{B}_{423}\delta_3\\
		&=&-\varepsilon_1\lambda\delta_1+\varepsilon_3\frac12\left(h-\mu-\nu\right)\delta_3\\
		R_{43}^\delta&=&R_{43}+\varepsilon_1\mathcal{B}_{431}\delta_1+\varepsilon_2\mathcal{B}_{432}\delta_2\\
		&=&-\frac12\lambda\left(h-\mu+\nu\right)+\frac12\rho\left(h+\mu-\nu\right)-\varepsilon_2\frac12\left(h-\mu-\nu\right)\delta_2\\
		R_{51}^\delta&=&R_{51}+\varepsilon_2\mathcal{B}_{512}\delta_2+\varepsilon_3\mathcal{B}_{513}\delta_3\\
		&=&-\varepsilon_3\frac12\left(h-\mu+\nu\right)\delta_3\\
		R_{52}^\delta&=&R_{52}+\varepsilon_1\mathcal{B}_{521}\delta_1+\varepsilon_3\mathcal{B}_{523}\delta_3\\
		&=&-\lambda^2+\frac12h^2+\frac12\left(\mu+\nu\right)^2-\rho^2\\
		R_{53}^\delta&=&R_{53}+\varepsilon_1\mathcal{B}_{531}\delta_1+\varepsilon_2\mathcal{B}_{532}\delta_2\\
		&=&\varepsilon_1\frac12\left(h-\mu+\nu\right)\delta_1\\
		R_{61}^\delta&=&R_{61}+\varepsilon_2\mathcal{B}_{612}\delta_2+\varepsilon_3\mathcal{B}_{613}\delta_3\\
		&=&\frac12\lambda\left(h+\mu-\nu\right)-\frac12\rho\left(h-\mu+\nu\right)+\varepsilon_2\frac12\left(h+\mu+\nu\right)\delta_2\\
		R_{62}^\delta&=&R_{62}+\varepsilon_1\mathcal{B}_{621}\delta_1+\varepsilon_3\mathcal{B}_{623}\delta_3\\
		&=&-\varepsilon_1\frac12\left(h+\mu+\nu\right)\delta_1-\varepsilon_3\rho\delta_3\\
		R_{63}^\delta&=&R_{63}+\varepsilon_1\mathcal{B}_{631}\delta_1+\varepsilon_2\mathcal{B}_{632}\delta_2\\
		&=&-\frac12\left(h^2-\mu^2+\nu^2\right)+\varepsilon_2\rho\delta_2\\
		R_{14}^\delta&=&R_{41}-\varepsilon_2\mathcal{B}_{145}\delta_5-\varepsilon_3\mathcal{B}_{146}\delta_6\\
		&=&\frac12\left(h^2+\mu^2-\nu^2\right)+\varepsilon_2\lambda\delta_5\\
		R_{24}^\delta&=&R_{42}-\varepsilon_2\mathcal{B}_{245}\delta_5-\varepsilon_3\mathcal{B}_{246}\delta_6\\
		&=&\varepsilon_3\frac12\left(h+\mu-\nu\right)\delta_6\\
		R_{34}^\delta&=&R_{43}-\varepsilon_2\mathcal{B}_{345}\delta_5-\varepsilon_3\mathcal{B}_{346}\delta_6\\
		&=&-\frac12\lambda\left(h-\mu+\nu\right)+\frac12\rho\left(h+\mu-\nu\right)-\varepsilon_2\frac12\left(h-\mu-\nu\right)\delta_5\\
		R_{15}^\delta&=&R_{51}-\varepsilon_1\mathcal{B}_{154}\delta_4-\varepsilon_3\mathcal{B}_{156}\delta_6\\
		&=&-\varepsilon_1\lambda\delta_4-\varepsilon_3\frac12\left(h+\mu+\nu\right)\delta_6\\
		R_{25}^\delta&=&R_{52}-\varepsilon_1\mathcal{B}_{254}\delta_4-\varepsilon_3\mathcal{B}_{256}\delta_6\\
		&=&-\lambda^2+\frac12h^2+\frac12\left(\mu+\nu\right)^2-\rho^2\\
		R_{35}^\delta&=&R_{53}-\varepsilon_1\mathcal{B}_{354}\delta_4-\varepsilon_3\mathcal{B}_{356}\delta_6\\
		&=&\varepsilon_1\frac12\left(h-\mu-\nu\right)\delta_4-\varepsilon_3\rho\delta_6\\
		R_{16}^\delta&=&R_{61}-\varepsilon_1\mathcal{B}_{164}\delta_4-\varepsilon_2\mathcal{B}_{165}\delta_5\\
		&=&\frac12\lambda\left(h+\mu-\nu\right)-\frac12\rho\left(h-\mu+\nu\right)+\varepsilon_2\frac12\left(h+\mu+\nu\right)\delta_5\\
		R_{26}^\delta&=&R_{62}-\varepsilon_1\mathcal{B}_{264}\delta_4-\varepsilon_2\mathcal{B}_{265}\delta_5\\
		&=&-\varepsilon_1\frac12\left(h+\mu-\nu\right)\delta_4\\
		R_{36}^\delta&=&R_{63}-\varepsilon_1\mathcal{B}_{364}\delta_4-\varepsilon_2\mathcal{B}_{365}\delta_5\\
		&=&-\frac12\left(h^2-\mu^2+\nu^2\right)+\varepsilon_2\rho\delta_5.
	\end{eqnarray*}
	Note first that $ R_{41}^\delta=R_{14}^\delta $ and $ R_{63}^\delta=R_{36}^\delta $ yield $ \varepsilon_2\lambda\left(\delta_2-\delta_5\right)=0 $ and $ \varepsilon_2\rho\left(\delta_2-\delta_5\right)=0 $. Therefore $ \delta_2=\delta_5 $, because $ 0\neq\varepsilon_1\lambda +\varepsilon_{3}\rho $. If $ \delta_2=\delta_5=0 $, we see that $ R_{ii'}^\delta=R_{i'i}^\delta=R_{ii'} $ for all $ i\in\{4,5,6\} $.  So we can deduce the same way as in the proof of Proposition~\ref{non-unimodular,non-degenerate:prop}, that $ h=0 $ and there is a positive constant $ \theta>0 $ such that $ \varepsilon_1\lambda=-\varepsilon_3\mu=\varepsilon_1\nu=\varepsilon_3\rho=\theta $. Then we see that $ 0=R_{42}^\delta=-\theta\delta_1+\theta\delta_3 $ and $ R_{15}^\delta=-\theta\delta_4+\theta\delta_6 $ imply $ \delta_1=\delta_3 $ and $ \delta_4=\delta_6 $. Similarly, $R_{62}^\delta=-\theta \delta_1 -\theta \delta_3=0$ and $R_{35}^\delta=-\theta \delta_4-\theta \delta_6=0$ imply 
	$ \delta_1=-\delta_3 $ and $ \delta_4=-\delta_6 $. This proves that $\delta=0$ if $\delta_2=\delta_5=0$. 
	Note that the endomorphism  $ M\in\End(\mathfrak{u}) $, defined as the restriction of $ \mathrm{ad}_{v_2} $ to $ \mathfrak{u} $, has the two complex eigenvalues $ \theta\pm i\theta $. Assume now $ \delta_2=\delta_5\neq0 $. Then 
	\begin{equation}\label{41-63:eq}
		 0=R_{41}^\delta-R_{63}^\delta=h^2+\varepsilon_2\delta_2\left(\lambda-\rho\right).
	\end{equation}
	Using $ \lambda-\rho=\varepsilon_1\left(\varepsilon_1\lambda+\varepsilon_3\rho\right)\neq0 $, we see $ h\neq0 $. From $ 0=R_{61}^\delta-R_{43}^\delta=h\left(\lambda-\rho+\varepsilon_2\delta_2\right) $ we see $ \varepsilon_2\delta_2=-\lambda+\rho=-\varepsilon_2\tr M $. Then equation (\ref{41-63:eq}) is equivalent to $ h^2=\left(\lambda-\rho\right)^2=\left(\tr M\right)^2 $. 
	Since
	\begin{eqnarray*}
		R_{52}^\delta&=&-\lambda^2+\frac12h^2+\frac12\left(\mu+\nu\right)^2-\rho^2\\
		&=&-\lambda^2-\rho^2+\frac12\left(\lambda-\rho\right)^2+\frac12\left(\mu+\nu\right)^2\\
		&=&-\frac12\left(\lambda+\rho\right)^2+\frac12\left(\mu+\nu\right)^2,
	\end{eqnarray*}
	the equation $ R_{52}^\delta=0 $ is equivalent to  
	\begin{equation}\label{52:eq}
		\left(\lambda+\rho\right)^2=\left(\mu+\nu\right)^2.
	\end{equation}
	Note now that the discriminant $ \Delta $ of the characteristic polynomial $ X^2-\varepsilon_1\left(\lambda-\rho\right)X-\lambda\rho+\mu\nu $ of $ M $ is 
	\begin{eqnarray*}
		\Delta&=&\left(\lambda-\rho\right)^2+4\lambda\rho-4\mu\nu\\
		&=&\left(\lambda+\rho\right)^2-4\mu\nu\\
		&=&\left(\mu+\nu\right)^2-4\mu\nu\\
		&=&\left(\mu-\nu\right)^2,
	\end{eqnarray*}
	which is never negative. Therefore $ M $ has real eigenvalues. Hence $ M $ is either diagonalizable with two distinct eigenvalues, or it has a double eigenvalue.  The latter happens precisely if the discriminant is zero, that is if $ \mu=\nu $. But then $ M $ is a symmetric endomorphism and hence takes one of the normal forms 
	\begin{eqnarray*}
		M_1(\theta,\eta )  &=& \left( \begin{array}{cc}
			\theta& 0\\
			0&\eta
		\end{array}\right),\quad 
		M_2(\theta,\eta )=\left( \begin{array}{cc}
			\theta& -\eta\\
			\eta&\theta
		\end{array}\right),\\
		M_3(\theta)&=&\left( \begin{array}{cc}\frac12+\theta &\frac12\\
			-\frac12&-\frac12 +\theta
		\end{array}\right),\quad 
		M_4(\theta)= \left( \begin{array}{cc} -\frac12+\theta &-\frac12\\
			\frac12&\frac12 +\theta\\
		\end{array}\right)
	\end{eqnarray*}
	with respect to an orthonormal basis $ v_1,v_3 $ of $ \mathfrak{u} $, as in the proof of Proposition \ref{non-unimodular,non-degenerate:prop}. If it takes the normal form $ M_1(\theta,\eta) $, then $ \theta=\eta $, since $ M $ has a double eigenvalue. Hence that $ M $ is of the form   
	\begin{equation*}
		M_1(\theta,\theta)=\left( \begin{array}{cc}
			\theta& 0\\
			0 &\theta
		\end{array}\right). 
	\end{equation*}
	We may assume $ \theta $ is positive by replacing $ v_2 $ with $ -v_2 $. If it takes the normal form $ M_2(\theta,\eta) $, then we have $ R_{43}^\delta+R_{61}^\delta=-2\theta\eta $. But $ \theta\ne0 $, since $ \tr M \ne 0 $. Therefore $ \eta=0 $ and we get the same normal form as before. The normal forms $ M_3(\theta) $ and $ M_4(\theta) $ are excluded, because in both cases $ R_{43}^\delta+R_{61}^\delta=2\theta $, which cannot be zero because $ \tr M \ne 0 $. Note, that in the case that $ M $ takes the normal form $ M_1(\theta,\theta) $, we have $ \delta_1=\delta_3=\delta_4=\delta_6=0 $, due to $ R_{53}^\delta=R_{51}^\delta=R_{26}^\delta=R_{24}^\delta=0 $.
	
	It remains to consider the case $\mu-\nu\neq 0$. Recall that $\delta_2=\delta_5 = \varepsilon_2(-\lambda +\rho)\neq 0$ and $h=\pm (\lambda -\rho )\neq 0$. We distinguish the following cases. Note that the expressions $h-\mu+\nu$ and $h+\mu-\nu$ cannot both vanish simultaneously, since $h\neq 0$
	
	\noindent
	\textbf{Case 1: $h-\mu+\nu \neq 0$ and $h+\mu-\nu\neq 0$.} 
	In this case we conclude from $R_{51}^\delta= R_{53}^\delta=R_{24}^\delta=R_{26}^\delta=0$ that 
$\delta_1=\delta_3=\delta_4=\delta_6=0$. Then from the remaining equations we obtain
\[ \mu^2 -\lambda^2 =-(\rho^2-\nu^2),\; \rho\mu-\lambda\nu=0,\; (\lambda+\rho)^2 = (\mu +\nu)^2.\]
The first two equations are satisfied if and only if $(\mu,\lambda)$ and $(\rho, \nu)$ are non-zero orthogonal
vectors of equal length in the Minkowski plane. This implies that  
$(\mu ,\lambda )= - (\nu,\rho)$, since $\mu-\nu\neq 0$. This yields the first family of solutions for which $M$ has two distinct real eigenvalues. 

	\textbf{Case 2: $h-\mu+\nu = 0$ and $h+\mu-\nu\neq 0$.} In this case the equations $R_{51}^\delta= R_{53}^\delta=R_{24}^\delta=R_{26}^\delta=0$ 
	reduce to $\delta_4=\delta_6=0$. Recall that $R_{52}^\delta=0$ yields the equation (\ref{52:eq}) and hence $\mu +\nu = \sigma_+ (\lambda +\rho )$ for some 
	$\sigma_+\in \{ \pm 1\}$. Since now $h=\mu -\nu$, the equation 
	$h^2 = (\lambda -\rho)^2$ yields $\mu -\nu = \sigma_-(\lambda -\rho)$, $\sigma_-\in \{\pm 1\}$. 
	We consider four subcases depending on the signs $\sigma_\pm$. In each case we first solve the two equations 
	$\mu \pm \nu = \sigma_\pm (\lambda \pm \rho)$. 
	
	\textbf{Case 2A: $\sigma_+=\sigma_-=1$.} In this case $\mu=\lambda$ and $\nu = \rho$. It turns out that the remaining components of the 
	generalized Ricci curvature vanish if $R_{42}^\delta$ does. Its vanishing is equivalent to $\mu \delta_1-\nu \delta_3=0$.  
	
	\textbf{Case 2B: $\sigma_+=\sigma_-=-1$.} In this case $\mu=-\lambda$ and $\nu = -\rho$. The remaining components of the 
	generalized Ricci curvature vanish if and only if $\mu \delta_1+\nu \delta_3=0$.  
	
	\textbf{Case 2C: $\sigma_+=1$ and $\sigma_-=-1$.} Then $\mu=\rho$, $\nu = \lambda$ and from $R_{41}^\delta=0$ we get $\mu^2-\nu^2=0$.
	We conclude that $\mu = -\nu$, since $h=\mu-\nu \neq 0$. The equation $R_{42}^\delta=0$ reduces to $\delta_1=\delta_3$ and the 
	remaining components of the generalized Ricci tensor then vanish.
	
	\textbf{Case 2D: $\sigma_+=-1$ and $\sigma_-=1$.} Then $\mu=-\rho$, $\nu = -\lambda$ and from $R_{41}^\delta=0$ we get again $\mu = -\nu$.
	In this case, the equation $R_{42}^\delta=0$ reduces to $\delta_1=-\delta_3$ and the 
	remaining components of the generalized Ricci tensor vanish.
	
	\textbf{Case 3: $h-\mu+\nu \neq 0$ and $h+\mu-\nu= 0$.} In this case the equations $R_{51}^\delta= R_{53}^\delta=R_{24}^\delta=R_{26}^\delta=0$ 
	reduce to $\delta_1=\delta_3=0$. From (\ref{52:eq}) and $h=\nu-\mu$, we still obtain $\mu \pm \nu = \sigma_\pm (\lambda \pm \rho)$ with 
	$\sigma_+, \sigma_-\in \{\pm 1\}$.	We consider
	again four subcases depending on the values of $\sigma_\pm$.
	
	\textbf{Case 3A: $\sigma_+=\sigma_-=1$.} As above, $\mu=\lambda$ and $\nu = \rho$. The equation $R_{15}^\delta=0$ yields 
	$\mu \delta_4-\nu\delta_6=0$ and the remaining components then vanish.
	
	\textbf{Case 3B: $\sigma_+=\sigma_-=-1$.} Here $\mu=-\lambda$, $\nu = -\rho$ and the equation $R_{15}^\delta=0$ yields 
	$\mu \delta_4+\nu\delta_6=0$. The remaining components then vanish.
	
	\textbf{Case 3C: $\sigma_+=1$ and $\sigma_-=-1$.} Here $\mu=\rho$, $\nu = \lambda$ and $R_{41}^\delta=0$ implies $\mu = -\nu$. Finally, $R_{15}^\delta=0$ yields $\delta_4=\delta_6$ and the remaining components vanish. 
	
	\textbf{Case 3D: $\sigma_+=-1$ and $\sigma_-=1$.} Here $\mu = -\rho$, $\nu = -\lambda$ and $R_{41}^\delta=0$ implies $\mu = -\nu$. 
	Finally, the equation $R_{15}^\delta=0$ yields $\delta_4=-\delta_6$ and all other components vanish. 	
\end{proof}
\begin{prop}
\label{divnotzero_u_deg:prop}
%
	
	Let $ G $ be an three-dimensional non-unimodular Lie group. As any three-dimensional non-unimodular Lie algebra, its Lie algebra $ \mathfrak{g} $ is isomorphic to a semidirect product of $ \R $ and $ \R^2 $, with $ \R $ acting on $ \R^2 $ by a 2 by 2 matrix $ M $ (of non-zero trace). Then there exists a generalized Einstein structure $(H,\mathcal G_g,\delta)$ on $ G $, such that the restriction of $ g $ to the unimodular kernel $ \mathfrak{u} $ is degenerate, if and only if $ H=0 $ and $ M $ has real eigenvalues. (All such structures have $\delta\neq 0$ and are described at the end of the proof.)
\end{prop}
\begin{proof}
	Note first that the metric $ g $ necessarily has to be indefinite. As in the proof of Proposition~\ref{non-unimodular,degenerate:prop}, there exists an orthonormal basis $ (v_a)_a $ of $ (\mathfrak{g},g) $ such that $g(v_1,v_1)=g(v_2,v_2)$ and $ \lambda,\mu,\nu,\rho\in\R $ such that 
	\begin{eqnarray*}
		\left[v_1,v_2\right]&=&\varepsilon_{1}\lambda v_1+\varepsilon_2\mu v_2-\varepsilon_3\mu v_3\\
		\left[v_2,v_3\right]&=&\varepsilon_{1}\nu v_1+\varepsilon_2\rho v_2-\varepsilon_3\rho v_3\\
		\left[v_3,v_1\right]&=&\varepsilon_{1}\lambda v_1+\varepsilon_2\mu v_2-\varepsilon_3\mu v_3
	\end{eqnarray*}
	with $ \lambda+\rho\neq0 $. Using the Dorfman coefficients, that were computed in the proof of Proposition~\ref{non-unimodular,degenerate:prop}, we obtain the components of the Ricci tensor
		\begin{eqnarray*}
		R_{41}^\delta&=&R_{41}+\varepsilon_2\mathcal{B}_{412}\delta_2+\varepsilon_3\mathcal{B}_{413}\delta_3\\
		&=&\frac12h^2-\frac12\nu^2-\varepsilon_2\lambda\delta_2+\varepsilon_3\lambda\delta_3\\
		R_{42}^\delta&=&R_{42}+\varepsilon_1\mathcal{B}_{421}\delta_1+\varepsilon_3\mathcal{B}_{423}\delta_3\\
		&=&\frac12\rho\left(h-\nu\right)-\frac12\lambda\left(h+\nu\right)+\varepsilon_1\lambda\delta_1+\varepsilon_3\frac12\left(h-\nu\right)\delta_3\\
		R_{43}^\delta&=&R_{43}+\varepsilon_1\mathcal{B}_{431}\delta_1+\varepsilon_2\mathcal{B}_{432}\delta_2\\
		&=&\frac12\lambda\left(h+\nu\right)-\frac12\rho\left(h-\nu\right)-\varepsilon_1\lambda\delta_1-\varepsilon_2\frac12\left(h-\nu\right)\delta_2\\
		R_{51}^\delta&=&R_{51}+\varepsilon_2\mathcal{B}_{512}\delta_2+\varepsilon_3\mathcal{B}_{513}\delta_3\\
		&=&\frac12\lambda\left(h-\nu\right)-\frac12\rho\left(h+\nu\right)-\varepsilon_2\mu\delta_2-\varepsilon_3\frac12\left(h-2\mu+\nu\right)\delta_3\\
		R_{52}^\delta&=&R_{52}+\varepsilon_1\mathcal{B}_{521}\delta_1+\varepsilon_3\mathcal{B}_{523}\delta_3\\
		&=&-\lambda^2+\frac12h^2+\frac12\nu^2+\mu\nu-\rho^2+\varepsilon_1\mu\delta_1-\varepsilon_3\rho\delta_3\\
		R_{53}^\delta&=&R_{53}+\varepsilon_1\mathcal{B}_{531}\delta_1+\varepsilon_2\mathcal{B}_{532}\delta_2\\
		&=&\lambda^2-\mu\nu+\rho^2+\varepsilon_1\frac12\left(h-2\mu+\nu\right)\delta_1+\varepsilon_2\rho\delta_2\\
		R_{61}^\delta&=&R_{61}+\varepsilon_2\mathcal{B}_{612}\delta_2+\varepsilon_3\mathcal{B}_{613}\delta_3\\
		&=&-\frac12\lambda\left(h-\nu\right)+\frac12\rho\left(h+\nu\right)+\varepsilon_2\frac12\left(h+2\mu+\nu\right)\delta_2-\varepsilon_3\mu\delta_3\\
		R_{62}^\delta&=&R_{62}+\varepsilon_1\mathcal{B}_{621}\delta_1+\varepsilon_3\mathcal{B}_{623}\delta_3\\
		&=&\lambda^2-\mu\nu+\rho^2-\varepsilon_1\frac12\left(h+2\mu+\nu\right)\delta_1+\varepsilon_3\rho\delta_3\\
		R_{63}^\delta&=&R_{63}+\varepsilon_1\mathcal{B}_{631}\delta_1+\varepsilon_2\mathcal{B}_{632}\delta_2\\
		&=&-\lambda^2-\frac12h^2-\frac12\nu^2+\mu\nu-\rho^2+\varepsilon_1\mu\delta_1-\varepsilon_2\rho\delta_2\\
		R_{14}^\delta&=&R_{41}-\varepsilon_2\mathcal{B}_{145}\delta_5-\varepsilon_3\mathcal{B}_{146}\delta_6\\
		&=&\frac12h^2-\frac12\nu^2-\varepsilon_2\lambda\delta_5+\varepsilon_3\lambda\delta_6\\
		R_{24}^\delta&=&R_{42}-\varepsilon_2\mathcal{B}_{245}\delta_5-\varepsilon_3\mathcal{B}_{246}\delta_6\\
		&=&\frac12\rho\left(h-\nu\right)-\frac12\lambda\left(h+\nu\right)-\varepsilon_2\mu\delta_5+\varepsilon_3\frac12\left(h+2\mu-\nu\right)\delta_6\\
		R_{34}^\delta&=&R_{43}-\varepsilon_2\mathcal{B}_{345}\delta_5-\varepsilon_3\mathcal{B}_{346}\delta_6\\
		&=&\frac12\lambda\left(h+\nu\right)-\frac12\rho\left(h-\nu\right)-\varepsilon_2\frac12\left(h-2\mu-\nu\right)\delta_5-\varepsilon_3\mu\delta_6\\
		R_{15}^\delta&=&R_{51}-\varepsilon_1\mathcal{B}_{154}\delta_4-\varepsilon_3\mathcal{B}_{156}\delta_6\\
		&=&\frac12\lambda\left(h-\nu\right)-\frac12\rho\left(h+\nu\right)+\varepsilon_1\lambda\delta_4-\varepsilon_3\frac12\left(h+\nu\right)\delta_6\\
		R_{25}^\delta&=&R_{52}-\varepsilon_1\mathcal{B}_{254}\delta_4-\varepsilon_3\mathcal{B}_{256}\delta_6\\
		&=&-\lambda^2+\frac12h^2+\frac12\nu^2+\mu\nu-\rho^2+\varepsilon_1\mu\delta_4-\varepsilon_3\rho\delta_6\\
		R_{35}^\delta&=&R_{53}-\varepsilon_1\mathcal{B}_{354}\delta_4-\varepsilon_3\mathcal{B}_{356}\delta_6\\
		&=&\lambda^2-\mu\nu+\rho^2+\varepsilon_1\frac12\left(h-2\mu-\nu\right)\delta_4+\varepsilon_3\rho\delta_6\\
		R_{16}^\delta&=&R_{61}-\varepsilon_1\mathcal{B}_{164}\delta_4-\varepsilon_2\mathcal{B}_{165}\delta_5\\
		&=&-\frac12\lambda\left(h-\nu\right)+\frac12\rho\left(h+\nu\right)-\varepsilon_1\lambda\delta_4+\varepsilon_2\frac12\left(h+\nu\right)\delta_5\\
		R_{26}^\delta&=&R_{62}-\varepsilon_1\mathcal{B}_{264}\delta_4-\varepsilon_2\mathcal{B}_{265}\delta_5\\
		&=&\lambda^2-\mu\nu+\rho^2-\varepsilon_1\frac12\left(h+2\mu-\nu\right)\delta_4+\varepsilon_2\rho\delta_5\\
		R_{36}^\delta&=&R_{63}-\varepsilon_1\mathcal{B}_{364}\delta_4-\varepsilon_2\mathcal{B}_{365}\delta_5\\
		&=&-\lambda^2-\frac12h^2-\frac12\nu^2+\mu\nu-\rho^2+\varepsilon_1\mu\delta_4-\varepsilon_2\rho\delta_5.
	\end{eqnarray*}
	Assume now that $(H,\mathcal G_g,\delta)$ is generalized Einstein. We first want to show that $ h=\nu=0 $ and $ \varepsilon_2\delta_2=\varepsilon_3\delta_3 $. For this, consider the system of equations $ 0=R_{42}^\delta+R_{43}^\delta=-\frac12\left(h-\nu\right)\left(\varepsilon_2\delta_2-\varepsilon_3\delta_3\right) $ and $ 0=R_{51}^\delta+R_{61}^\delta=\frac12\left(h+\nu\right)\left(\varepsilon_2\delta_2-\varepsilon_3\delta_3\right) $. This implies that either $ h=\nu=0 $ or $ \varepsilon_2\delta_2=\varepsilon_3\delta_3 $. If $ h=\nu=0 $, then $ 0=R_{41}^\delta=-\lambda\left(\varepsilon_2\delta_2-\varepsilon_3\delta_3\right) $ and $ 0=R_{63}^\delta-R_{52}^\delta=-\rho\left(\varepsilon_2\delta_2-\varepsilon_3\delta_3\right) $, which can only be the case if  $ \varepsilon_2\delta_2=\varepsilon_3\delta_3 $, since $ \lambda+\rho\ne0 $. If we otherwise assume, that $ \varepsilon_2\delta_2=\varepsilon_3\delta_3 $, then $ 0=R_{63}^\delta-R_{52}^\delta=-h^2-\nu^2 $ and therefore $ h=\nu=0 $. Similarly one can also show that $ \varepsilon_2\delta_5=\varepsilon_3\delta_6 $. Hence, the Einstein condition is equivalent to the set of equations 
	\begin{equation}\label{divnotzero_u_deg:eq}
		\begin{split}
			h=\nu&=0\\
			\lambda\varepsilon_1\delta_1=\lambda\varepsilon_1\delta_4&=0\\
			\varepsilon_2\delta_2&=\varepsilon_3\delta_3\\
			\varepsilon_2\delta_5&=\varepsilon_3\delta_6\\
			\lambda^2+\rho^2-\varepsilon_1\mu\delta_1+\varepsilon_2\rho\delta_2&=0\\
			\lambda^2+\rho^2-\varepsilon_1\mu\delta_4+\varepsilon_2\rho\delta_5&=0.
		\end{split}	
	\end{equation}
	Now, as in the proof of Proposition~\ref{non-unimodular,degenerate:prop}, there exists a basis $ (w_a) $ of $ \mathfrak{g} $, such that $ w_1,w_2\in\mathfrak{u} $,
	\begin{eqnarray*}
		\left[w_1,w_2\right]&=&0\\
		\left[w_3,w_1\right]&=&-\varepsilon_1\lambda w_1-2\varepsilon_2\mu w_2\\
		\left[w_3,w_2\right]&=&-\frac12\varepsilon_1\nu w_1-\varepsilon_2\rho w_2 
	\end{eqnarray*}
	and $ g(w_a,w_b) $ satisfies Equation (\ref{indef basis:eq}). Hence $ \mathfrak{g} $ is a semidirect product of $ \R\cong \mathrm{span}\{ w_3\} $ and $ \R^2\cong \mathrm{span}\{ w_1,w_2\} $, the former acting on the latter with the matrix 
	\begin{equation*}
		 M=\left( \begin{array}{cc}-\varepsilon_1\lambda& -\frac12\varepsilon_1\nu\\ -2\varepsilon_2\mu & -\varepsilon_2\rho \end{array}\right).
	\end{equation*}
	Since in the Einstein case $ \nu=0 $, its eigenvalues $ -\varepsilon_1\lambda $ and $ -\varepsilon_2\rho $ are real. Furthermore, for any such matrix, with $ \nu=0 $, one can find $ \delta\in E^* $, such that $(H=0,\mathcal G_g,\delta)$ is generalized Einstein. In fact, 
	if $\lambda\neq 0$, then $h=\nu=\delta_1=\delta_4=0$, $\rho\neq 0$ and the solution is uniquely determined by the free parameters  
	$\lambda \neq 0, \rho\neq 0$ and $\mu$ as $\delta_2=\delta_5=-\delta_3=-\delta_6= -\varepsilon_2 (\lambda^2 +\rho^2)/\rho\neq 0$. 
	If $\lambda=0$, then $\rho \neq 0$ (as $\lambda +\rho \neq 0$), $h=\nu = \lambda=0$ and  the solution is uniquely determined by the free parameters 
	$\mu, \delta_1$ and $\delta_4$ as $\delta_2= -\delta_3=-\varepsilon_2(\rho^2-\varepsilon_1\mu \delta_1)/\rho$ and $\delta_5=-\delta_6=-\varepsilon_2(\rho^2 -\varepsilon_1\mu \delta_4)/\rho$. Note that 
	all the solutions have non-zero divergence and that $M$ has rank 1 if and only if $\lambda =0$. 
\end{proof}

\subsection{Riemannian divergence}
In this section we want to determine those solutions $(G,H,\mathcal{G},\delta)$ to the generalized Einstein equation for which the divergence $ \delta $ coincides with the Riemannian divergence $\delta^\mathcal{G} =-\tau \circ \pi \in E^*$ (see Proposition \ref{Riemdiv:prop}). If the Lie group is unimodular, the trace-form $ \tau $, and therefore the Riemannian divergence, is zero. This was covered in Theorem \ref{Einstein:thm}. It remains to specify the results of Proposition \ref{divnotzero_u_nondeg:prop} and Proposition \ref{divnotzero_u_deg:prop} to the case $ \delta=\delta^\mathcal{G} $. 

In the case that $ g $ is non-degenerate on the unimodular kernel $ \mathfrak{u} $, $ \delta=\delta^\mathcal{G} $ holds if and only if the components of $ \delta $ in the basis $ (v_a) $ of $ \mathfrak{g} $ from Proposition \ref{divnotzero_u_nondeg:prop} are
\begin{equation*}
	\begin{split}
		\delta_1=\delta_4=-\tr \mathrm{ad}_{v_1}&=0\\
		\delta_2=\delta_5=-\tr \mathrm{ad}_{v_2}&\ne0\\
		\delta_3=\delta_6=-\tr \mathrm{ad}_{v_3}&=0.
	\end{split}
\end{equation*}
Therefore the relevant solutions are those for which $ M=\mathrm{ad}_{v_2}|_{\mathfrak{u}} $ is diagonalizable and 
$\delta_1=\delta_3=\delta_4=\delta_6=0$, in virtue of Proposition~\ref{divnotzero_u_nondeg:prop}.

In the case that $ g $ is degenerate on the unimodular kernel $ \mathfrak{u} $, we compute the components of $ \delta $ in the basis $ (v_a) $ of $ \mathfrak{g} $ from Proposition~\ref{divnotzero_u_deg:prop} as
\begin{equation*}
\begin{split}
\delta_1=\delta_4=-\tr \mathrm{ad}_{v_1}&=0\\
\delta_2=\delta_5=-\tr \mathrm{ad}_{v_2}&=\varepsilon_1\left(\lambda-\rho\right)\\
\delta_3=\delta_6=-\tr \mathrm{ad}_{v_3} 
&=-\varepsilon_1\left(\lambda-\rho\right).
\end{split}
\end{equation*}
From the system of equations (\ref{divnotzero_u_deg:eq}), which is equivalent to the Einstein condition, we see now that $ \lambda^2+\lambda\rho=0 $. Hence $ \lambda=0 $, since $ \lambda+\rho\ne0 $. Finally we conclude that $ \mathfrak{g}\cong\R\ltimes_A \R^2 $, where $ A $ has one eigenvalue equal to zero and one non-zero eigenvalue.

\begin{prop}
\label{Riem_div_sol:prop}
	Let $(H,\mathcal G_g,\delta)$ be a generalized Einstein structure on an three\--di\-men\-sio\-nal non-unimodular Lie group $G$, with $ \delta=\delta^{\mathcal{G}_g} $ the Riemannian divergence of $ \mathcal{G}_g $. Let $ \mathfrak{u} $ be the unimodular kernel of the Lie algebra $ \mathfrak{g} $. If the pseudo-Riemannian metric $ g $ is non-degenerate on $ \mathfrak{u} $, then $\mathfrak{g} \cong \mathbb{R}\ltimes_A \mathbb{R}^2$ for a diagonalizable matrix $ A $, with $ \tr A \ne 0 $. If  $ g $ is degenerate on  $ \mathfrak{u} $, then $ \mathfrak{g}\cong\R\ltimes_A \R^2 $ for a matrix $ A $, whose kernel is one-dimensional. In both cases,  the matrix $A$ can be brought to the form 
\[A = \left( \begin{array}{cc}1&0\\ 0&s\end{array}\right), \quad s \in (-1,1],\] 
by an automorphism of $\mathfrak{g}$, where $s =0$ if 
$\mathfrak{u}$ is degenerate. (The precise tensors $H$, $g$ and $\delta$ are specified in Proposition~\ref{divnotzero_u_nondeg:prop}, Proposition~\ref{divnotzero_u_deg:prop} 
by specializing to the formulas for $\delta=\delta^{\mathcal{G}_g}$ given in this section.)
\end{prop}

\section{Tables}
\label{tables:sec}
In this section we want to summarize our results. For further details we refer to Section~\ref{class:sec}. Here $ L \mathsf{D} $ and $ L \neg\mathsf{D} $ mean that the endomorphism $ L $ defined in Equation~(\ref{defL:eq}) is diagonalizable and not diagonalizable, respectively. Furthermore we write def, indef, deg and non-deg instead of definite, indefinite, degenerate and non-degenerate. For the notations of the isomorphism classes of Lie algebras we refer to \cite[Ch.\ 7, Theorem 1.4]{GOV}. Following  \cite[Ch.\ 7, Theorem 1.4]{GOV}, we restrict the parameter $\lambda$ in $\mathfrak{r}_{3,\lambda}(\R)$  to 
$0<|\lambda| \le 1$. In addition, we exclude $\lambda=-1$, since  $\mathfrak{r}_{3,-1}(\R)\cong \mathfrak{e}(1,1)$.

\begin{table}[h]
	\centering
\begin{tabular}{ |c|c|c|c| } 
	\hline
	class of Lie algebras & H & g &  \\ 
	\hline
	$ \R^3 $ & $ =0 $ & flat & $ L \mathsf{D}$ \\ 
	$ \mathfrak{so}(3) $ & $ \ne0 $ & def & $ L \mathsf{D}$ \\ 
	$ \mathfrak{so}(2,1) $ & $ \ne0 $ & indef & $ L \mathsf{D}$ \\
	$ \mathfrak{e}(2) $ & $ =0 $ & flat, def on $ [\mathfrak{g},\mathfrak{g}] $ & $ L \mathsf{D}$ \\
	$ \mathfrak{e}(1,1) $ & $ =0 $ & flat, indef on $ [\mathfrak{g},\mathfrak{g}] $ & $ L \mathsf{D}$ \\
	$ \mathfrak{heis} $ & $ =0 $ & flat, indef & $ L \neg\mathsf{D}$ \\
	$ \mathfrak{r}'_{3,1}(\R)$ & $ =0 $ & indef & $g|_{\mathfrak{u}\times\mathfrak{u}} $ non-deg \\
	\hline
\end{tabular}
\caption{Divergence-free solutions to the generalized Einstein equation}
\end{table}

\begin{sidewaystable}[p]	
	\begin{tabular}{ |c|c|c|c|c| } 
		\hline
		class of Lie algebra & H & g & $ \delta $ &  \\ 
		\hline
		$ \R^3 $ & $ =0 $ &  & $ \delta\in E^* $ arbitrary & $ L \mathsf{D}$ \\ 
		$ \mathfrak{so}(3) $ & $ \ne0 $ & def & $ \delta|_{E_+}=0 $ or $ \delta|_{E_-}=0 $ & $ L \mathsf{D}$ \\ 
		$ \mathfrak{so}(2,1) $ & $ \ne0 $ & indef & $ \delta|_{E_+}=0 $ or $ \delta|_{E_-}=0 $ & $ L \mathsf{D}$ \\
		$ \mathfrak{e}(2) $ & $ =0 $ & def on $ [\mathfrak{g},\mathfrak{g}] $& $ \delta_{\sigma(1)}=\delta_{\sigma(2)}=\delta_{\sigma(1)+3}=\delta_{\sigma(2)+3}=0 $ &  $ L \mathsf{D}$ \\
		$ \mathfrak{e}(1,1) $ & $ =0 $ & indef on $ [\mathfrak{g},\mathfrak{g}] $& $ \delta_{\sigma(1)}=\delta_{\sigma(2)}=\delta_{\sigma(1)+3}=\delta_{\sigma(2)+3}=0 $ & $ L \mathsf{D}$ \\
		$ \mathfrak{heis} $ & $ =0 $ & indef & $ \delta_1=\delta_4=0,\delta_2=\delta_3,\delta_5=\delta_6 $ & $ L \neg\mathsf{D}$ \\
		$ \mathfrak{e}(1,1) $ & $ =0 $ & indef & $ \delta_1=\delta_4\ne0,\delta_2=\delta_3=\delta_5=\delta_6=0 $ & $ L \neg\mathsf{D}$\\
		$ \mathfrak{e}(1,1) $ & $ =0 $ & indef & $ \delta_{1}=-\delta_4=-\delta_3=\delta_6=-\sqrt2, \delta_2=\delta_5 $ & $ L \neg\mathsf{D}$ \\
		$ \mathfrak{r}_2(\R)\oplus\R $ & $ \ne0 $ & indef & $\delta_2=\delta_5=-\tr \mathrm{ad}_{v_2}\ne0 $, $\delta$ specified in 
		Prop.~\ref{divnotzero_u_nondeg:prop}  & $g|_{\mathfrak{u}\times\mathfrak{u}} $ non-deg \\
		$ \mathfrak{r}_{3,\lambda}(\R),\quad \lambda\ne1 $ & $ \ne0 $ & indef & $\delta_2=\delta_5=-\tr \mathrm{ad}_{v_2}\ne0$, $\delta$ specified in 
		Prop.~\ref{divnotzero_u_nondeg:prop} & $g|_{\mathfrak{u}\times\mathfrak{u}} $ non-deg \\
		$ \mathfrak{r}_{3,1}(\R) $ & $ \ne0 $ & indef & $\delta_A=0, A=1,3,4,6, \delta_2=\delta_5=-\tr \mathrm{ad}_{v_2}\ne0 $ & $g|_{\mathfrak{u}\times\mathfrak{u}} $ non-deg \\
		$ \mathfrak{r}_{2}(\R)\oplus\R $ & $ =0 $ & indef &$\delta_1$ and $\delta_4$ arbitrary determine $\delta\neq 0$, cf.\ Prop.~\ref{divnotzero_u_deg:prop}& $g|_{\mathfrak{u}\times\mathfrak{u}} $ deg\\
		$ \mathfrak{r}_{3}(\R)$ & $ =0 $ & indef &$\delta_1=\delta_4=0$, $\delta_2=\delta_5=-\delta_3=-\delta_6\neq 0$ & $g|_{\mathfrak{u}\times\mathfrak{u}} $ deg\\
		$ \mathfrak{r}_{3,\lambda}(\R), $ & $ =0 $ & indef &$\delta_1=\delta_4=0$, $\delta_2=\delta_5=-\delta_3=-\delta_6\neq 0$ & $g|_{\mathfrak{u}\times\mathfrak{u}} $ deg\\
		\hline
	\end{tabular}
	\centering
	\caption{Solutions to the generalized Einstein equation with arbitrary divergence}
\end{sidewaystable}

\end{document}